\documentclass[12pt]{article}
\usepackage[utf8]{inputenc}
\usepackage{amssymb}
\usepackage{amsmath}
\usepackage{amsthm}
\usepackage{enumitem}
\usepackage{ dsfont }
\usepackage{comment}
\usepackage[margin=0.8in]{geometry}
\usepackage{soul}
\usepackage{hyperref}
\hypersetup{
    colorlinks=true,
    linkcolor=blue,
    citecolor=red
}
\usepackage{orcidlink}
\usepackage{authblk} 
\usepackage{colonequals} 
\usepackage{color}

\title{\textbf{Diffusion approximation of critical controlled multi-type branching processes}}
\author{
  {\scshape Mátyás Barczy\textsuperscript{1}, Miguel González\textsuperscript{2}, Pedro Martín-Chávez\textsuperscript{2,$\ast$} and Inés del Puerto\textsuperscript{2}}\\
  \textsuperscript{1}{\small  HUN-REN--SZTE Analysis and Applications Research Group, Bolyai Institute, University of Szeged, 6720 Szeged, Aradi vértanúk tere 1., Hungary}\\
  \textsuperscript{2}{\small Department of Mathematics, University of Extremadura, 06006 Badajoz, Av.\ de Elvas s/n, Spain}\\
  \medskip
  {\small Emails: barczy@math.u-szeged.hu (M. Barczy), mvelasco@unex.es (M. González), pedromc@unex.es (P. Martín-Chávez) and idelpuerto@unex.es (I. del Puerto)}\\
  \textsuperscript{$\ast$}{\small Corresponding author}
}

\date{}

\newcommand{\prob}[1]{\operatorname{P} \left[ #1 \right]}
\newcommand{\ev}[1]{\operatorname{E}\! \left[ #1 \right]} 
\newcommand{\evcond}[2]{\operatorname{E} \left[ #1 \;\middle|\; #2 \right]} 
\newcommand{\var}[1]{\operatorname{Var}\! \left[ #1 \right]} 
\newcommand{\varcond}[2]{\operatorname{Var} \left[ #1 \;\middle|\; #2 \right]} 
\newcommand{\cov}[2]{\operatorname{Cov} \left[ #1 ,\; #2 \right]} 
\newcommand{\N}{\mathbb{N}}
\newcommand{\Z}{\mathbb{Z}}
\newcommand{\R}{\mathbb{R}}
\newcommand{\tr}[1]{\operatorname{tr} \left( #1 \right)} 
\newcommand{\dif}{\,\mathrm{d}}
\newcommand{\1}{\mathds{1}}

\newcommand{\OO}{\operatorname{O}}
\newcommand{\oo}{\operatorname{o}}

\newcommand{\bz}{{\boldsymbol{z}}}
\newcommand{\bZ}{{\boldsymbol{Z}}}

\newtheorem{theorem}{Theorem}[section]
\newtheorem{prop}[theorem]{Proposition}
\newtheorem{corollary}[theorem]{Corollary}
\newtheorem{lemma}[theorem]{Lemma}

\theoremstyle{definition}
\newtheorem{hyp}{Hypothesis}
\newtheorem{remark}[theorem]{Remark}
\newtheorem{example}[theorem]{Example}

\newcommand{\proofend}{\hfill\mbox{$\Box$}}
\newcommand{\proofendb}{\hfill\mbox{$\blacksquare$}}

\numberwithin{equation}{section}

\usepackage{nameref}
\begin{document}

\maketitle

\vspace*{-10mm}

\begin{abstract}
Branching processes  form an important family of stochastic processes that have been successfully applied in many fields. In this paper, we focus our attention on controlled multi-type branching processes (CMBPs). A  Feller-type diffusion approximation is derived for 
some  critical CMBPs. Namely, we consider a sequence of appropriately scaled random step
functions formed from a critical CMBP with control distributions having expectations 
that satisfy a kind of linearity assumption. 
It is proved that such a sequence converges weakly toward
a squared Bessel process supported by a ray determined by 
an eigenvector of a matrix related to the
offspring mean matrix and the control distributions of the branching process in question.
As applications, among others, we derive  Feller-type diffusion approximations of critical, primitive multi-type branching processes with immigration and some two-sex branching processes.   
We also describe the asymptotic behaviour of the relative frequencies of distinct types of individuals  for critical CMBPs.


\end{abstract}

\section{Introduction}

Branching processes can be well-applied to describe evolutionary
 systems, where elements reproduce according to certain probability laws. 
These processes are commonly used in population dynamics, and, in this context,  
 systems are referred to as populations and elements as individuals or cells. 
In this framework, multi-type branching processes, that 
 are appropriate for modelling the evolution of populations in which different types of individuals coexist, have been successfully applied. 
For instance, these processes are used for modelling
 polymerase chain reaction (see, e.g., Sagitov and Ståhlberg \cite{SAGITOV}),  for modelling cell proliferation kinetics (see, e.g., Gonz\'alez et al.\ \cite{gmpv}, Kimmel and Axelrod \cite[Chapter 5]{ka2015} or Yanev \cite{YANEV2020}) or  for studying extinction of outbreak of diseases (see, e.g., Mwasunda et al.\ \cite{MWASUNDA202273}). We will focus our attention on the class of controlled multi-type branching processes (CMBPs). This class  was first introduced in Gonz\'alez et al.\ \cite{GONZALEZ_MARTINEZ_MOTA_2005}.  The key feature of {CMBPs} is that the number of parents of each type at a given generation is determined by a random control mechanism  that depends on the number of individuals of different types in the previous generation.  
 
CMBPs  form a wide family of branching processes that include, as particular cases, the well-known classical branching processes, namely, multi-type branching processes  without or with immigration (MBPs or MBPIs), or two-sex branching processes (see details in Example \ref{Example1}). Although, for all these  processes, individuals give rise to offspring independently of each others,  in general,  a CMBP  no longer  satisfies the additive property (see, e.g., Athreya and Ney \cite[Chapter 1, page 3]{Athreya-Ney}).  This is an important property that is satisfied, for instance, by MBPs and MBPIs. 
The lack of the additive property for a general CMBP
makes it difficult to study this kind of processes and 
it comes from the following fact. 
Provided that we know the number of different types of individuals in   a generation, in case of a general CMBP, the number of different types of individuals in the next generation is a random sum of some independent random variables, while, in case of a MBP or a MBPI, it is a non-random sum
  (in the sense that the number of summands is deterministic).

This paper aims to obtain a Feller-type diffusion approximation for some critical  CMBPs (details of the model are given in Section \ref{Sec_Prob_model}) and to study the asymptotic behaviour of the relative frequencies of distinct types of individuals. 
The problems under consideration  have a dual motivation. 
On the one hand, it has interest in itself from a theoretical point of view.
Our theoretical results (see Theorem \ref{Thm_main} and its corollaries) may allow us to study relevant applied problems as well. For example, our result on the asymptotic behaviour of relative frequencies may find applications in cell kinetics.
 On the other hand, the obtained scaling limit theorem  can enable us to carry out further research on statistical inference.  For example, one may 
 complete the study on weighted conditional least square estimators 
 for some critical  CMBPs, started in Gonz\'alez et al.\ \cite{GONZALEZ_DEL-PUERTO_2010}, 
or one might start to investigate more general quasi-likelihood estimators. 
This intended research could contribute to the modeling of real data using CMBPs. 
 
From a mathematical point of view, achieving functional limit theorems for {critical branching} processes has attracted the interest of many researchers since 1951, when  
 Feller \cite{FELLER_1951} was able to provide the first formulation for  Galton--Watson processes.
  He proved that the sequence of appropriately scaled 
 critical Galton--Watson processes converges in distribution to a non-negative diffusion process without drift (for a detailed proof based on infinitesimal generators, see also Ethier and Kurtz \cite[Theorem 9.1.3]{EthKur}).
The extension to branching processes with immigration started
 with the pioneering paper of  Wei and Winnicki \cite{WEI_WINNICKI_1989},  where the limit process
 is a squared Bessel process that 
 can be characterized as the pathwise unique
 strong solution of a certain stochastic differential equation (SDE).
Focusing on controlled branching models, a further extension to the critical, single-type case was first carried out by Sriram et al.\ \cite{SRIRAM2007928}, and, more recently, by Gonz\'alez et al.\ \cite{GONZALEZ_MARTIN-CHAVEZ_DEL-PUERTO_2022}.
 Furthermore, for some critical CMBPs, under quite involved technical conditions, 
Gonz\'alez et al.\ \cite[Corollary 4.1]{gonzalez_martinez_mota_2006} proved a conditional weak limit theorem for the one-dimensional distributions 
  provided that the explosion set has a positive probability.
However, according to our knowledge, functional limit theorems for critical CMBPs are not available in the literature.
Our present paper fills this gap, it is a natural extension of the  
 result for critical single-type  controlled branching processes 
 (CBPs) in Gonz\'alez et al.\ \cite{GONZALEZ_MARTIN-CHAVEZ_DEL-PUERTO_2022} 
 to the multi-type case.

Assuming that the expectations of the control distributions satisfy a linear relationship with the population size additively perturbed by a function also depending on the population size (see \eqref{cond_epsilon_z}), we introduce a classification  for such CMBPs based on the spectral radius of a matrix related to the offspring mean matrix and to the control  distributions. Under some  additional hypotheses, we prove that a suitably scaled and normalized critical CMBP converges weakly  toward a squared Bessel process supported by a ray determined by an eigenvector of the aforementioned matrix, see Theorem \ref{Thm_main}.  As corollaries, we are able to  rediscover the known results on  Feller-type diffusion approximations  for critical, primitive MPBIs (see Corollaries \ref{cor-MBPI} and \ref{cor-MBPI-2}). 
 What is even more interesting is that we can apply our main theorem to get Feller-type diffusion approximations for some two-sex branching  processes. We emphasize that no such results are available in the literature. 
Very recently, Bansaye et al.\ \cite{bcms23} have also
proved a scaling limit theorem for a class of two-sex branching processes 
that combine classical asexual Galton–Watson processes and two-sex Galton–Watson branching processes introduced by Daley \cite{DALEY_1968}.
For a comparison of our results and theirs, see Remark \ref{Rem_Bansaye}. Finally, a result on the asymptotic behaviour of the relative frequencies  for critical CMBPs is also derived from Theorem \ref{Thm_main} (see Corollary \ref{Cor_rel_frequency}). This kind of 
 result has potential applications, for instance, in the field of cell kinetics, where it is more usual to measure  relative frequencies instead of the absolute cell counts.

The proof of our main result  (Theorem \ref{Thm_main})  follows 
 the proof scheme of Theorem 3.1 in Ispány and Pap \cite{ispany-pap-2014} for MBPIs, which is based on 
 a weak convergence result for random step processes 
 due to Isp\'any and Pap \cite[Corollary 2.2]{ispany-pap-2010} (see also Theorem \ref{thm-ispany-pap}).
This latter result has been applied in other papers to prove several scaling limit theorems, see, e.g., 
   Isp\'any and Pap \cite{ispany-pap-2014} and R\'ath \cite{Rat}.
The lack of the additive property for a general CMBP
 makes the proof of our Theorem \ref{Thm_main} more involved.
Next, we outline the course of the proof, and we also point out the new ingredients in it.
We start with determining the conditional moments of the branching process (see Proposition \ref{Pro_cond_moment}), which are essential to find out  the asymptotic behaviour of some moments of the branching process 
 (see Lemma \ref{lemma-moments}).
The heart of the proof of Theorem \ref{Thm_main} 
 is an application of Theorem \ref{thm-ispany-pap} for a sequence
 of martingale differences formed from the CMBP in question
 (see \eqref{eq-conv-M} in \nameref{step1}).
Then, in case of MBPIs, the  Feller-type diffusion approximation   follows straightforwardly from a continuous mapping theorem (see Theorem \ref{cont-map-thm}), as Ispány and Pap \cite{ispany-pap-2014} showed.
However, this is not the case for general CMBPs,  
 an extra additional work (see \nameref{step3}) is required.
We also emphasize that the additive perturbation of the 
 expectations of the control distributions mentioned above  (see also \eqref{cond_epsilon_z}) results a new difficulty in the proofs compared to MBPIs,
 and it is addressed in the proofs of Steps 1 and 3.

The paper is structured as follows. 
In Section \ref{Sec_Prob_model}, CMBPs are defined,  and, 
 under the linearity assumption \eqref{formula_cond_exp_hyp} 
 on the conditional expectations, 
we introduce their classification  by distinguishing 
 subcritical, critical and supercritical CMBPs.
By giving examples, we also point out that different types 
 of classical branching processes can be viewed as particular  cases of our model, thus illustrating its wide scope of applicability. 
 In Section \ref{Section_Results}, we collect all the hypotheses that  
 are assumed and we present all of our results obtained.
Section \ref{section-proof} is devoted to the proof of Theorem \ref{Thm_main}, which is structured in four steps
 for an easy reading. 
We close the paper with an Appendix which contains 
 some auxiliary results  such as the asymptotic behaviour of the first and second moments of CMBPs and the second and fourth moments of a corresponding martingale difference (see Lemma \ref{lemma-moments}).

\section{Controlled multi-type branching processes}\label{Sec_Prob_model}

Let $(\Omega,\mathcal{F},\mathrm{P})$ be a fixed probability space on which all the random variables will be defined, and let $\N$, $\Z_+$, $\R$, $\R_+$ and $\R_{++}$ be the set of positive integers, non-negative integers, real numbers, non-negative real numbers, and positive real numbers, respectively. 
For all $\boldsymbol{x},\boldsymbol{y}\in\R^p,$ let us denote by $\boldsymbol{x}
 \preceq \boldsymbol{y}$ if each coordinate of $\boldsymbol{x}$ is less than 
 or equal to the corresponding coordinate of $\boldsymbol{y}$.
For $\bz=(z_1,\ldots,z_p)^\top \in\R^p$, let 
  $|\boldsymbol{z}|\colonequals(\vert z_1\vert,\ldots,\vert z_p\vert)^\top\in\R_+^p$,
 and $\bz^+ \colonequals (z_1^+,\ldots,z_p^+)^\top \in\R_+^p$, 
 where  $x^+$ stands for the positive part of $x\in\R$ and  $^\top$ for the transpose. 
  The natural basis in $\R^p$ is denoted by $\boldsymbol{e}_1,\ldots,\boldsymbol{e}_p$.
The null vector in $\R^p$ is denoted by $\boldsymbol{0}_p$.
The Borel sigma-algebra on $\R^p$ is denoted by $\mathcal{B}(\R^p)$.
The trace of a matrix $\mathsf{A}\in\R^{p\times p}$ is denoted by $\tr{\mathsf{A}}$.
The $p\times p$ identity matrix is denoted by $\mathsf{I}_p$.  
For a matrix  $\mathsf{A}\in\R^{l\times p}$, let 
 $\operatorname{Null}(\mathsf{A}) \colonequals \{\boldsymbol{x}\in\R^p : \mathsf{A}\boldsymbol{x} =  \boldsymbol{0}_l\}$.
Along the paper, we will not distinguish between the notations of 
 the norm of a vector in $\R^p$ and that of a matrix in $\R^{p\times p}$. 
For $\bz\in \R^p$, let $\|\bz\|$ denote the Euclidean norm of $\bz$, 
 and for a matrix $\mathsf{A}\in \R^{p\times p}$, let $\|\mathsf{A}\| \colonequals \max\{\|\mathsf{A}\bz\|: \|\bz\|=1, \, \bz\in\R^p\}$.
For a positive semi-definite matrix $\mathsf{A}\in\R^{p\times p}$,  
 let $\sqrt{\mathsf{A}}$ denote the unique symmetric positive semi-definite square root of $\mathsf{A}$.
For a function $\boldsymbol{h}:\R^p\to\R^p$, by the notation $\boldsymbol{h}(\bz) = \OO(f(\bz))$ as $\Vert\bz\Vert\to\infty$, where $f:\R^p\to\R_+$,
 we mean that there exist $C>0$ and $R>0$ such that $\Vert \boldsymbol{h}(\bz)\Vert\leq C f(\bz)$ for 
 all $\bz\in\R^p$ with $\Vert \bz\Vert>R$. 
Further, if for all $\epsilon>0$ there exists $R>0$ such that $\Vert \boldsymbol{h}(\bz)\Vert\leq \epsilon f(\bz)$ for 
 all $\bz\in\R^p$ with $\Vert \bz\Vert>R$, then we write $\boldsymbol{h}(\bz) = \operatorname{o}(f(\bz))$ as $\Vert\bz\Vert\to\infty$.
Convergence in probability is denoted by $\stackrel{\mathrm{P}}{\longrightarrow}$.
Some further notations for weak convergence of stochastic processes with  càdlàg sample paths are recalled in Section \ref{Section_Results}.

For a fixed $p\in\N$  and a $\Z_+^p$--valued random vector $\boldsymbol{Z}_0$, let us consider a controlled $p$--type branching process $(\boldsymbol{Z}_{k})_{k\in\Z_+}$, defined recursively as
\begin{equation}\label{CMBP_def}
    \boldsymbol{Z}_{k+1}  \colonequals \sum_{i=1}^{p} \sum_{j=1}^{\phi_{k,i}(\boldsymbol{Z}_{k})} \boldsymbol{X}_{k,j,i}, \quad \quad k\in\Z_+,
\end{equation}
where $\boldsymbol{Z}_{k}$, 
 $\boldsymbol{\phi}_{k} (\boldsymbol{z})$, $\boldsymbol{z}\in\Z_+^p$,
 and $\boldsymbol{X}_{k,j,i}$ are $\Z_+^p$--valued random vectors:
\begin{equation*}
    \boldsymbol{Z}_{k}  \equalscolon 
    \begin{pmatrix}
        Z_{k,1} \\
        \vdots \\
        Z_{k,p}
    \end{pmatrix}, \qquad
    \boldsymbol{\phi}_{k} (\boldsymbol{z})  \equalscolon 
    \begin{pmatrix}
        \phi_{k,1} (\boldsymbol{z}) \\
        \vdots \\
        \phi_{k,p} (\boldsymbol{z})
    \end{pmatrix}, \qquad
    \boldsymbol{X}_{k,j,i}  \equalscolon 
    \begin{pmatrix}
        X_{k,j,i,1} \\
        \vdots \\
        X_{k,j,i,p}
    \end{pmatrix}.
\end{equation*}
The intuitive interpretation of the process  $(\boldsymbol{Z}_{k})_{k\in\Z_+}$ is as follows: 
\begin{itemize}
    \item $Z_{k,i}$ is the number of $i$--type individuals in the $k$--th generation,
    \item $\phi_{k,i}(\boldsymbol{Z}_{k})$ is the number of $i$--type progenitors in the $k$--th generation,
    \item $X_{k,j,i,l}$ is the number of $l$--type offsprings of the $j$--th $i$--type progenitor in the $k$--th generation.
\end{itemize}
 
Assume that $\{\boldsymbol{Z}_{0}, \boldsymbol{\phi}_{k} (\boldsymbol{z}),\boldsymbol{X}_{k,j,i}: k\in\Z_+, \, j\in\N,\, \boldsymbol{z}\in\Z_+^p,\,i\in\{1,\ldots,p\}\}$ are independent, $\{\boldsymbol{\phi}_{k}(\boldsymbol{z}) :k\in\Z_+\}$ are identically distributed for each $\boldsymbol{z}\in\Z_+^p$ and $\{\boldsymbol{X}_{k,j,i}: k\in\Z_+,\,j\in\N\}$ are also identically distributed for each $i\in\{1,\ldots,p\}$.
The distributions of $\boldsymbol{\phi}_0(\boldsymbol{z})$, $\bz\in\Z_+^p$, 
 are called the control distributions. 
 
Moreover, for clarity of the readers, let us write out \eqref{CMBP_def} for
 controlled two-type branching processes. 
The number of $i$--type individuals in the initial generation 
 is $Z_{0,i}$, $i=1,2$, and
 the recursive definition in \eqref{CMBP_def} takes the form
\begin{equation*}
    \begin{pmatrix}
        Z_{k+1,1} \\
        Z_{k+1,2}
    \end{pmatrix} = \sum_{j=1}^{\phi_{k,1}(Z_{k,1}, Z_{k,2})} \begin{pmatrix}
        X_{k,j,1,1} \\
        X_{k,j,1,2}
    \end{pmatrix} + \sum_{j=1}^{\phi_{k,2}(Z_{k,1}, Z_{k,2})} \begin{pmatrix}
        X_{k,j,2,1} \\
        X_{k,j,2,2}
    \end{pmatrix},
    \qquad k\in\Z_+.
\end{equation*} 
Part (iv) of Example \ref{Example1} serves as a good illustration 
 of a two-type case.

It is easy to check that the process defined in \eqref{CMBP_def} is a 
 $\Z_+^p$-valued Markov chain. This model is very general, several other popular branching processes can be seen as particular cases of the CMBPs, see Example \ref{Example1}. 

\begin{example}\label{Example1}
\hfill
\begin{itemize}
    \item[(i)] Multi-type branching processes.
      We get the subclass of $p$--type  branching processes 
      without immigration by defining the  deterministic control function as $\boldsymbol{\phi}_{k} (\boldsymbol{z})  \colonequals  \boldsymbol{z}$, $\bz\in\Z_+^p$, $k\in\Z_+$, 
      in \eqref{CMBP_def}.
    \item[(ii)] Multi-type branching processes with immigration. Let us consider such a MBPI ${(\boldsymbol{Y}_{k})_{k\in\Z_+}}$ given by
    \begin{equation}\label{def_MBPI}
        \boldsymbol{Y}_{k+1} \colonequals \sum_{i=1}^{p} \sum_{j=1}^{Y_{k,i}} \boldsymbol{\xi}_{k,j,i} + \boldsymbol{I}_{k+1}, \quad \quad k\in\Z_+,
    \end{equation}
    where $\{ \boldsymbol{Y}_0, \, \boldsymbol{\xi}_{k,j,i},  \boldsymbol{I}_k :  k\in\Z_+, \, j\in\N,\, i\in\{1,\ldots,p\}\}$ are independent $\Z_+^p$--valued random vectors, 
    $\{ \boldsymbol{\xi}_{k,j,i}: k\in\Z_+,\, j\in\N \}$ are identically distributed for each $i\in\{1,\ldots,p\}$ (offspring distributions) and $\{ \boldsymbol{I}_k:  k\in\Z_+\}$ are also identically distributed (immigration distribution).
    Note that $(\boldsymbol{Y}_k)_{k\in\Z_+}$
    can be written as a $(p+1)$--type branching process 
     with
    \begin{equation}\label{help_MBPI_2nd}
     \boldsymbol{Z}_{k} \colonequals
        \begin{pmatrix}
            \boldsymbol{Y}_{k} \\
            {1}
        \end{pmatrix}, 
        \qquad
        \boldsymbol{X}_{k,j,i} \colonequals
        \begin{pmatrix}
            \boldsymbol{\xi}_{k,j,i} \\
            0
        \end{pmatrix},  \qquad
        {\boldsymbol{X}_{k,1,p+1} } \colonequals
        \begin{pmatrix}
            \boldsymbol{I}_{k+1} \\
            1
        \end{pmatrix}
    \end{equation}
    for $k\in\Z_+$, $j\in \N$,
    $i\in\{1,\ldots,p\}$,  when $\boldsymbol{\phi}_{k} (\boldsymbol{z})  \colonequals  \boldsymbol{z}$, $\bz\in\Z_+^{p+1}$, $k\in\Z_+$ (see (i)).
    Moreover, $(\boldsymbol{Y}_k)_{k\in\Z_+}$ can also be written as a controlled $(p+1)$--type branching process with the choices given in \eqref {help_MBPI_2nd} 
    and the control functions
    \[ 
      \boldsymbol{\phi}_{k} (\boldsymbol{z}) \colonequals
        \begin{pmatrix}
            z_1 \\
            \vdots \\
            z_p \\
            1
        \end{pmatrix},\qquad \bz\in\Z_+^{p+1},\quad  k\in\Z_+.
    \]
    \item[(iii)] Multi-type branching processes with migration. Let us consider a stochastic process ${(\boldsymbol{Z}_{k})_{k\in\Z_+}}$ given by
    \begin{equation*}
        \boldsymbol{Z}_{k+1}  \colonequals \sum_{i=1}^{p} \sum_{j=1}^{Z_{k,i}+M_{k,i}(Z_{k,i})} \boldsymbol{X}_{k,j,i}, \quad \quad k\in\Z_+,
    \end{equation*}
    where $\{\boldsymbol{Z}_{0}, \boldsymbol{M}_{k}(\bz) ,\boldsymbol{X}_{k,j,i}: k\in\Z_+, \, j\in\N,\, \boldsymbol{z}\in\Z_+^p,\,i\in\{1,\ldots,p\}\}$ are independent $\Z^p$--valued random vectors 
    such that 
    $\boldsymbol{Z}_{0}$ and $\boldsymbol{X}_{k,j,i}$ have non-negative coordinates,
    $\{\boldsymbol{M}_{k}(\bz) \colonequals (M_{k,1}(z_1),\ldots,M_{k,p}(z_p))^\top : k\in\Z_+\}$ are identically distributed for each $\bz\in\Z_+^p$ with  range contained in $[-z_1,\infty) \times \cdots \times [-z_p,\infty)$, and $\{\boldsymbol{X}_{k,j,i}: k\in\Z_+,\,j\in\N\}$ are also identically distributed for each $i\in\{1,\ldots,p\}$. For each $i\in\{1,\ldots,p\}$, $M_{k,i}(z_i)$ can be interpreted as a migration component for the $i$--type individuals in the $k$--th generation. Depending on the sign of $M_{k,i}(z_i)$, there is emigration (negative value), immigration (positive  value) or no migration (zero value). 
    In fact, this is an equivalent way of writing a CMBP, because every $\boldsymbol{\phi}_{k} (\boldsymbol{z})$ can be written as $\boldsymbol{z} + \boldsymbol{M}_{k}(\bz)$, $\bz\in\Z_+^p$, $k\in\N$, with appropriate choices  of $\boldsymbol{M}_{k}(\bz)$.

    \item[(iv)] Two-sex Galton--Watson branching processes with immigration (2SBPIs). 
    Let us consider such a process $(F_k,M_k)_{k\in\Z_+}$
    with offspring and immigration distribution defined as follows. Let $(F_0,M_0)$ be an $\Z_+^2$--valued random variable 
    (random initial generation) and
    \begin{equation}\label{eq-definition}
    \begin{aligned}
    (F_{k+1},M_{k+1}) & \colonequals  \sum_{j=1}^{U_k} (f_{k,j},m_{k,j}) + (F_{k+1}^{I},M_{k+1}^{I}),\qquad  k\in\Z_+,\\
    U_k & \colonequals  L(F_k,M_k), \qquad k\in\Z_+ ,
    \end{aligned}
    \end{equation}
    where $\{(F_0,M_0),(f_{k,j},m_{k,j}), (F_k^I,M_k^I) : k\in\Z_+,\, j\in\N\}$ are independent $\Z_+^2$--valued random vectors, 
    $\{ (f_{k,j},m_{k,j}) : k\in\Z_+,\, j\in\N \}$ are identically distributed 
    (offspring distribution), and $\{ (F_k^I,M_k^I) : k\in\Z_+\}$ are also identically distributed (immigration distribution). 
    Further, $(U_k)_{k\in\Z_+}$ is a sequence of mating units corresponding to the mating function $L:\Z_+\times \Z_+ \to\Z_+$, assumed to be non-decreasing in each  argument. 
    In the terminology of Asmussen \cite{Asm}, $L$ can be called a marriage function as well.
    Then the 2SBPI $(F_k,M_k)_{k\in\Z_+}$ can be considered a 
    CMBP $(\bZ_k)_{k\in\Z_+}$ given by 
    \begin{align*}
     \boldsymbol{Z}_{k}  \colonequals 
        \begin{pmatrix}
            F_{k} \\
            M_{k}
        \end{pmatrix}, \qquad
        \boldsymbol{\phi}_{k} (\boldsymbol{z})  \colonequals 
        \begin{pmatrix}
            L(\boldsymbol{z}) \\
            1 
        \end{pmatrix}, \qquad
        \boldsymbol{X}_{k,j,1}  \colonequals 
        \begin{pmatrix}
            f_{k,j} \\
            m_{k,j}
        \end{pmatrix}, \qquad
        \boldsymbol{X}_{k,j,2}  \colonequals 
        \begin{pmatrix}
            F_{k+1}^{I} \\[1mm]
            M_{k+1}^{I}
        \end{pmatrix} 
    \end{align*}
    for $k\in\Z_+$, $j\in\N$, and $\bz\in\Z_+^2$.
  In this case, we emphasize that the coordinates of the control 
  can be interpreted in a different way compared to what is written after \eqref{CMBP_def}.
 Namely, the first coordinate of the control denotes the number of couples, while the second one denotes an extra couple, which brings in female and male immigrants.
    \proofendb
\end{itemize}
\end{example}

\begin{remark}\label{rem-interpretation}
    As happens for the 2SBPI, the intuitive interpretation of the branching process that appears 
     in  part (ii) of Example \ref{Example1} 
     is lost when we see it as a $(p+1)$--type branching process or as a controlled $(p+1)$--type branching process.  Indeed,  we add  the possibility that there is an extra type of individuals in the population,  but the number of them in all generations is $1$.
     The only one individual of extra type always gives birth one individual of extra type and possibly some individuals of other types, which correspond to the immigrants arriving in the population.  
    Despite the fact that both cases lack practical interpretability,  this mathematical description is perfectly valid and useful: 
    \begin{itemize}
        \item[(i)] It will enable us to recover the known result 
         on the asymptotic behaviour of critical, primitive MBPIs 
        due to Isp\'any and Pap \cite[Theorem 3.1]{ispany-pap-2014}
        (see Corollary \ref{cor-MBPI}). Furthermore, if we do not want to lose the intuitive interpretation, we could consider the controlled $p$--type branching process $(\boldsymbol{Z}_{k})_{k\in\Z_+}$ 
        with
        $\bZ_0 \colonequals \boldsymbol{Y}_0$,
        $\boldsymbol{\phi}_{k} (\boldsymbol{z})  \colonequals  \boldsymbol{z} + \boldsymbol{I}_{k+1}$, and $\boldsymbol{X}_{k,j,i} \colonequals  \boldsymbol{\xi}_{k,j,i}$ for $k\in\Z_+$, $j\in\N$, and $\bz\in\Z_+^p$, 
         where $\{ \boldsymbol{Y}_0, \, \boldsymbol{\xi}_{k,j,i},  \boldsymbol{I}_k :   k\in\Z_+, \, j\in\N,\, i\in\{1,\ldots,p\}\}$ are
         given in part (ii) of Example \ref{Example1}.
         This is the third representation of a MBPI (for the other two ones, see part (ii) of Example \ref{Example1}). 
       Note  also that $(\boldsymbol{Z}_{k})_{k\in\Z_+}$ is 
       a particular case of the multi-type branching process with migration  
        in part (iii) in Example \ref{Example1}.
        In the single-type critical case, Gonz\'alez et al.\ \cite[Remark 3.2.2]{GONZALEZ_MARTIN-CHAVEZ_DEL-PUERTO_2022} proved that (under some moment assumptions)  $(n^{-1} Z_{\lfloor nt\rfloor})_{t\in\R_+}$
        converges weakly as $n\to\infty$ toward a limit process, which coincides with the limit process for a critical (usual) single-type branching processes with immigration scaled and normalized in the same way.
        Corollary \ref{cor-MBPI-2} will imply that a version of this statement remains true in the $p$--type critical case.

        \item[(ii)] It is the first time that a 2SBPI is written as a CMBP 
        (see part (iv) of Example \ref{Example1}), such a rewriting has not been considered until now in the literature. In fact, since the control distributions are deterministic in case of a 2SBPI, it suggests a possible way to generalize the notion of a 2SBPI by allowing random mating functions. Earlier, only two-sex Galton--Watson branching processes without immigration, introduced by Daley \cite{DALEY_1968}, were viewed as special controlled branching processes (see Sevast'yanov and Zubkov \cite[model 3]{SEVASTYANOV_ZUBKOV_1974}).
        Furthermore, Theorem \ref{Thm_main} together with part (iv) of Example \ref{Example1} will allow us to obtain 
        scaling limit theorems for some critical 2SBPIs previously not considered in the literature (see Corollary \ref{cor-2SBPI}).
        \proofendb
    \end{itemize}  
\end{remark}

Let us introduce notations for some moments. 
In all what follows, we suppose that $\ev{\|\boldsymbol{X}_{0,1,i}\|^4}<\infty$ for $ i\in\{1,\ldots,p\}$ and $\ev{\|\boldsymbol{\phi}_{0}(\boldsymbol{z})\|^4}<\infty$ for $\boldsymbol{z}\in\Z_+^p$, and we denote
\begin{align}
    \boldsymbol{m}_i & \colonequals  \ev{\boldsymbol{X}_{0,1,i}}\in\R_+^p, & \boldsymbol{\varepsilon}(\boldsymbol{z}) & \colonequals   \ev{\boldsymbol{\phi}_{0}(\boldsymbol{z})}\in\R_+^p, \label{notats_moments_1} \\
    \mathsf{\Sigma}_i & \colonequals  \var{\boldsymbol{X}_{0,1,i}}\in\R^{p\times p}, & \mathsf{\Gamma}(\boldsymbol{z}) & \colonequals  \var{\boldsymbol{\phi}_{0}(\boldsymbol{z})}\in\R^{p\times p}, \label{notats_moments_2} \\\label{notats_moments_3}
     \zeta_{i,l} & \colonequals  \ev{(X_{0,1,i,l}-m_{i,l})^4}\in\R_+, & \kappa_i(\boldsymbol{z}) & \colonequals  \ev{(\phi_{0,i}(\boldsymbol{z}) - \varepsilon_i(\boldsymbol{z}))^4}\in\R_+,
\end{align}
where $i,l\in\{1,\ldots,p\}$ and $\bz\in\Z_+^p$.

Note that the moments defined in \eqref{notats_moments_1} and \eqref{notats_moments_2}
 are of course well-defined under weaker assumptions, namely, under the existence
  of first and second order moments  of the offspring and control distributions, respectively.
We also remark that  $\boldsymbol{m}_i$, $\mathsf{\Sigma}_i$  and $\zeta_{i,l}$, $i,l\in\{1,\ldots,p\}$, do not depend on the control distributions, while
 $\boldsymbol{\varepsilon}(\boldsymbol{z})$, $\mathsf{\Gamma}(\boldsymbol{z})$,
 and $\kappa_i(\boldsymbol{z})$, $\bz\in\Z_+^p$, $i\in\{1,\ldots,p\}$, do depend.

Let us consider the canonical filtration of the process $\mathcal{F}_k \colonequals \sigma(\boldsymbol{Z}_{0},\ldots,\boldsymbol{Z}_{k}),$ $k\in\Z_+,$ introduce the matrix
$\mathsf{m}  \colonequals  (\boldsymbol{m}_1,\ldots, \boldsymbol{m}_p)
\in\R_+^{p\times p},$ and the operator $\odot :\R^p \times (\R^{p\times p})^p  
 \to  \R^{p\times p}$, $\boldsymbol{z}\odot\boldsymbol{\mathsf{A}} 
     \colonequals  \sum_{i=1}^p z_i\mathsf{A}_i$ for 
 $\boldsymbol{z}=(z_1,\ldots,z_p)^\top\in\R^p$ and $\boldsymbol{\mathsf{A}}=(\mathsf{A}_1,\ldots,\mathsf{A}_p) \in (\R^{p\times p})^p$.

The following proposition is a multi-type counterpart of 
 Proposition 3.5 in Gonz\'alez et al.\ \cite{GonPueYan_2018}.
    
\begin{prop}\label{Pro_cond_moment}
    For each $k\in\N$, we have
    \begin{align}
        \evcond{\boldsymbol{Z}_{k}}{\mathcal{F}_{k-1}} &= \mathsf{m}\boldsymbol{\varepsilon}(\boldsymbol{Z}_{k-1}), \label{eq-conditional-expected-value} \\
        \varcond{\boldsymbol{Z}_{k}}{\mathcal{F}_{k-1}} &= \boldsymbol{\varepsilon}(\boldsymbol{Z}_{k-1})\odot\boldsymbol{\mathsf{\Sigma}} + \mathsf{m} \mathsf{\Gamma}(\boldsymbol{Z}_{k-1}) \mathsf{m}^{\top}, \label{eq-conditional-variance}
    \end{align}
   where $\boldsymbol{\mathsf{\Sigma}} \colonequals (\mathsf{\Sigma}_1, \ldots,\mathsf{\Sigma}_p)\in (\R^{p\times p})^p$.  
\end{prop}

\begin{proof}
Let $k\in\N$ be fixed arbitrarily.
 By the Markov property, we get
 $\ev{\bZ_k \mid \mathcal{F}_{k-1}} = \ev{\bZ_k \mid \bZ_{k-1}}$.
Using that $\boldsymbol{\phi}_{k-1}(\boldsymbol{z})$, $\bz\in\Z_+^p$,
 and $\mathbf{X}_{k-1,j,i}$, $j\in\N$, $i\in\{1,\ldots,p\}$, are independent
 of each other and of $\bZ_{k-1}$, we have
    \begin{align*}
      \ev{\bZ_k \mid \bZ_{k-1} = \bz}
       & = \ev{\sum_{i=1}^{p} \sum_{j=1}^{\phi_{k-1,i}(\boldsymbol{z})} \boldsymbol{X}_{k-1,j,i}} = \sum_{i=1}^{p} \ev{\evcond{\sum_{j=1}^{\phi_{k-1,i}(\boldsymbol{z})} \boldsymbol{X}_{k-1,j,i}}{\boldsymbol{\phi}_{k-1}(\boldsymbol{z})}} \\
       &= \sum_{i=1}^{p} \ev{\phi_{k-1,i}(\boldsymbol{z})    
                                \ev{\boldsymbol{X}_{k-1,j,i} } } 
        = \sum_{i=1}^{p} \ev{\phi_{0,i}(\boldsymbol{z})}  \boldsymbol{m}_i
        = \mathsf{m}\boldsymbol{\varepsilon}(\boldsymbol{z}), \qquad \bz\in\Z_+^p,
    \end{align*}
    which implies \eqref{eq-conditional-expected-value}.
    
Now we turn to prove \eqref{eq-conditional-variance}.
Using again the Markov property and \eqref{eq-conditional-expected-value}, 
we have that $ \varcond{\bZ_k}{\mathcal{F}_{k-1}} = \varcond{\bZ_k}{\bZ_{k-1}}$,
and, using also the law of total variance, similarly as before, 
we get
 \begin{align*}
     \varcond{\bZ_k}{\bZ_{k-1}=\bz} = \var{\sum_{i=1}^{p} \sum_{j=1}^{\phi_{k-1,i}(\boldsymbol{z})} \boldsymbol{X}_{k-1,j,i}} &= \ev{\varcond{\sum_{i=1}^{p} \sum_{j=1}^{\phi_{k-1,i}(\boldsymbol{z})} \boldsymbol{X}_{k-1,j,i}}{\boldsymbol{\phi}_{k-1}(\boldsymbol{z})}} \\
        &\quad + \var{\evcond{\sum_{i=1}^{p} \sum_{j=1}^{\phi_{k-1,i}(\boldsymbol{z})} \boldsymbol{X}_{k-1,j,i}}{\boldsymbol{\phi}_{k-1}(\boldsymbol{z})}}
 \end{align*}
    for $\bz\in\Z_+^p$, where 
    \begin{align*}
        \ev{\varcond{\sum_{i=1}^{p} \sum_{j=1}^{\phi_{k-1,i}(\boldsymbol{z})} \boldsymbol{X}_{k-1,j,i}}{\boldsymbol{\phi}_{k-1}(\boldsymbol{z})}} 
       & = \ev{\sum_{i=1}^{p} \sum_{j=1}^{\phi_{k-1,i}(\boldsymbol{z})} 
        \var{\boldsymbol{X}_{k-1,j,i}} }  \\
       & = \sum_{i=1}^{p} \ev{\phi_{k-1,i}(\boldsymbol{z}) \mathsf{\Sigma}_i} = \boldsymbol{\varepsilon}(\boldsymbol{z})\odot\boldsymbol{\mathsf{\Sigma}},
    \end{align*}
     and
    \begin{align*}
        \var{\evcond{\sum_{i=1}^{p} \sum_{j=1}^{\phi_{k-1,i}(\boldsymbol{z})} \boldsymbol{X}_{k-1,j,i}}{\boldsymbol{\phi}_{k-1}(\boldsymbol{z})}} 
        &= \var{\sum_{i=1}^{p} \phi_{k-1,i}(\boldsymbol{z})\boldsymbol{m}_{i}} =  \var{\mathsf{m}\boldsymbol{\phi}_{k-1}(\boldsymbol{z})}\\
        &= \mathsf{m}\var{\boldsymbol{\phi}_{k-1}(\boldsymbol{z})}\mathsf{m}^\top = \mathsf{m}\mathsf{\Gamma}(\boldsymbol{z})\mathsf{m}^\top.
    \end{align*}
   This yields \eqref{eq-conditional-variance}.   
\end{proof}

From now on, we assume that there exist a matrix
 $\mathsf{\Lambda}\in\R^{p\times p}$ and a function $\boldsymbol{h}: \Z_+^p
  \to\R^p$ with $\|\boldsymbol{h}(\boldsymbol{z})\| =  \oo(\|\bz\|)$ as $\|\boldsymbol{z}\|\to\infty$ such that
\begin{align}\label{cond_epsilon_z}
    \boldsymbol{\varepsilon}(\boldsymbol{z}) &= \mathsf{\Lambda}\boldsymbol{z} + \boldsymbol{h}(\boldsymbol{z}), \qquad \boldsymbol{z}\in\Z_+^p.
\end{align}
Roughly speaking, the assumption \eqref{cond_epsilon_z} means that the average quantity of ancestors can be expressed as a linear map of the number of individuals affected by a minor perturbation, which becomes insignificant (negligible) in comparison with the population size when this latter is large enough. It is worth noting that migration of parents, both emigration and immigration, is permitted under (\ref{cond_epsilon_z})
(see, for instance, part (iii) of Example \ref{Example1} assuming  that 
  $\|\ev{\boldsymbol{M}_0(\bz)}\| = \oo(\|\bz\|)$ as $\|\bz\|\to\infty$).
In the single-type case, i.e., in case of $p=1$, 
a  corresponding condition has already been assumed (see, e.g.,  
 the condition (i) in Theorems 1 and 2 in
Gonz\'alez et al.\ \cite{GONZALEZ_MOLINA_DEL-PUERTO_2005}).
In the multi-type case, the assumption \eqref{cond_epsilon_z} was also considered in  Gonz\'alez and del Puerto 
 \cite[condition (11.2)]{GONZALEZ_DEL-PUERTO_2010}
 and in Gonz\'alez et al.\ \cite[condition (4.2)]{gonzalez_martinez_mota_2006}, but only with a diagonal matrix $\mathsf{\Lambda}$ having 
 non-negative entries.
 
Using \eqref{eq-conditional-expected-value} and \eqref{cond_epsilon_z},  we get
 \begin{align}\label{formula_cond_exp_hyp}
    \evcond{\boldsymbol{Z}_{k}}{\mathcal{F}_{k-1}} = \mathsf{m}\mathsf{\Lambda}\boldsymbol{Z}_{k-1} + \mathsf{m}\boldsymbol{h}(\boldsymbol{Z}_{k-1}), \qquad k\in\N.
 \end{align}

We can introduce a classification for a subclass of CMBPs  satisfying  \eqref{formula_cond_exp_hyp} based on the  asymptotic behaviour of 
$\ev{\boldsymbol{Z}_k}$ as $k\to\infty$.
Taking expectations of both sides of \eqref{formula_cond_exp_hyp}, by 
 a recursive argument, one can derive a formula for $\ev{\boldsymbol{Z}_k}$,
 $k\in\N$ (see \eqref{help_Exp_Z_k_rec} in Appendix \ref{appendix}). 
 From this, one can see that in the description of the asymptotic behaviour of 
 $\ev{\boldsymbol{Z}_k}$ as $k\to\infty$, the powers of the matrix 
 $\tilde{\mathsf{m}}  \colonequals  \mathsf{m}\mathsf{\Lambda}$
 play a crucial role. 
We assume that $\Tilde{\mathsf{m}}\in\R_+^{p\times p}$ and it has only one eigenvalue  of maximum modulus $\tilde{\rho}$ with algebraic and geometric multiplicities equal one. 
 By the spectral theory of matrices, 
 the asymptotic behaviour of $(\mathsf{m}\mathsf{\Lambda})^k$ as $k\to\infty$ is determined by $\tilde{\rho}$, the other eigenvalues do not come into play.
Motivated by this,  by definition, we say that  such a CMBP is subcritical if $\tilde{\rho}<1$, critical if $\tilde{\rho}=1$, and supercritical if $\tilde{\rho}>1$.  It is important to highlight that whenever the matrix $\tilde{\mathsf{m}}$ is primitive 
(i.e., there exists $n\in\N$ such that 
all the entries of $\tilde{\mathsf{m}}^n$ are positive), 
the Perron--Frobenius theorem (see also Lemma \ref{lemma-primitive}) guarantees the unique existence of $\tilde{\rho}$. Moreover, the usual classification of multi-type branching processes without and with immigration 
coincides with the one proposed here (see, e.g., Athreya and Ney \cite[Chapter V, page 186]{Athreya-Ney} and Kaplan \cite[page 948]{kaplan}).
  
\section{Results}\label{Section_Results}

Let $(\bZ_k)_{k\in\Z_+}$ be a controlled $p$--type branching process  given in \eqref{CMBP_def}.
Let us introduce the following hypotheses.

\begin{hyp}\label{hyp-second-moment-Z0}
    $\ev{\|\bZ_0\|^2}$, $\ev{\|\boldsymbol{X}_{0,1,i}\|^4}$ and $\ev{\|\boldsymbol{\phi}_{0}(\boldsymbol{z})\|^4}$ are finite for each $i\in\{1,\ldots,p\}$ and $\bz\in\Z_+^p$.
\end{hyp}

\begin{hyp}\label{hyp-g}
    There exist $\mathsf{\Lambda}\in\R^{p\times p}$, $\boldsymbol{\alpha}\in\R^p$ and a function $\boldsymbol{g}:\Z_+^p\to\R^p$
     with $\|\boldsymbol{g}(\boldsymbol{z})\| = \oo(1)$ as $\|\boldsymbol{z}\|\to\infty$
     such that
    \begin{equation*}
        \boldsymbol{\varepsilon}(\boldsymbol{z}) = \mathsf{\Lambda}\boldsymbol{z} + \boldsymbol{\alpha} + \boldsymbol{g}(\boldsymbol{z}), \qquad \boldsymbol{z}\in\Z_+^p.
    \end{equation*}
\end{hyp}

\begin{hyp}\label{hyp-variance-control-o}
    $\|\mathsf{\Gamma}(\boldsymbol{z})\|=\oo(\|\boldsymbol{z}\|)$ as $\|\boldsymbol{z}\|\to\infty .$
\end{hyp}

\begin{hyp}\label{hyp-moment-4}
    $\kappa_i(\boldsymbol{z}) =\OO(\|\boldsymbol{z}\|^2)$ as $\|\boldsymbol{z}\|\to\infty $ for $i=1,\ldots,p.$
\end{hyp}

\begin{hyp}\label{hyp-primitive-relaxed}
    The matrix $\tilde{\mathsf{m}} := \mathsf{m}\mathsf{\Lambda}$  belongs to $\R_+^{p\times p}$ and
        \begin{enumerate}[label=(\alph*)]
            \item $\tilde{\rho}\colonequals 1$ is an eigenvalue of $\tilde{\mathsf{m}}$  having algebraic and geometric multiplicities $1$, and the absolute values of the other eigenvalues of $\tilde{\mathsf{m}}$ are less than $1$, 
            \item there exist a unique right eigenvector $\Tilde{\boldsymbol{u}}\in\R_{+}^p$ and a unique left eigenvector $\Tilde{\boldsymbol{v}}\in\R_{+}^p$ corresponding to $\tilde{\rho} = 1$ such that $\Tilde{u}_1 + \ldots + \Tilde{u}_p = 1$, $\mathsf{\Lambda}\Tilde{\boldsymbol{u}}\in\R_{+}^p$ and $\Tilde{\boldsymbol{v}}^\top \Tilde{\boldsymbol{u}} = 1$, \label{hyp-eigenvectors}
            \item $\lim_{k\to\infty} \tilde{\mathsf{m}}^k = \Tilde{\mathsf{\Pi}} $ and there exist $\Tilde{c}\in\R_{++}$ and $\Tilde{r}\in(0,1)$ such that $\|\tilde{\mathsf{m}}^k - \Tilde{\mathsf{\Pi}}\|\leq \Tilde{c}\Tilde{r}^k$ for each $k\in\N$, where $\Tilde{\mathsf{\Pi}} \colonequals \Tilde{\boldsymbol{u}}\Tilde{\boldsymbol{v}}^\top $. \label{hyp-pi-tilde}
        \end{enumerate}
\end{hyp}

\begin{hyp}\label{hyp-explosion}
For all $\epsilon>0$ and $B>0$, there exists $k_0(\epsilon,B)\in\N$ such that 
$\prob{\Vert \bZ_k \Vert\leq B }\leq \epsilon$ for each $k\geq k_0(\epsilon,B)$, $k\in\N$. 
\end{hyp}

\begin{remark}\label{remark-hyp}
\hfill
\begin{enumerate}[label=(\roman*)]
    \item In Hypotheses \ref{hyp-second-moment-Z0} and \ref{hyp-moment-4}, 
     the fourth order moment assumptions for the offspring and control distributions are used in the proofs only for checking the conditional
      Lindeberg condition, namely,  condition (iii) of Theorem \ref{thm-ispany-pap}, in order to prove convergence of some random step processes toward a diffusion process.
      For critical  single-type CBPs, assuming that the explosion set has probability one, a detailed exposition of a proof of the conditional Lindeberg condition in question under second order moment assumptions can be found in Gonz\'alez et al.\ \cite{GONZALEZ_MARTIN-CHAVEZ_DEL-PUERTO_2022}.

    \item If Hypothesis \ref{hyp-g} holds for $\boldsymbol{\phi}_0(\boldsymbol{z})$, $\bz\in\Z_+^p$, then it also holds for any  of its linear transformations  $\mathsf{B}\boldsymbol{\phi}_0(\boldsymbol{z}) + \boldsymbol{\beta}$, 
    $\bz\in\bZ_+^p$, with $\mathsf{B}\in\R^{p\times p}$ and $\boldsymbol{\beta}\in\R^p$, by replacing  $\mathsf{\Lambda}$, $\boldsymbol{\alpha}$ and $\boldsymbol{g}$ with $\mathsf{B\Lambda}$, $\mathsf{B}\boldsymbol{\alpha} + \boldsymbol{\beta}$ and $\mathsf{B}\boldsymbol{g}$, respectively.
    Note that, in Hypothesis \ref{hyp-g}, $\boldsymbol{\alpha}\in\R^p$ can have negative coordinates as well, for an example, see part (i) of Remark \ref{rem-abs-state}.
    
    \item \label{rem-primitive} Taking into account Lemma \ref{lemma-primitive}, we have that Hypothesis \ref{hyp-primitive-relaxed} holds if the matrix $\tilde{\mathsf{m}}$ is primitive, its Perron--Frobenius eigenvalue $\Tilde{\rho}$  equals $1$, and $\mathsf{\Lambda} \Tilde{\boldsymbol{u}}\in\R_{+}^p$, where $\Tilde{\boldsymbol{u}}$ is the right Perron--Frobenius eigenvector of $\tilde{\mathsf{m}}$ corresponding to $\Tilde{\rho}$. Moreover, notice that if $\mathsf{m}$ is primitive, then a set of sufficient conditions for $\Tilde{\mathsf{m}}$ to be primitive is that $\mathsf{\Lambda}\in\R_{++}^{p\times p}$ and that $\mathsf{m}$ and $\mathsf{\Lambda}$ commute.  Note also that if $\mathsf{\Lambda}\in\R_+^{p\times p}$,
    then $\Tilde{\mathsf{m}}\in\R_+^{p\times p}$ and $\mathsf{\Lambda}\boldsymbol{x}\in\R_+^p$ for all $\boldsymbol{x}\in\R_+^p$.

    \item     A sufficient condition for Hypothesis \ref{hyp-explosion} is the almost sure explosion of the process, i.e. $\prob{\|\boldsymbol{Z}_{k}\| \to\infty  \text{ as } k\to\infty} = 1$.
    Indeed,
    for all $B>0$, we have
    \begin{equation*}
        \{ \|\boldsymbol{Z}_{k}\| 
        \to\infty  \text{ as } k\to\infty \} \subset  \bigcup_{k=0}^\infty \bigcap_{n = k}^{\infty} \{\|\boldsymbol{Z}_{n}\| > B \}.
    \end{equation*}
    Therefore, by the continuity of probability, we have for all $B>0$,
    \begin{equation*}
        1 = \prob{ \bigcup_{k=0}^\infty \bigcap_{n = k}^{\infty} \{\|\boldsymbol{Z}_{n}\| > B \}} = \lim_{k\to\infty} \prob{\bigcap_{n = k}^{\infty} \{\|\boldsymbol{Z}_{n}\| > B \}}
        \leq \liminf_{k\to\infty} \prob{\|\boldsymbol{Z}_k\| > B},
    \end{equation*} 
    yielding that $\lim_{k\to\infty} \prob{\|\boldsymbol{Z}_k\| > B} =1$ for all $B>0$, or equivalently 
    $\lim_{k\to\infty} \prob{\|\boldsymbol{Z}_k\| \leq B} =0$ for all $B>0$.
    Consequently, Hypothesis \ref{hyp-explosion} 
 holds.
  It is also interesting to notice that Hypothesis \ref{hyp-explosion} can not be derived from Hypotheses \ref{hyp-second-moment-Z0}--\ref{hyp-primitive-relaxed}, see 
  part (ii) of Remark \ref{rem-abs-state}  for a nontrivial example, where Hypotheses \ref{hyp-second-moment-Z0}--\ref{hyp-primitive-relaxed} hold, but Hypothesis \ref{hyp-explosion} does not.
    
    \item In case of $\boldsymbol{g}\equiv  \boldsymbol{0}_p$, 
     Hypothesis \ref{hyp-explosion} is not needed  for our forthcoming main Theorem \ref{Thm_main}.
    In fact, in the proof of Theorem \ref{Thm_main} it will be seen that Hypothesis \ref{hyp-explosion} is only used  for deriving \eqref{hyp-explosion-use-1} and \eqref{hyp-explosion-use-2},  which are trivially satisfied for 
    $\boldsymbol{g}\equiv  \boldsymbol{0}_p$.
    \label{rem-explosion}
    \proofendb
\end{enumerate}
\end{remark}

Next, we give an example for a controlled $p$--type branching process for which the Hypotheses \ref{hyp-second-moment-Z0}--\ref{hyp-primitive-relaxed} hold.

\begin{example}
    Let us consider a controlled $p$--type branching process with $\ev{\|\bZ_0\|^2}<\infty$, $\ev{\|\boldsymbol{X}_{0,1,i}\|^4}<\infty$, $i\in\{1,\ldots,p\}$, and  
 $$ \phi_{k,i}(\bz) =   z_i +  U_{k,i}\1_{\{z_i > 0\}},
       \qquad k\in\Z_+,\quad \bz\in\Z_+^p,\quad i\in\{1,\ldots,p\},
 $$
where $\{U_{k,i} : k\in\Z_+, i\in\{1,\ldots,p\}\}$ 
 are independent and identically distributed $\Z_+^p$--valued random variables
 such that their common distribution is the uniform distribution on the set $\{-1,0,1\}$, and $\1_A$ denotes the indicator function of a set $A$.
The interpretation is as follows. Whenever there are $i$--type individuals, 
 there is a migratory component $U_{k,i}$ that allows the emigration/immigration 
 of one $i$--type progenitor or there could be no migration
 with the same probabilities $\frac{1}{3}$. 
 This is a particular case of model (iii) in Example \ref{Example1}. 
Hypothesis \ref{hyp-second-moment-Z0} is trivially fulfilled and Hypotheses 
 \ref{hyp-g}, \ref{hyp-variance-control-o} and \ref{hyp-moment-4} hold with $\mathsf{\Lambda}=\mathsf{I}_p,$ $\boldsymbol{\alpha}= \boldsymbol{0}_p$ and $\boldsymbol{g}\equiv \boldsymbol{0}_p$,  since, for each $\boldsymbol{z}\in\Z_+^p$, we have
 \[
     \mathsf{\Gamma}(\boldsymbol{z})
       = \Big( \cov{U_{0,i}}{U_{0,j}} \1_{\{z_i > 0,\, z_j >0\}}\Big)_{i,j=1}^p
       = \Big( \frac{2}{3} \delta_{i,j} \1_{\{z_i > 0,\, z_j >0\}}\Big)_{i,j=1}^p,
 \]
 and 
 \[
 \kappa_i(\boldsymbol{z}) = \ev{U_{0,i}^4 \1_{\{z_i > 0\} } }
                          = \frac{2}{3}\1_{\{z_i > 0\} },
 \]
 yielding that $\Vert\mathsf{\Gamma}(\boldsymbol{z})\Vert\leq \frac{2}{3}$ and 
 $\kappa_i(\boldsymbol{z})\leq \frac{2}{3}$ for $\bz\in\Z_+^p$.
Moreover, taking into account parts \ref{rem-primitive} and \ref{rem-explosion} of Remark \ref{remark-hyp}, if $\mathsf{m}$ is primitive with Perron--Frobenius  eigenvalue 1, then Hypothesis \ref{hyp-primitive-relaxed} is also satisfied, and  Hypothesis \ref{hyp-explosion} is not needed.
\proofendb
\end{example}

Let $\mathbf{C}(\R_+,\R^p)$ be the space of $\R^p$--valued continuous functions on $\R_+$ and let us denote by $\mathbf{D}(\R_+,\R^p)$ the space of $\R^p$--valued càdlàg functions on $\R_+$ and by $\mathcal{D}_\infty (\R_+,\R^p)$ its Borel $\sigma$--algebra corresponding to the metric defined in equation (1.26) {in Jacod and Shiryaev} \cite[Chapter VI]{JACOD_SHIRYAEV_2003}. For $\R^p$--valued stochastic processes with càdlàg paths $(\boldsymbol{\mathcal{Y}}_t^{(n)})_{t\in\R_+}$, $n\in\N$, and $(\boldsymbol{\mathcal{Y}}_t)_{t\in\R_+}$, if the distribution of $(\boldsymbol{\mathcal{Y}}_t^{(n)})_{t\in\R_+}$ on the space $\left(\mathbf{D}(\R_+,\R^p),\mathcal{D}_\infty (\R_+,\R^p)\right)$ converges weakly as $n\to\infty$ to the distribution of  $(\boldsymbol{\mathcal{Y}}_t)_{t\in\R_+}$ on the same space, then we use the notation $(\boldsymbol{\mathcal{Y}}_t^{(n)})_{t\in\R_+} \stackrel{\mathcal{L}}{\longrightarrow} (\boldsymbol{\mathcal{Y}}_t)_{t\in\R_+}$ as $n\to\infty$.

Let us define a sequence of random step processes $(\boldsymbol{\mathcal{Z}}_t^{(n)})_{t\in\R_+}$, $n\in\N,$ using the 
 controlled $p$--type branching process $(\bZ_k)_{k\in\Z_+}$, as
\begin{equation*}
    \boldsymbol{\mathcal{Z}}_t^{(n)}  \colonequals  n^{-1} \boldsymbol{Z}_{\lfloor nt \rfloor}, \qquad t\in\R_+,\quad n\in\N.
\end{equation*}

Recall that, under Hypothesis \ref{hyp-primitive-relaxed}, 
 $\tilde{\boldsymbol{u}}, \tilde{\boldsymbol{v}}\in\R_+^p$ are
 the right and left eigenvectors of $\tilde{\mathsf{m}}$ corresponding to  
  eigenvalue $\tilde{\rho}=1$, respectively.

\begin{theorem}\label{Thm_main}
    Suppose that Hypotheses \ref{hyp-second-moment-Z0}--\ref{hyp-explosion} hold for the CMBP $(\boldsymbol{Z}_{k})_{k\in\Z_+}$ given in \eqref{CMBP_def}. 
    Then
    \begin{equation}\label{eq-convergence-main-thm}
        (\boldsymbol{\mathcal{Z}}_t^{(n)})_{t\in\R_+} 
            \stackrel{\mathcal{L}}{\longrightarrow} 
          (\mathcal{Z}_t\tilde{\boldsymbol{u}})_{t\in\R_+}
          \qquad \text{as } n\to\infty,
    \end{equation}
    where $(\mathcal{Z}_t)_{t\in\R_+}$ is the pathwise unique strong solution of the SDE
  \begin{equation}\label{eq-SDE-main-thm-1-dim}
    \dif\mathcal{Z}_t = \tilde{\boldsymbol{v}}^\top \mathsf{m}\boldsymbol{\alpha}\dif t + \sqrt{\tilde{\boldsymbol{v}}^\top ((\mathsf{\Lambda}\tilde{\boldsymbol{u}})\odot\boldsymbol{\mathsf{\Sigma}})\tilde{\boldsymbol{v}}\mathcal{Z}_t^+}\dif\mathcal{W}_t, \qquad t\in\R_+,
  \end{equation}
  with initial value 0, where $(\mathcal{W}_t)_{t\in\R_+}$ is a standard Wiener process.
\end{theorem}

\begin{remark}\label{rem-unique}
 The coefficient functions $\R \ni x\mapsto \tilde{\boldsymbol{v}}^\top \mathsf{m}\boldsymbol{\alpha} \in \R$ and $\R \ni x\mapsto \sqrt{\tilde{\boldsymbol{v}}^\top ((\mathsf{\Lambda}\tilde{\boldsymbol{u}})\odot\boldsymbol{\mathsf{\Sigma}})\tilde{\boldsymbol{v}}x^+} \in \R$ satisfy the conditions of Theorem 1 in Yamada and Watanabe \cite{Yamada_Watanabe_1971}, so pathwise uniqueness is guaranteed for SDE \eqref{eq-SDE-main-thm-1-dim}. Indeed, $\int_0^\epsilon u^{-1} \dif u = \infty$ for all $\epsilon>0$ and square root is a strictly increasing non-negative function vanishing at zero such that $|\sqrt{x}-\sqrt{y}|\leq \sqrt{|x-y|}$ for all $x,y\in\R_+$. Moreover, if $\tilde{\boldsymbol{v}}^\top \mathsf{m}\boldsymbol{\alpha}\in\R_+$ (respectively, $\tilde{\boldsymbol{v}}^\top \mathsf{m}\boldsymbol{\alpha}>0$), then $\mathcal{Z}_t$ is non-negative (respectively, positive) for all $t>0$ by the comparison theorem (see, e.g., Revuz and Yor \cite[Chapter IX, Theorem 3.7]{Revuz_Yor_1998}) and, consequently, $\mathcal{Z}_t^+$ in \eqref{eq-SDE-main-thm-1-dim} may be replaced by $\mathcal{Z}_t$.
Indeed, in case of $\tilde{\boldsymbol{v}}^\top ((\mathsf{\Lambda}\tilde{\boldsymbol{u}})\odot\boldsymbol{\mathsf{\Sigma}})\tilde{\boldsymbol{v}}=0$, we have $\mathcal{Z}_t=\tilde{\boldsymbol{v}}^\top \mathsf{m}\boldsymbol{\alpha}\,t$, $t\in\R_+$, and in case of
$\tilde{\boldsymbol{v}}^\top ((\mathsf{\Lambda}\tilde{\boldsymbol{u}})\odot\boldsymbol{\mathsf{\Sigma}})\tilde{\boldsymbol{v}}>0$ and $t>0$,
 we have $\mathcal{Z}_t$ has a Gamma distribution with parameters
 $2\tilde{\boldsymbol{v}}^\top \mathsf{m}\boldsymbol{\alpha}/(\tilde{\boldsymbol{v}}^\top ((\mathsf{\Lambda}\tilde{\boldsymbol{u}})\odot\boldsymbol{\mathsf{\Sigma}})\tilde{\boldsymbol{v}})$ and
 $2/(\tilde{\boldsymbol{v}}^\top ((\mathsf{\Lambda}\tilde{\boldsymbol{u}})\odot\boldsymbol{\mathsf{\Sigma}})\tilde{\boldsymbol{v}} \, t)$
 (following, e.g., from Ikeda and Watanabe \cite[Chapter IV, Example 8.2]{IkeWat}).
\proofendb
\end{remark}

\begin{remark}\label{rem-limit-process-multi}
 \hfill
\begin{enumerate}[label=(\roman*)]
  \item It turns out that the limit process in (\ref{eq-convergence-main-thm}) can also be characterized as the pathwise unique strong solution of the SDE
    \begin{equation}\label{SDE_limit_new}
        \dif \boldsymbol{\mathcal{Z}}_t = \tilde{\mathsf{\Pi}}\mathsf{m}\boldsymbol{\alpha}\dif t + \tilde{\mathsf{\Pi}}\sqrt{(\mathsf{\Lambda}\boldsymbol{\mathcal{Z}}_t)^+ \odot \boldsymbol{\mathsf{\Sigma}}} \dif \boldsymbol{\mathcal{W}}_t, \qquad t\in\R_+,
    \end{equation}
    with initial value $ \boldsymbol{0}_p$,
    where $(\boldsymbol{\mathcal{W}}_t)_{t\in\R_+}$ is a standard $p$--dimensional Wiener process (for details, see \nameref{step4} in the proof of Theorem \ref{Thm_main}).
  \item  If all the coordinates of the vector $\boldsymbol{\alpha}\in\R^p$ in Hypothesis \ref{hyp-g} are negative, then $\mathcal{Z}_t=0$, $t\in\R_+$ (where $(\mathcal{Z}_t)_{t\in\R_+}$ is the pathwise unique strong solution of the SDE \eqref{eq-SDE-main-thm-1-dim}).
  Indeed, by Remark \ref{rem-unique}, we have 
  $\mathrm{E}[{\mathcal{Z}_t}] = \tilde{\boldsymbol{v}}^\top \mathsf{m}\boldsymbol{\alpha} t \in (-\infty, 0]$ for all $t\in\R_+$.
  Since $\boldsymbol{\mathcal{Z}}_t^{(n)}\in\R_+^p$, $n\in\N$, $t\in\R_+$,
  the limit process $(\mathcal{Z}_t\tilde{\boldsymbol{u}})_{t\in\R_+}$
  in \eqref{eq-convergence-main-thm} is $\R_+^p$-valued, and, using that 
  $\tilde{\boldsymbol{u}}\in\R_+^p$, it implies that $\mathcal{Z}_t\in\R_+$, $t\in\R_+$. Consequently, if all the coordinates of the vector $\boldsymbol{\alpha}\in\R^p$ in Hypothesis \ref{hyp-g} are negative, then 
  $\tilde{\boldsymbol{v}}^\top \mathsf{m}\boldsymbol{\alpha} = 0$, 
  implying that $\mathcal{Z}_t=0$, $t\in\R_+$, as desired.
  \proofendb
 \end{enumerate} 
\end{remark}

In the  next remark, we point out that 
 Hypothesis \ref{hyp-explosion} basically excludes that
 $\prob{\boldsymbol{\phi}_0(\boldsymbol{0}_p)=\boldsymbol{0}_p}=1$, this condition appears
 for the conditional weak limit theorems 
 for one-dimensional distributions of critical CMBPs 
 in Gonz\'alez et al.\ \cite[Section 4]{gonzalez_martinez_mota_2006}. 
 Further, we also demonstrate that Hypothesis \ref{hyp-explosion} 
 is independent of Hypotheses \ref{hyp-second-moment-Z0}--\ref{hyp-primitive-relaxed}.

\begin{remark}\label{rem-abs-state}
\hfill
\begin{enumerate}[label=(\roman*)]
 \item If $\prob{\boldsymbol{\phi}_0(\boldsymbol{0}_p)=\boldsymbol{0}_p}=1$, then $\boldsymbol{\alpha}=-\boldsymbol{g}(\boldsymbol{0}_p)$ in Hypothesis \ref{hyp-g}.    
 Moreover, if $\prob{\boldsymbol{\phi}_0(\boldsymbol{0}_p)=\boldsymbol{0}_p}=1$ and $\boldsymbol{0}_p$ 
 is a state of $(\bZ_k)_{k\in\Z_+}$, i.e., there exists an $n_0\in\Z_+$ such that $\prob{\bZ_{n_0}=\boldsymbol{0}_p}>0$,  then Hypothesis \ref{hyp-explosion} does not hold.
Indeed, in this case, we have $\{\bZ_{n_0}=\boldsymbol{0}_p
 \}\subseteq \{ \bZ_k=\boldsymbol{0}_p\}$ for each 
 $k\geq n_0$, $k\in\Z_+$, yielding that $\{\bZ_{n_0}=\boldsymbol{0}_p\}\subseteq \{\Vert \bZ_k \Vert \leq B\}$
 for all $B>0$ and $k\geq n_0$, $k\in\Z_+$.
Consequently, we get $0<\prob{\bZ_{n_0}=\boldsymbol{0}_p} 
 \leq\prob{\Vert \bZ_k \Vert \leq B}$ 
 for all $B>0$ and $k\geq n_0$, $k\in\Z_+$.
By choosing $\epsilon:=\frac{1}{2}\prob{\bZ_{n_0}=
 \boldsymbol{0}_p}>0$, one can see that
 Hypothesis \ref{hyp-explosion} does not hold.
If $\prob{\boldsymbol{\phi}_0(\boldsymbol{0}_p)=
 \boldsymbol{0}_p}=1$ and $\boldsymbol{0}_p$ 
 is not a state of $(\bZ_k)_{k\in\Z_+}$, then, in principle, Hypothesis \ref{hyp-explosion} can hold.
For example, let $p\colonequals 2$, 
    \begin{align*}
   & \boldsymbol{\phi}_{k}(\boldsymbol{z})  
    \equalscolon
    \begin{cases}
       \boldsymbol{0}_2 & \text{if $\bz=\boldsymbol{0}_2$,}\\
       \begin{pmatrix}
        z_1+1  \\
        z_1+z_2-1
    \end{pmatrix} & \text{if $\bz\ne \boldsymbol{0}_2$, $\bz\in\Z_+^2$,}
    \end{cases} \qquad  k\in\Z_+,\\[1mm]
&\boldsymbol{X}_{k,j,1}  \equalscolon
    \begin{pmatrix}
        1 \\
        0 
    \end{pmatrix},
    \quad 
    \boldsymbol{X}_{k,j,2}  
    \equalscolon 
    \begin{pmatrix}
        0 \\
        0
    \end{pmatrix},
    \quad 
    \bZ_0  \equalscolon \begin{pmatrix}
        1 \\
        0
    \end{pmatrix},\qquad  k\in\Z_+,\,j\in\N.
 \end{align*}
Then $\prob{\boldsymbol{\phi}_0(\boldsymbol{0}_2)=\boldsymbol{0}_2}=1$ and
 \[
   \bZ_k=(k+1)
       \begin{pmatrix}
           1 \\
           0
        \end{pmatrix},\qquad k\in\N.
  \] 
 Hence one can easily see that Hypotheses \ref{hyp-second-moment-Z0}--\ref{hyp-explosion} hold with 
  \[
     \mathsf{\Lambda}\colonequals \begin{pmatrix}
                         1 & 0 \\
                          1 & 1 \\
                       \end{pmatrix} ,
    \qquad 
    \boldsymbol{\alpha}\colonequals \begin{pmatrix}
                             1 \\
                              -1
                          \end{pmatrix},
   \qquad 
     \boldsymbol{g}(\boldsymbol{z})\colonequals                             
         \begin{cases}
             \begin{pmatrix}
              -1 \\
               1
              \end{pmatrix}  & \text{if $\bz = \boldsymbol{0}_2$,}\\
            \boldsymbol{0}_2   & \text{if $\bz \ne \boldsymbol{0}_2$, $\bz\in\Z_+^2$,}    
         \end{cases}
  \]
  and 
  \[
  \mathsf{m} = \tilde{\mathsf{m}} \colonequals \begin{pmatrix}
                         1 & 0 \\
                         0 & 0 \\
                       \end{pmatrix} ,
   \qquad   
   \Tilde{\boldsymbol{u}}\colonequals\begin{pmatrix}
                             1 \\
                             0 \\
                           \end{pmatrix},\qquad 
     \Tilde{\boldsymbol{v}}\colonequals\begin{pmatrix}
                             1 \\
                             0 \\
                           \end{pmatrix}.                           
  \] 
  In this case, without the application of Theorem \ref{Thm_main}, we readily have that
  \[
   (\boldsymbol{\mathcal{Z}}_t^{(n)})_{t\in\R_+}
     = \left(\frac{1}{n}(\lfloor nt\rfloor +1)
                          \begin{pmatrix}
                             1 \\
                             0 \\
                           \end{pmatrix} \right)_{t\in\R_+}
  \stackrel{\mathcal{L}}{\longrightarrow}
    \left(t \begin{pmatrix}
                             1 \\
                             0 \\
            \end{pmatrix}
      \right)_{t\in\R_+} 
      \qquad \text{as } n\to\infty.
  \]
Of course, Theorem \ref{Thm_main} also gives the same result, since in this case the SDE \eqref{eq-SDE-main-thm-1-dim} takes
 the form $\dif\mathcal{Z}_t = \dif t$, $t\in\R_+$, with initial value 0, yielding that $\mathcal{Z}_t = t$, $t\in\R_+$. 

\item   We give a nontrivial example, where Hypotheses \ref{hyp-second-moment-Z0}--\ref{hyp-primitive-relaxed} hold, but  
 Hypothesis \ref{hyp-explosion} does not.
Let $p\colonequals 2$, 
    \begin{align*}
   & \boldsymbol{\phi}_{k}(\boldsymbol{z})  
    \equalscolon
    \begin{cases}
       \boldsymbol{0}_2 & \text{if $\bz=\boldsymbol{0}_2$,}\\
       \begin{pmatrix}
        z_1+1  \\
        z_2+1
    \end{pmatrix} & \text{if $\bz\ne \boldsymbol{0}_2$, $\bz\in\Z_+^2$,}
    \end{cases} \qquad  k\in\Z_+,\\[1mm]
&\boldsymbol{X}_{k,j,i}  \equalscolon
   \begin{cases}
    \begin{pmatrix}
        1 \\
        1
    \end{pmatrix} & \text{with probability $\frac{1}{2}$,}\\
     \boldsymbol{0}_2 & \text{with probability $\frac{1}{2}$},
 \end{cases}
    \qquad\quad 
    \bZ_0  \equalscolon \begin{pmatrix}
        1 \\
        1
    \end{pmatrix},\qquad k,j\in\N, \quad i\in\{1,2\}.
 \end{align*}
Then 
 \[
  \boldsymbol{m}_i=\ev{\boldsymbol{X}_{0,1,i}} 
    = \frac{1}{2} \begin{pmatrix}
        1 \\
        1 
    \end{pmatrix}, \qquad i\in\{1,2\},
    \qquad \text{yielding that}\qquad
  \mathsf{m} = \frac{1}{2} \begin{pmatrix}
        1 & 1 \\
        1 & 1
    \end{pmatrix}.
 \]
Further, let 
 \[
     \mathsf{\Lambda}\colonequals \begin{pmatrix}
                         1 & 0 \\
                         0 & 1 \\
                       \end{pmatrix} ,
    \qquad 
    \boldsymbol{\alpha}\colonequals \begin{pmatrix}
                             1 \\
                             1
                          \end{pmatrix},
   \qquad 
     \boldsymbol{g}(\boldsymbol{z})\colonequals                            
         \begin{cases}
             \begin{pmatrix}
              -1 \\
              -1
              \end{pmatrix}  & \text{if $\bz = \boldsymbol{0}_2$,}\\
            \boldsymbol{0}_2   & \text{if $\bz \ne \boldsymbol{0}_2$, $\bz\in\Z_+^2$.}    
         \end{cases}
  \]
Hence 
  \[
   \tilde{\mathsf{m}} 
    = \mathsf{m} \mathsf{\Lambda}
    =  \frac{1}{2}
                       \begin{pmatrix}
                         1 & 1 \\
                         1 & 1 \\
                       \end{pmatrix} ,
   \]
  which is a primitive matrix with Perron--Frobenius eigenvalue 
  $\tilde\rho = \max(0,1)=1$.
Consequently, Hypotheses \ref{hyp-second-moment-Z0}--\ref{hyp-primitive-relaxed} hold with the above introduced  $\mathsf{\Lambda}$, $\boldsymbol{\alpha}$, $\boldsymbol{g}$, $\tilde{\mathsf{m}}$, and
\[
   \Tilde{\boldsymbol{u}}\colonequals
                \frac{1}{\sqrt{2}} \begin{pmatrix}
                             1 \\
                             1 \\
                           \end{pmatrix},\qquad 
     \Tilde{\boldsymbol{v}}\colonequals
                 \frac{1}{\sqrt{2}} \begin{pmatrix}
                             1 \\
                             1 \\
                           \end{pmatrix}.                           
  \] 
Further, since $\boldsymbol{0}_2$ is a state of $(\bZ_k)_{k\in\Z_+}$
 and $\prob{\boldsymbol{\phi}_0(\boldsymbol{0}_2)=\boldsymbol{0}_2}=1$,
 part (i) of this remark implies that Hypothesis \ref{hyp-explosion} does not hold.
\proofendb
\end{enumerate}
\end{remark}

 In the next remark, we point out that Theorem  \ref{Thm_main} 
 with $p=1$ gives back the result of Gonz\'alez et al.\ \cite{GONZALEZ_MARTIN-CHAVEZ_DEL-PUERTO_2022}  for critical single-type CBPs.

\begin{remark}\label{Rem_1CBP_spec}
Theorem \ref{Thm_main} with $p=1$, i.e. for critical single-type CBPs, yields that $(\mathcal{Z}_t^{(n)})_{t\in\R_+}$ converges weakly toward $(\mathcal{Z}_t)_{t\in\R_+}$ as $n\to\infty$, where
 $(\mathcal{Z}_t)_{t\in\R_+}$ is the pathwise unique strong solution of the SDE 
 $$
  \dif \mathcal{Z}_t = m \alpha \dif t+\sqrt{m^{-1} \Sigma
   \mathcal{Z}_t^+}\dif \mathcal{W}_t,
   \qquad t\in\R_+,
 $$
 with initial value $0$, where $(\mathcal{W}_t)_{t\in\R_+}$ is a standard Wiener process, and $m>0$ and $\Sigma$ 
 are the offspring mean and variance, respectively 
 ($\Lambda$ must be $m^{-1}$ in order that 
 $(Z_k)_{k\in\Z_+}$ be critical). 
 It is worth mentioning that Gonz\'alez et al.\ \cite[Theorem 3.1]{GONZALEZ_MARTIN-CHAVEZ_DEL-PUERTO_2022} obtained 
  the same limit process  $(\mathcal{Z}_t)_{t\in\R_+}$
  assuming the following three hypotheses, namely,  
  (C0): the explosion set has probability one,
  (C1): $\varepsilon(z)=m^{-1}z+m^{-1}\alpha'+\oo(1)$ as $z\to\infty$ with some $\alpha'>0$, and 
  (C2): $\Gamma(z)=\oo(z)$  as $z\to\infty$. 
Notice that  if (C0), (C1) and (C2) hold, then
 our Hypotheses 2, 3, 5 and 6  are satisfied.
 Indeed, Hypotheses 2, 3 and 5  follow trivially from 
 (C1) and (C2),  and (C0) implies Hypothesis 6 as it was pointed out in part (iv) in Remark \ref{remark-hyp}.
 Concerning Hypotheses 1 and 4,  in part (i) of Remark   \ref{remark-hyp}, we already mentioned the reason why  the fourth moments on the offspring and control distributions are 
   assumed in our paper. 
\proofendb
\end{remark}

As a consequence of Theorem \ref{Thm_main} one can deduce the asymptotic behaviour of a critical primitive MBPI previously proved by Isp\'any and Pap \cite[Theorem 3.1]{ispany-pap-2014}.  
Let $(\boldsymbol{Y}_k)_{k\in\Z_+}$ be a MBPI defined in  
 \eqref{def_MBPI}  such that the offspring mean matrix $\mathsf{m}_{\boldsymbol{\xi}} \colonequals 
  (\ev{\boldsymbol{\xi}_{0,1,1}}, \ldots, \ev{\boldsymbol{\xi}_{0,1,p}} )
\in\R_+^{p\times p}$ is primitive with Perron--Frobenius eigenvalue 1, and let $\boldsymbol{u}$ 
 and $\boldsymbol{v}$  be its right and left  Perron--Frobenius eigenvectors, respectively (see Lemma \ref{lemma-primitive}). 
 Then $(\boldsymbol{Y}_k)_{k\in\Z_+}$ is a critical primitive MBPI. 
Let us denote $\boldsymbol{m}_{\boldsymbol{I}}  \colonequals \ev{\boldsymbol{I}_1}\in\R_+^p$ and $\boldsymbol{\mathsf{V}}  \colonequals (\var{\boldsymbol{\xi}_{0,1,1}},\ldots,\var{\boldsymbol{\xi}_{0,1,p}})\in (\R^{p\times p})^p$.

\begin{corollary}[Ispány and Pap {\cite[Theorem 3.1]{ispany-pap-2014}}]\label{cor-MBPI}
    Let $(\boldsymbol{Y}_k)_{k\in\Z_+}$ be  a critical primitive $p$--type branching process with immigration
    such that $\ev{\|\boldsymbol{Y}_0\|^2}<\infty$, $\ev{\|\boldsymbol{\xi}_{0,1,i}\|^4}<\infty$, $i\in\{1,\ldots,p\}$, and $\ev{\|\boldsymbol{I}_{1}\|^4}<\infty$. Then
    \begin{equation*}
        (n^{-1}\boldsymbol{Y}_{\lfloor nt \rfloor})_{t\in\R_+} 
        \stackrel{\mathcal{L}}{\longrightarrow}
        (\mathcal{Y}_t \boldsymbol{u})_{t\in\R_+}
        \qquad \text{as } n\to\infty ,
    \end{equation*}
    where $(\mathcal{Y}_t)_{t\in\R_+}$ is the pathwise unique strong solution of the SDE
    \begin{equation*}
        \dif\mathcal{Y}_t = \boldsymbol{v}^\top \boldsymbol{m}_{\boldsymbol{I}} \dif t + \sqrt{\boldsymbol{v}^\top (\boldsymbol{u}\odot\boldsymbol{\mathsf{V}})\boldsymbol{v}\mathcal{Y}_t^+}\dif\mathcal{W}_t, \qquad t\in\R_+,
    \end{equation*}
    with initial value $\mathcal{Y}_0 = 0$, where $(\mathcal{W}_t)_{t\in\R_+}$ is a standard Wiener  process.
\end{corollary}

\begin{proof}
  To apply Theorem \ref{Thm_main},  let us rewrite $(\boldsymbol{Y}_k)_{k\in\Z_+}$ as a CMBP  (see part (ii) of Example \ref{Example1}). 
  For $k\in\Z_+$, $j\in \N$ and $i\in\{1,\ldots,p\}$, we have
  \begin{equation*}
     \boldsymbol{Z}_{k} =
        \begin{pmatrix}
            \boldsymbol{Y}_{k} \\
            {1}
        \end{pmatrix},
        \qquad
        \boldsymbol{X}_{k,j,i} =
        \begin{pmatrix}
            \boldsymbol{\xi}_{k,j,i} \\
            0
        \end{pmatrix},  \qquad
        {\boldsymbol{X}_{k,1,p+1} } =
        \begin{pmatrix}
            \boldsymbol{I}_{k+1} \\
            1
        \end{pmatrix},\qquad
         \boldsymbol{\phi}_{k} (\boldsymbol{z}) =
        \begin{pmatrix}
            z_1 \\
            \vdots \\
            z_p \\
            1
        \end{pmatrix}
    \end{equation*}
    for each $\bz\in\Z_+^{p+1}$.  Further, the deterministic control functions  can be written in the form  $\boldsymbol{\phi}_{k} (\boldsymbol{z}) = \mathsf{\Lambda}\boldsymbol{z} + \boldsymbol{\alpha} + \boldsymbol{g} (\boldsymbol{z})$, 
     $\bz\in\Z_+^{p+1}$, $k\in\Z_+$, with
    \begin{equation*}
        \mathsf{\Lambda} = \begin{pmatrix}
            \mathsf{I}_p & \boldsymbol{0}_{p} \\
            \boldsymbol{0}_{p}^\top & 0
        \end{pmatrix}, \qquad \boldsymbol{\alpha} = \begin{pmatrix}
            \boldsymbol{0}_{p} \\
            1
        \end{pmatrix}
        \qquad \mbox{ and }\qquad\boldsymbol{g}\equiv\boldsymbol{0}_{p+1}.
    \end{equation*}
  Then we have
    \begin{align*}
        \ev{\|\bZ_0\|^2} &= \ev{ 1+\|\boldsymbol{Y}_0\|^2} < \infty, \\
        \ev{\|\boldsymbol{X}_{0,1,i}\|^4} &= \ev{\|\boldsymbol{\xi}_{0,1,i}\|^4} < \infty, \qquad i\in\{1,\ldots,p\} ,\\
        \ev{\|\boldsymbol{X}_{0,1,p+1}\|^4} &= \ev{(1+\|\boldsymbol{I}_{1}\|^2)^2} < \infty,\\
       \ev{\|\boldsymbol{\phi}_0(\bz)\|^4} &  = \|\boldsymbol{\phi}_0(\bz)\|^4 < \infty, \qquad \bz\in\Z_+^{p+1}.  
    \end{align*}
    Thus  $\mathsf{\Gamma}(\boldsymbol{z})=\mathsf{0} 
    \in\R^{(p+1)\times (p+1)}$
     and $\kappa_i(\bz)=0$, $\bz\in \Z_+^{p+1}$, 
     $i\in\{1,\ldots,p\}$, 
     and hence Hypotheses \ref{hyp-second-moment-Z0}, \ref{hyp-g}, \ref{hyp-variance-control-o} and \ref{hyp-moment-4} hold. 
    Since $\boldsymbol{g} \equiv \boldsymbol{0}_{p+1}$, Hypothesis \ref{hyp-explosion} is not needed, see part \ref{rem-explosion} of Remark \ref{remark-hyp}.
    Finally, Hypothesis \ref{hyp-primitive-relaxed} is  satisfied with
    \begin{equation*}
       \mathsf{m} = \begin{pmatrix}
            \mathsf{m}_{\boldsymbol{\xi}} & \boldsymbol{m}_{\boldsymbol{I}} \\
            \boldsymbol{0}_{p}^\top &  1
        \end{pmatrix}, \qquad
        \tilde{\mathsf{m}} = \begin{pmatrix}
            \mathsf{m}_{\boldsymbol{\xi}} & \boldsymbol{0}_{p} \\
            \boldsymbol{0}_{p}^\top & 0
        \end{pmatrix}, \qquad 
        \Tilde{\boldsymbol{u}} = \begin{pmatrix}
            \boldsymbol{u} \\
            0
        \end{pmatrix}, \qquad \Tilde{\boldsymbol{v}} = \begin{pmatrix}
            \boldsymbol{v} \\
            0
        \end{pmatrix},
        \qquad 
        \tilde{\mathsf{\Pi}} 
         =  \begin{pmatrix}
             \mathsf{\Pi}_{\boldsymbol{\xi}}   & \boldsymbol{0}_{p} \\
            \boldsymbol{0}_{p}^\top & 0
        \end{pmatrix},
    \end{equation*}
      where $ \mathsf{\Pi}_{\boldsymbol{\xi}} 
         \colonequals\lim_{k\to\infty}\mathsf{m}_{\boldsymbol{\xi}}^k$, 
    and $\tilde{c}$ and $\tilde{r}$ in part (c) of Hypothesis \ref{hyp-primitive-relaxed}
    can be chosen as for $ \mathsf{\Pi}_{\boldsymbol{\xi}}$ 
    (since $\mathsf{m}_{\boldsymbol{\xi}}$ is primitive with 
    Frobenius--Perron eigenvalue $1$, and see part (iii) of Lemma \ref{lemma-primitive}).
 Further, we get
    \begin{align*}
        \mathsf{\Sigma}_i &= \begin{pmatrix}
            \var{\boldsymbol{\xi}_{0,1,i}} & \boldsymbol{0}_{p} \\
            \boldsymbol{0}_{p}^\top & 0
        \end{pmatrix}, \qquad i\in\{1,\ldots,p\}, &
        \mathsf{\Sigma}_{p+1} &= \begin{pmatrix}
            \var{\boldsymbol{I}_{1}} & \boldsymbol{0}_{p} \\
            \boldsymbol{0}_{p}^\top & 0
        \end{pmatrix},
    \end{align*}
    and, consequently, Theorem \ref{Thm_main}  yields the statement.
\end{proof}

 Motivated by part (i) of Remark \ref{rem-interpretation}, we derive another corollary of Theorem \ref{Thm_main}.

\begin{corollary}\label{cor-MBPI-2}
 Let $(\boldsymbol{Z}_k)_{k\in\Z_+}$ be a  critical controlled $p$--type branching 
  process given in part (i) of Remark \ref{rem-interpretation}
  such that $\ev{\|\boldsymbol{Y}_0\|^2}<\infty$, 
    $\ev{\|\boldsymbol{\xi}_{0,1,i}\|^4}<\infty$, $i\in\{1,\ldots,p\}$, and $\ev{\|\boldsymbol{I}_{1}\|^4}<\infty$.
  Then
    \begin{equation*}
        (n^{-1}\boldsymbol{Z}_{\lfloor nt \rfloor})_{t\in\R_+} 
        \stackrel{\mathcal{L}}{\longrightarrow}
        (\mathcal{Z}_t \boldsymbol{u})_{t\in\R_+}
        \qquad \text{as } n\to\infty ,
    \end{equation*}
    where $(\mathcal{Z}_t)_{t\in\R_+}$ is the pathwise unique strong solution of the SDE
    \begin{equation*}
        \dif\mathcal{Z}_t =  \boldsymbol{v}^\top  \boldsymbol{m}_{\boldsymbol{I}}
    \dif t + \sqrt{\boldsymbol{v}^\top (\boldsymbol{u}\odot\boldsymbol{\mathsf{V}})\boldsymbol{v}\mathcal{Z}_t^+}\dif\mathcal{W}_t, \qquad t\in\R_+,
    \end{equation*}
    with initial value $\mathcal{Z}_0 = 0$, where $(\mathcal{W}_t)_{t\in\R_+}$ is a standard Wiener process.
\end{corollary}

\begin{proof}
  We follow the same steps as in the proof of
   Corollary \ref{cor-MBPI}. We have
    \begin{align*}
        \ev{\|\bZ_0\|^2} &= \ev{\|\boldsymbol{Y}_0\|^2} < \infty, \\
        \ev{\|\boldsymbol{X}_{0,1,i}\|^4} &= \ev{\|\boldsymbol{\xi}_{0,1,i}\|^4} < \infty, \qquad i\in\{1,\ldots,p\} \\
        \ev{\|\boldsymbol{\phi}_0(\bz)\|^4} &\leq 8 \ev{\|\bz\|^4 + \|\boldsymbol{I}_{1}\|^4} < \infty, 
    \end{align*}
    where, for the last inequality, we used
    the power mean inequality (see \eqref{mean_power_ineq}). 
    Therefore, Hypothesis \ref{hyp-second-moment-Z0} is satisfied. Hypothesis \ref{hyp-g} also holds because $\boldsymbol{\varepsilon}(\boldsymbol{z}) = \mathsf{\Lambda}\boldsymbol{z} + \boldsymbol{\alpha} + \boldsymbol{g}(\boldsymbol{z})$, $\bz\in\Z_+^p$, with $\mathsf{\Lambda} = \mathsf{I}_p$, $\boldsymbol{\alpha} = \boldsymbol{m}_{\boldsymbol{I}}$ and $\boldsymbol{g}\equiv\boldsymbol{0}_{p}$. 
    Due to $\ev{\|\boldsymbol{I}_{1}\|^4} <\infty$, it is clear that $\mathsf{\Gamma}(\boldsymbol{z}) = \var{\boldsymbol{I}_1}  
    \in\R^{p\times p}$ and $\kappa_i (\bz) = \ev{(I_{1,i}-m_{\boldsymbol{I},i})^4} < \infty$, $\bz\in\Z_+^p$, $i\in\{1,\ldots,p\}$, 
    so Hypotheses \ref{hyp-variance-control-o} and \ref{hyp-moment-4} 
    are satisfied.
    Furthermore, Hypothesis \ref{hyp-primitive-relaxed} holds with
    \begin{equation*}
       \mathsf{m} = \mathsf{m}_{\boldsymbol{\xi}}, \qquad
        \tilde{\mathsf{m}} =             \mathsf{m}_{\boldsymbol{\xi}}, \qquad 
        \Tilde{\boldsymbol{u}} =\boldsymbol{u}, \qquad \Tilde{\boldsymbol{v}} =             \boldsymbol{v},
        \qquad 
        \tilde{\mathsf{\Pi}} 
         =   \mathsf{\Pi}_{\boldsymbol{\xi},}
    \end{equation*}
     where $\mathsf{\Pi}_{\boldsymbol{\xi}} 
         := \lim_{k\to\infty}\mathsf{m}_{\boldsymbol{\xi}}^k$, 
    and $\tilde{c}$ and $\tilde{r}$ in part (c) of Hypothesis \ref{hyp-primitive-relaxed} can be chosen as for $\mathsf{\Pi}_{\boldsymbol{\xi}}$ 
    (see part (iii) of Lemma \ref{lemma-primitive}).    
    Indeed, by our assumption,  $\mathsf{m}_{\boldsymbol{\xi}}$ is primitive with Perron--Frobenius eigenvalue equals 1. Finally, Hypothesis \ref{hyp-explosion} is not needed, since $\boldsymbol{g}\equiv\boldsymbol{0}_p$ (see part \ref{rem-explosion} 
     of Remark \ref{remark-hyp}). 
    Hence the result follows from Theorem \ref{Thm_main}  using also that $\boldsymbol{v}^\top  \mathsf{m}_{\boldsymbol{\xi}}  \boldsymbol{m}_{\boldsymbol{I}}
    = \boldsymbol{v}^\top \boldsymbol{m}_{\boldsymbol{I}}$, since $\boldsymbol{v}$
    is the left Perron--Frobenius eigenvector for 
    $\mathsf{m}_{\boldsymbol{\xi}}$.
\end{proof}

Next, we establish a scaling limit theorem for a 2SBPI with the promiscuous mating function $L(x,y) \colonequals x\min\{1,y\}$, $x,y\in\Z_+$. This type of mating function was first considered by Daley \cite{DALEY_1968}, and it was deeply studied by Alsmeyer and R\"osler \cite{ALSMEYER_1996, ALSMETER_2002}.
 Up to our knowledge, no such result as the following one is available in the literature  (for more details, see Remark \ref{Rem_Bansaye}).

\begin{corollary}\label{cor-2SBPI}
 Let $(F_k,M_k)_{k\in\Z_+}$ be a 2SBPI defined in \eqref{eq-definition}  with the promiscuous mating function. 
 Assume that $(F_0,M_0)$ is a $\N^2$--valued random vector, 
 $\ev{\|(F_0,M_0)\|^2}<\infty$,  $\ev{\|(f_{0,1},m_{0,1})\|^4}<\infty$, and 
  $\ev{\|(F_1^I,M_1^I)\|^4}<\infty$.
   If $\prob{M_1^I=0}=0$   
   and $\ev{f_{0,1}}=1$, then 
    \begin{equation*}
        (n^{-1} (F_{\lfloor nt \rfloor}, M_{\lfloor nt \rfloor}))_{t \in\R_+}
        \stackrel{\mathcal{L}}{\longrightarrow} 
        (\mathcal{X}_t (1,\ev{m_{0,1}}))_{t \in\R_+} \qquad \text{as } n\to\infty, 
    \end{equation*}
    where $(\mathcal{X}_t)_{t\in\R_+}$ is the pathwise unique strong solution of the SDE 
    \begin{equation*}
        \dif\mathcal{X}_t = \ev{F_1^I}\dif t 
             + \sqrt{ \var{f_{0,1}} \mathcal{X}_t^+} \dif\mathcal{W}_t, 
             \qquad t \in\R_+,
    \end{equation*}
    with initial value $\mathcal{X}_0 = 0$, where $(\mathcal{W}_t)_{t\in\R_+}$ is a standard Wiener process.
\end{corollary}

\begin{proof}
Recall that $L(x,y) = x\min\{1,y\}$, $x,y\in\Z_+$.

First, using that $\prob{M_1^I=0}=0$ and $(F_0,M_0)$ is $\N^2$--valued, 
   we check that $L(F_n,M_n) = F_n$ for each $n\in\Z_+$ almost surely.
For this, since $M_n$ is $\Z_+$--valued, taking into account the form of $L$, it is enough to verify that 
 $\prob{M_n=0}=0$, $n\in\N$. Using that $\prob{M_1^I=0}=0$, 
  by a conditioning argument with respect to $U_n$, 
we can get for each $n\in\Z_+$, 
         \begin{align*}
       \prob{M_{n+1}=0}
        & = \ev{\1_{\{M_{n+1}=0\}}}
         = \ev{ \evcond{\1_{\{M_{n+1}=0\}}}{U_n}}
         = \ev{ \evcond{\1_{\{ \sum_{j=1}^{U_n}m_{n,j} + M_{n+1}^I=0\}}}{U_n}} \\
        & = \ev{ \evcond{
          \1_{\{ M_{n+1}^I=0\}}
          \prod_{j=1}^{U_n}\1_{\{ m_{n,j} =0\}}}{U_n}} = \ev{ \prob{M_1^I=0} (\prob{m_{0,1} =0})^{U_n} } \\
         &= \prob{M_1^I=0} \ev{(\prob{m_{0,1} =0})^{U_n}}
         = 0,
        \end{align*} 
   as desired.
  
 Consequently, $(F_k,M_k)_{k\in\Z_+}$ coincides almost surely with the  
  2-type controlled branching process $(\bZ_k)_{k\in\Z_+}$ given by
 \begin{align*}
     \boldsymbol{Z}_{k}  \colonequals 
        \begin{pmatrix}
            F_{k} \\
            M_{k}
        \end{pmatrix}, \qquad
        \boldsymbol{\phi}_{k} (\boldsymbol{z})  \colonequals 
        \begin{pmatrix}
            z_1 \\
            1 
        \end{pmatrix}, \qquad
        \boldsymbol{X}_{k,j,1}  \colonequals 
        \begin{pmatrix}
            f_{k,j} \\
            m_{k,j}
        \end{pmatrix}, \qquad
        \boldsymbol{X}_{k,j,2}  \colonequals 
        \begin{pmatrix}
            F_{k+1}^{I} \\[1mm]
            M_{k+1}^{I}
        \end{pmatrix} 
    \end{align*}
    for $k\in\Z_+$, $j\in\N$, and $\bz=(z_1,z_2)\in\Z_+^2$. 
    
We want to apply Theorem \ref{Thm_main} for $(\bZ_k)_{k\in\Z_+}$.
The control distributions are deterministic, so Hypothesis \ref{hyp-variance-control-o} and \ref{hyp-moment-4} are trivially satisfied. In addition,  we get
    \begin{align*}
        \ev{\|\bZ_0\|^2} &= \ev{\|(F_0,M_0)\|^2} < \infty, \\
        \ev{\|\boldsymbol{X}_{0,1,1}\|^4} &= \ev{\|(f_{0,1},m_{0,1})\|^4} < \infty, \\
        \ev{\|\boldsymbol{X}_{0,1,2}\|^4} &= \ev{\|(F_1^I,M_1^I)\|^4} < \infty, \\
        \ev{\|\boldsymbol{\phi}_0(\bz)\|^4} &=  (z_1^2+1)^2 < \infty, \qquad \bz\in\Z_+^2 ,
    \end{align*}
     yielding that Hypothesis \ref{hyp-second-moment-Z0} holds.
For each $\bz\in\Z_+^2$, we have
    \begin{equation*}
             \boldsymbol{\phi}_0 (\boldsymbol{z}) 
             = \begin{pmatrix}
                z_1 \\
                 1 
             \end{pmatrix} 
            = \mathsf{\Lambda}\boldsymbol{z} + \boldsymbol{\alpha} \qquad \text{ with }\quad 
            \mathsf{\Lambda} = \begin{pmatrix}
                1 & 0 \\
                0 & 0
            \end{pmatrix} \quad\mbox{ and }\quad  \boldsymbol{\alpha} =
            \begin{pmatrix}
                0 \\
                1
            \end{pmatrix},
        \end{equation*}       
 and hence Hypothesis \ref{hyp-g} holds with 
        the given $\mathsf{\Lambda}$, $\boldsymbol{\alpha}$, and
        $\boldsymbol{g}\equiv  \boldsymbol{0}_2$. 
    Hypothesis \ref{hyp-explosion} is not needed, see part \ref{rem-explosion} of Remark \ref{remark-hyp}.
          Further, we get
        \begin{align*}
            \mathsf{m} &= \begin{pmatrix}
                \ev{f_{0,1}} & \ev{F_1^I} \\[1mm] 
                \ev{m_{0,1}} & \ev{M_1^I}
            \end{pmatrix}, & \mathsf{\Sigma}_1 &= \begin{pmatrix}
                \var{f_{0,1}} & \cov{f_{0,1}}{m_{0,1}} \\
                \cov{f_{0,1}}{m_{0,1}} & \var{m_{0,1}}
            \end{pmatrix}, \\
            \tilde{\mathsf{m}} &= \begin{pmatrix}
                \ev{f_{0,1}} & 0 \\
                \ev{m_{0,1}} & 0
            \end{pmatrix}, & \mathsf{\Sigma}_2 &= \begin{pmatrix}
                \var{F_1^I} & \cov{F_1^I}{M_1^I} \\[1mm]
                \cov{F_1^I}{M_1^I} & \var{M_1^I}
            \end{pmatrix}.
        \end{align*}
         Since $\ev{f_{0,1}}=1$  and $\tilde{\mathsf{m}}^k = \tilde{\mathsf{m}}$, $k\in\N$, the conditions of Hypothesis \ref{hyp-primitive-relaxed} are satisfied with
        \begin{align*}
        &\tilde{\boldsymbol{u}} = \frac{1}{1+\ev{m_{0,1}}} \begin{pmatrix}
                1 \\
                \ev{m_{0,1}}
            \end{pmatrix}, \qquad \tilde{\boldsymbol{v}} = \begin{pmatrix}
                1+\ev{m_{0,1}} \\
                0
            \end{pmatrix}, \qquad
            \mathsf{\Lambda}\tilde{\boldsymbol{u}} = \frac{1}{1+\ev{m_{0,1}}} \begin{pmatrix}
                1 \\
                0
            \end{pmatrix},\\[1mm]
         &\tilde{\mathsf{\Pi}} =  \lim_{k\to\infty} \tilde{\mathsf{m}}^k  
             = \tilde{\mathsf{m}}.
        \end{align*}
        Therefore, Theorem \ref{Thm_main}  implies that
        \begin{equation*}
            \left(\frac{1}{n} \begin{pmatrix}
                F_{\lfloor nt \rfloor} \\
                M_{\lfloor nt \rfloor}
            \end{pmatrix}\right)_{t\in\R_+}
            \stackrel{\mathcal{L}}{\longrightarrow} 
          \left(\frac{\mathcal{Z}_t}{1+\ev{m_{0,1}}} \begin{pmatrix}
                1 \\
                \ev{m_{0,1}}
            \end{pmatrix}\right)_{t\in\R_+} \qquad \text{as } n\to \infty,
        \end{equation*}
        where $(\mathcal{Z}_t)_{t\in\R_+}$ is the pathwise unique strong solution of the SDE
        \begin{equation*}
            \dif\mathcal{Z}_t = (1+\ev{m_{0,1}})\ev{F_1^I} \dif t + \sqrt{(1+\ev{m_{0,1}}) \var{f_{0,1}}\mathcal{Z}_t^+}\dif\mathcal{W}_t, \qquad t\in\R_+,
        \end{equation*}
        with initial value $\mathcal{Z}_0 = 0$, where $(\mathcal{W}_t)_{t\in\R_+}$ is a standard Wiener process. With the process $(\mathcal{X}_t)_{t\in\R_+}$ given by $\mathcal{X}_{t} = (1+\ev{m_{0,1}})^{-1} \mathcal{Z}_{t}$,  $t\in\R_+$, 
        we obtain the desired result.
\end{proof}

Next, we apply Theorem \ref{Thm_main} for a 2SBPI defined in \eqref{eq-definition} with
 a so-called self-fertilization mating function defined by $L(\bz):=z_1+z_2$, $\bz=(z_1,z_2)\in\R_+^2$,
 which allows both females and males have partenogenesis (asexual reproduction).
In this model, every individual (regardless whether it is a female or male) corresponds to 
 a mating unit and can have female and male offsprings independently of the other individuals.
There is a kind of phenomenon in nature as well, e.g., for aphids and for some reptiles. 
The following result can be considered as another new 
 contribution in the field of  Feller-type diffusion approximations of 2SBPIs.

\begin{corollary}\label{cor-2SBPI_2}
 Let $(F_k,M_k)_{k\in\Z_+}$ be a 2SBPI defined in \eqref{eq-definition}  with the self-fertilization mating function.
 Assume that $\ev{\|(F_0,M_0)\|^2}<\infty$,  $\ev{\|(f_{0,1},m_{0,1})\|^4}<\infty$,
  $\ev{\|(F_1^I,M_1^I)\|^4}<\infty$, and $\ev{f_{0,1}},\ev{m_{0,1}}\in(0,1)$ are such that $\ev{f_{0,1}} + \ev{m_{0,1}} = 1$.
Then
 \begin{equation*}
        (n^{-1} (F_{\lfloor nt \rfloor}, M_{\lfloor nt \rfloor}))_{t \in\R_+}
        \stackrel{\mathcal{L}}{\longrightarrow}
        (\mathcal{X}_t (\ev{f_{0,1}},\ev{m_{0,1}}))_{t \in\R_+} \qquad \text{as } n\to\infty,
 \end{equation*}
 where $(\mathcal{X}_t)_{t\in\R_+}$ is the pathwise unique strong solution of the SDE
 \begin{equation*}
        \dif\mathcal{X}_t = (\ev{F_1^I} + \ev{M_1^I})\dif t
             + \sqrt{ \var{f_{0,1} + m_{0,1}} \mathcal{X}_t^+} \dif\mathcal{W}_t,
             \qquad t \in\R_+,
 \end{equation*}
 with initial value $\mathcal{X}_0 = 0$, where $(\mathcal{W}_t)_{t\in\R_+}$ is a standard Wiener process.
\end{corollary}

\begin{proof}
Let us rewrite $(F_k,M_k)_{k\in\Z_+}$ as a CMBP $(\bZ_k)_{k\in\Z_+}$ (see  part (ii) of Example \ref{Example1}).
We want to apply Theorem \ref{Thm_main} for $(\bZ_k)_{k\in\Z_+}$.
The control distributions are deterministic, so Hypothesis \ref{hyp-variance-control-o} and \ref{hyp-moment-4} are trivially satisfied. In addition,  we get
    \begin{align*}
        \ev{\|\bZ_0\|^2} &= \ev{\|(F_0,M_0)\|^2} < \infty, \\
        \ev{\|\boldsymbol{X}_{0,1,1}\|^4} &= \ev{\|(f_{0,1},m_{0,1})\|^4} < \infty, \\
        \ev{\|\boldsymbol{X}_{0,1,2}\|^4} &= \ev{\|(F_1^I,M_1^I)\|^4} < \infty, \\
        \ev{\|\boldsymbol{\phi}_0(\bz)\|^4} &=  ( (z_1+z_2)^2+1)^2 < \infty, \qquad \bz\in\Z_+^2 ,
    \end{align*}
     yielding that Hypothesis \ref{hyp-second-moment-Z0} holds.
For each $\bz\in\Z_+^2$, we have
    \begin{equation*}
             \boldsymbol{\phi}_0 (\boldsymbol{z})
             = \begin{pmatrix}
                z_1+z_2 \\
                 1
             \end{pmatrix}
            = \mathsf{\Lambda}\boldsymbol{z} + \boldsymbol{\alpha} \qquad \text{ with }\quad
            \mathsf{\Lambda} = \begin{pmatrix}
                1 & 1 \\
                0 & 0
            \end{pmatrix} \quad\mbox{ and }\quad  \boldsymbol{\alpha} =
            \begin{pmatrix}
                0 \\
                1
            \end{pmatrix},
        \end{equation*}
 and hence Hypothesis \ref{hyp-g} holds with the given $\mathsf{\Lambda}$, $\boldsymbol{\alpha}$, and
  $\boldsymbol{g}\equiv  \boldsymbol{0}_2$.
Hypothesis \ref{hyp-explosion} is not needed, see part \ref{rem-explosion} of Remark \ref{remark-hyp}.
Further, we get
        \begin{align*}
            \mathsf{m} &= \begin{pmatrix}
                \ev{f_{0,1}} & \ev{F_1^I} \\[1mm]
                \ev{m_{0,1}} & \ev{M_1^I}
            \end{pmatrix}, & \mathsf{\Sigma}_1 &= \begin{pmatrix}
                \var{f_{0,1}} & \cov{f_{0,1}}{m_{0,1}} \\
                \cov{f_{0,1}}{m_{0,1}} & \var{m_{0,1}}
            \end{pmatrix}, \\
            \tilde{\mathsf{m}} &= \begin{pmatrix}
                \ev{f_{0,1}} & \ev{f_{0,1}} \\
                \ev{m_{0,1}} & \ev{m_{0,1}}
            \end{pmatrix}, & \mathsf{\Sigma}_2 &= \begin{pmatrix}
                \var{F_1^I} & \cov{F_1^I}{M_1^I} \\[1mm]
                \cov{F_1^I}{M_1^I} & \var{M_1^I}
            \end{pmatrix}.
        \end{align*}
The two eigenvalues of $\tilde{\mathsf{m}}$ are $0$ and $\ev{f_{0,1}}+\ev{m_{0,1}}=1$, and consequently the spectral radius of $\tilde{\mathsf{m}}$ is $1$. 
Using also that $\ev{f_{0,1}},\ev{m_{0,1}}\in(0,1)$, we have that
 $\tilde{\mathsf{m}}$ is a primitive matrix with Frobenius--Perron eigenvalue $1$.
Note also that $\tilde{\mathsf{m}}^k = \tilde{\mathsf{m}}$, $k\in\N$, since $\ev{f_{0,1}}+\ev{m_{0,1}}=1$. 
Taking into these considerations, one can check that the conditions of Hypothesis \ref{hyp-primitive-relaxed} are satisfied with
 \begin{align*}
   \tilde{\boldsymbol{u}} 
          = \begin{pmatrix}
                \ev{f_{0,1}} \\
                \ev{m_{0,1}}
            \end{pmatrix}, \qquad 
    \tilde{\boldsymbol{v}} 
         = \begin{pmatrix}
                1 \\
                1
            \end{pmatrix}, \qquad
            \mathsf{\Lambda}\tilde{\boldsymbol{u}} 
            = \begin{pmatrix}
                1 \\
                0
            \end{pmatrix},\qquad 
         \tilde{\mathsf{\Pi}} =  \lim_{k\to\infty} \tilde{\mathsf{m}}^k
                               = \tilde{\mathsf{m}},
 \end{align*}
 and the constants $\tilde{c}\in\R_{++}$ and $\tilde{r}\in(0,1)$ in part (c) of Hypothesis \ref{hyp-primitive-relaxed} can be chosen arbitrarily.
Therefore, Theorem \ref{Thm_main}  implies that
        \begin{equation*}
            \left(\frac{1}{n} \begin{pmatrix}
                F_{\lfloor nt \rfloor} \\
                M_{\lfloor nt \rfloor}
            \end{pmatrix}\right)_{t\in\R_+}
            \stackrel{\mathcal{L}}{\longrightarrow}
          \left(\mathcal{X}_t 
             \begin{pmatrix}
                \ev{f_{0,1}} \\
                \ev{m_{0,1}}
            \end{pmatrix}\right)_{t\in\R_+} \qquad \text{as } n\to \infty,
        \end{equation*}
 where $(\mathcal{X}_t)_{t\in\R_+}$ is the pathwise unique strong solution of the SDE
  \begin{align*}
   \dif\mathcal{X}_t 
   & =   \begin{pmatrix}
          1 & 1 \\
        \end{pmatrix}
        \mathsf{m}
        \begin{pmatrix}
          0 \\
          1 \\
        \end{pmatrix}
        \dif t    
       + \sqrt{\begin{pmatrix}
                1 & 1 \\
              \end{pmatrix} 
               \big( 1\cdot \mathsf{\Sigma}_1 + 0 \cdot \mathsf{\Sigma}_2 \big)
        \begin{pmatrix}
          1 \\
          1 \\
        \end{pmatrix} \mathcal{X}^+_t
        }  \dif\mathcal{W}_t \\       
     &= (\ev{F_1^I} + \ev{M_1^I})\dif t
        + \sqrt{(\var{f_{0,1}} + 2\cov{f_{0,1}}{m_{0,1}} + \var{m_{0,1}})\mathcal{X}^+_t}         
          \dif\mathcal{W}_t, \qquad t\in\R_+,
  \end{align*}
  with initial value $\mathcal{X}_0 = 0$, where $(\mathcal{W}_t)_{t\in\R_+}$ is a standard Wiener process.
Since $\var{f_{0,1}} + 2\cov{f_{0,1}}{m_{0,1}} + \var{m_{0,1}} = \var{f_{0,1}+m_{0,1}}$, this yields the assertion.  
\end{proof}

\begin{remark}\label{Rem_Bansaye}
    Scaling limits for a class of processes that combine classical asexual Galton--Watson processes and two-sex Galton--Watson branching processes without immigration (introduced by Daley \cite{DALEY_1968}) have been recently studied in Bansaye et al.\ \cite{bcms23}. However, this new family of processes does not include the class of two-sex Galton--Watson branching processes with immigration as a particular case, so the results in Bansaye et al.\ \cite{bcms23} do not imply our Corollaries \ref{cor-2SBPI} and \ref{cor-2SBPI_2}. 
    Even if we recover two-sex Galton--Watson branching processes without immigration  as particular cases of the model (1)-(2) in Bansaye et al.\ \cite{bcms23}  (using their notations, with the special choices $\mathcal{E}_{n,p}^{f,N} := -1$ and $\mathcal{E}_{n,p}^{m,N}:=-1$, $n,p\in\N$,  which intuitively mean that the individuals (females or males) do not survive in the next generation), then their  Theorem 2.1 cannot be applied, because their assumption (A1)  does not hold in this case. 
    Furthermore, note that 
    for two-sex Galton--Watson branching processes without immigration, our Corollary \ref{cor-2SBPI} cannot be applied as well,
    since the assumption $\prob{M_1^I=0}=0$ does not hold in this case, however Corollary  \ref{cor-2SBPI_2}
    can be applied and the limit process of the scaled two-sex Galton--Watson branching processes without immigration in question is the identically zero process.
    \proofendb
\end{remark}

Next, we formulate another corollary of Theorem \ref{Thm_main}, namely, we derive a limit 
 distribution for the relative frequencies of distinct types of individuals.
For different models, one can find similar results 
 in Jagers \cite[Corollary 1]{JAGERS_1969}, 
 in Yakovlev and Yanev  \cite[Proposition 1]{YAKOVLEV_YANEV_2009},   
 \cite[Theorem 2]{YAKOVLEV_YANEV_2010}, and in Barczy and Pap \cite[Corollary 4.1]{BARCZY_PAP_2016}.

Recall that if $\tilde{\mathsf{m}}$ is primitive with  
 Perron--Frobenius eigenvalue $1$, then $\Tilde{\boldsymbol{u}}$
 and $\Tilde{\boldsymbol{v}}$ denote the right and left Perron--Frobenius
 eigenvectors corresponding to $1$, respectively.

\begin{corollary}\label{Cor_rel_frequency}
Suppose that Hypotheses \ref{hyp-second-moment-Z0}--\ref{hyp-moment-4} and \ref{hyp-explosion} hold for the CMBP 
 $(\boldsymbol{Z}_{k})_{k\in\Z_+}$ given in \eqref{CMBP_def}. 
Assume also that the matrix $\tilde{\mathsf{m}}$ is primitive with 
 Perron--Frobenius eigenvalue $1$, and $\mathsf{\Lambda} \Tilde{\boldsymbol{u}}\in\R_{+}^p$.
If, in addition, $\tilde{\boldsymbol{v}}^\top \mathsf{m}\boldsymbol{\alpha}
 > 0$, then for all $t>0$ and $i,j\in\{1,\ldots,p\}$, we get 
 \[
 \1_{\{ \boldsymbol{e}_j^\top \bZ_{\lfloor nt\rfloor}\neq 0\}}
     \frac{\boldsymbol{e}_i^\top \bZ_{\lfloor nt\rfloor}}
          {\boldsymbol{e}_j^\top \bZ_{\lfloor nt\rfloor}}
          \stackrel{\mathrm{P}}{\longrightarrow}
          \frac{\boldsymbol{e}_i^\top \tilde{\boldsymbol{u}}}
                      {\boldsymbol{e}_j^\top \tilde{\boldsymbol{u}}}
  \qquad \text{and}\qquad 
  \1_{\{ \bZ_{\lfloor nt\rfloor}\neq  \boldsymbol{0}_p\} }
     \frac{\boldsymbol{e}_i^\top \bZ_{\lfloor nt\rfloor}}
          {\sum_{k=1}^p \boldsymbol{e}_k^\top \bZ_{\lfloor nt\rfloor}}
        \stackrel{\mathrm{P}}{\longrightarrow}
        \boldsymbol{e}_i^\top \tilde{\boldsymbol{u}}
        \qquad \text{as $n\to\infty$.}
 \]
\end{corollary}

\begin{remark}
The indicator functions $\1_{\{ \boldsymbol{e}_j^\top \bZ_{\lfloor nt\rfloor}\ne 0\}}$ and $\1_{\{ \bZ_{\lfloor nt\rfloor}\ne  \boldsymbol{0}_p\} }$ are needed in Corollary \ref{Cor_rel_frequency},
 since it can happen that $\prob{\bZ_{\lfloor nt\rfloor} =  \boldsymbol{0}_p}>0$ for some $t>0$.
\proofendb
\end{remark}

\noindent{\textit{Proof of Corollary \ref{Cor_rel_frequency}}.}
For all $t>0$ and $i,j\in\{1,\ldots,p\}$, Theorem \ref{Thm_main} yields that
 \[
  \frac{1}{n}\Big(\boldsymbol{e}_i^\top \bZ_{\lfloor nt\rfloor}, 
                   \boldsymbol{e}_j^\top \bZ_{\lfloor nt\rfloor} \Big)
    \stackrel{\mathcal{L}}{\longrightarrow} 
     \Big(\boldsymbol{e}_i^\top \tilde{\boldsymbol{u}}\,\mathcal{Z}_t,
          \boldsymbol{e}_j^\top \tilde{\boldsymbol{u}}\,\mathcal{Z}_t \Big)    
          \qquad \text{as $n\to\infty$.}
 \]
The function $g:\R^ 2\to \R$ defined by 
 \[
   g(x,y)\colonequals \begin{cases}
               \frac{x}{y}  & \text{if $x\in\R$ and $y\ne 0$,}\\
               0 & \text{if $x\in\R$ and $y=0$,}
           \end{cases}
 \]
 is continuous on the set $\R\times (\R\setminus\{0\})$, and the distribution of 
 $ \Big(\boldsymbol{e}_i^\top \tilde{\boldsymbol{u}}\,\mathcal{Z}_t,
          \boldsymbol{e}_j^\top \tilde{\boldsymbol{u}}\,\mathcal{Z}_t \Big)$
 is concentrated on this set, since, by Remark \ref{rem-unique}, $\prob{\mathcal{Z}_t >0}=1$ for all $t>0$ and $\boldsymbol{e}_j^\top \tilde{\boldsymbol{u}}>0$.
 Hence the continuous mapping theorem yields that 
  \begin{align*}
    \1_{\{\boldsymbol{e}_j^\top \bZ_{\lfloor nt\rfloor}\ne 0\}}
     \frac{\boldsymbol{e}_i^\top \bZ_{\lfloor nt\rfloor}}
          {\boldsymbol{e}_j^\top \bZ_{\lfloor nt\rfloor}}
    &= g(n^{-1}\boldsymbol{e}_i^\top \bZ_{\lfloor nt\rfloor}, 
     n^{-1}\boldsymbol{e}_j^\top \bZ_{\lfloor nt\rfloor}) \\   
    &\stackrel{\mathcal{L}}{\longrightarrow} 
     g(\boldsymbol{e}_i^\top \tilde{\boldsymbol{u}}\,\mathcal{Z}_t, \boldsymbol{e}_j^\top \tilde{\boldsymbol{u}}
         \,\mathcal{Z}_t)   
   = \1_{\{ \boldsymbol{e}_j^\top \tilde{\boldsymbol{u}} \,\mathcal{Z}_t\ne 0\}}
         \frac{\boldsymbol{e}_i^\top \tilde{\boldsymbol{u}}\,\mathcal{Z}_t}
           {\boldsymbol{e}_j^\top \tilde{\boldsymbol{u}}\,\mathcal{Z}_t}        
   = \frac{\boldsymbol{e}_i^\top \tilde{\boldsymbol{u}} }
           {\boldsymbol{e}_j^\top \tilde{\boldsymbol{u}}}
   \qquad \text{as $n\to\infty$,}       
   \end{align*}
 thus we obtain the first convergence (using also that the limit is not random, so
 convergence in distribution yields convergence in probability in this case).

Similarly, for each $i\in\{1,\ldots,p\}$, the function $h:\R^ p\to\R$ defined by
 \[
  h(x_1,\ldots,x_p)\colonequals \begin{cases}
               \frac{x_i}{\sum_{i=1}^p x_k}  & \text{if $(x_1,\ldots,x_p)\in\R^p$ 
                                              and $\sum_{k=1}^p x_k\ne 0$,}\\
               0 & \text{if $(x_1,\ldots,x_p)\in\R^p$ 
                                              and $\sum_{k=1}^p x_k = 0$,}
           \end{cases}
 \]
 is continuous on the set $\R^ p \setminus\{(x_1,\ldots,x_p)\in\R^p : \sum_{k=1}^p x_k =0\}$,
 and the distribution of $\tilde{\boldsymbol{u}}\mathcal{Z}_t$ is concentrated on this set,
 since $\prob{\mathcal{Z}_t >0}=1$, $t>0$, and $\tilde{\boldsymbol{u}}\in\R_{++}^p$
 imply that $\prob{\sum_{k=1}^p \boldsymbol{e}_k^\top \tilde{\boldsymbol{u}} \mathcal{Z}_t =0 } =0$.
Using that $\1_{\{ \bZ_{\lfloor nt\rfloor} \ne  \boldsymbol{0}_p\} } 
 = \1_{\{ \sum_{k=1}^p \boldsymbol{e}_k^\top \bZ_{\lfloor nt\rfloor} \ne 0\}}$,
 the continuous mapping theorem yields that 
 \begin{align*}
 \1_{\{ \bZ_{\lfloor nt\rfloor}\ne  \mathbf{0}_p\} }
     \frac{\boldsymbol{e}_i^\top \bZ_{\lfloor nt\rfloor}}
          {\sum_{k=1}^p \boldsymbol{e}_k^\top \bZ_{\lfloor nt\rfloor}}
  &= h(n^{-1}\bZ_{\lfloor nt\rfloor}) \\        
  & \stackrel{\mathcal{L}}{\longrightarrow} 
    h(\tilde{\boldsymbol{u}} \mathcal{Z}_t )
    = \1_{\{ \sum_{k=1}^p \boldsymbol{e}_k^\top \tilde{\boldsymbol{u}} \mathcal{Z}_t \ne 0 \}}
         \frac{\boldsymbol{e}_i^\top \tilde{\boldsymbol{u}} \mathcal{Z}_t}
            {\sum_{k=1}^p\boldsymbol{e}_k^\top \tilde{\boldsymbol{u}} \mathcal{Z}_t}         
   = \frac{\boldsymbol{e}_i^\top \tilde{\boldsymbol{u}} }
            {\sum_{k=1}^p\boldsymbol{e}_k^\top \tilde{\boldsymbol{u}} }
   = \boldsymbol{e}_i^\top \tilde{\boldsymbol{u}}
 \end{align*}
 as $n\to\infty$, where we also used that the 
sum of coordinates of $\tilde{\boldsymbol{u}}$ is 1 (see part (ii) of Lemma \ref{lemma-primitive}).
Thus we obtain the second convergence as well (using again that the limit is not random).
\proofend

\section{Proof of Theorem \ref{Thm_main}}\label{section-proof}

The proof is divided in four steps. 
In the first three Steps 1, 2 and 3, we introduce some auxiliary stochastic 
 processes, prove their convergence in distribution, and in Step 4,
 we discuss that how Steps 1--3 yield the assertion.
The detailed proofs for Steps 1--3 can be found after Step 4.

\subsection*{Step 1}\label{step1}

From the  process $(\boldsymbol{Z}_{k})_{k\in\Z_+},$ we define a  martingale difference sequence with respect to the filtration $(\mathcal{F}_{k})_{k\in\Z_+}$ given by
    \begin{equation}\label{eq-def-mtg-dif}
        \boldsymbol{M}_{k}  \colonequals  \boldsymbol{Z}_{k} - \evcond{\boldsymbol{Z}_{k}}{\mathcal{F}_{k-1}} , \qquad k\in\N ,
    \end{equation}
    and we also consider the following sequence of random step processes
    \begin{equation}\label{eq-def-step-from-mtg}
        \boldsymbol{\mathcal{M}}_t^{(n)} = n^{-1} \left( \boldsymbol{Z}_{0} + \sum_{k=1}^{\lfloor nt\rfloor} \boldsymbol{M}_{k} \right), \quad\quad t\in\R_+ , \quad n\in\N.
    \end{equation}
    
    We will prove that
    \begin{equation}\label{eq-conv-M}
        \boldsymbol{\mathcal{M}}^{(n)} \stackrel{\mathcal{L}}{\longrightarrow} \boldsymbol{\mathcal{M}} \qquad \text{as } n\to\infty, 
    \end{equation}
    where $(\boldsymbol{\mathcal{M}}_t)_{t\in\R_+}$ is the pathwise unique strong solution of the SDE 
    \begin{equation}\label{eq-SDE-M}
        \dif \boldsymbol{\mathcal{M}}_t = \sqrt{(\mathsf{\Lambda}\tilde{\mathsf{\Pi}}(\boldsymbol{\mathcal{M}}_t+t\mathsf{m}\boldsymbol{\alpha}))^+ \odot \boldsymbol{\mathsf{\Sigma}}} \dif \boldsymbol{\mathcal{W}}_t, \qquad t\in\R_+,
    \end{equation}
    with initial value $\boldsymbol{0}_p$, where $(\boldsymbol{\mathcal{W}}_t)_{t\in\R_+}$ is a $p$--dimensional standard Wiener process.

\subsection*{Step 2}
 
We will show that \eqref{eq-conv-M} implies 
 \begin{equation}\label{help_psi_n_to_psi}
    \boldsymbol{\Psi}_n(\boldsymbol{\mathcal{M}}^{(n)}) \stackrel{\mathcal{L}}{\longrightarrow} \boldsymbol{\Psi}(\boldsymbol{\mathcal{M}})
        \qquad \text{as } n\to\infty,
    \end{equation} 
    where the mappings $\boldsymbol{\Psi}:\mathbf{D}(\R_+,\R^p)\to\mathbf{D}(\R_+,\R^p)$ and $\boldsymbol{\Psi}_n:\mathbf{D}(\R_+,\R^p)\to\mathbf{D}(\R_+,\R^p)$, $n\in\N$, are given by
    \begin{align}\label{help_psi}
        \boldsymbol{\Psi}(\boldsymbol{f})(t) & \colonequals \tilde{\mathsf{\Pi}}(\boldsymbol{f}(t)+t\mathsf{m}\boldsymbol{\alpha}),\\ \label{help_psi_n}
        \boldsymbol{\Psi}_n(\boldsymbol{f})(t) &\colonequals \tilde{\mathsf{m}}^{\lfloor nt\rfloor}\boldsymbol{f}(0) + \sum_{j=1}^{\lfloor nt\rfloor} \tilde{\mathsf{m}}^{\lfloor nt\rfloor - j} \left(\boldsymbol{f}\left(\frac{j}{n}\right) - \boldsymbol{f}\left(\frac{j-1}{n}\right) + \frac{1}{n}\mathsf{m}\boldsymbol{\alpha}\right),
    \end{align}
    for $\boldsymbol{f}\in\mathbf{D}(\R_+,\R^p)$, $t\in\R_+$, $n\in\N$.

\subsection*{Step 3}\label{step3}

We will check that $\boldsymbol{\mathcal{Z}}^{(n)} = \boldsymbol{\Psi}_n(\boldsymbol{\mathcal{M}}^{(n)}) + \boldsymbol{\mathcal{V}}^{(n)}$, $n\in\N$, and
    \begin{equation}\label{eq-conv-sum-seq}
        \boldsymbol{\mathcal{Z}}^{(n)} \stackrel{\mathcal{L}}{\longrightarrow} \boldsymbol{\Psi}(\boldsymbol{\mathcal{M}})
        \qquad \text{as } n\to\infty,
    \end{equation}
    where $(\boldsymbol{\mathcal{V}}_t^{(n)})_{t\in\R_+},$ $n\in\N,$ is another sequence of random step processes defined by
    \begin{equation}\label{help_V}
        \boldsymbol{\mathcal{V}}_t^{(n)} 
         \colonequals  n^{-1} \sum_{j=1}^{\lfloor nt\rfloor} 
             \tilde{\mathsf{m}}^{\lfloor nt\rfloor - j} 
             \mathsf{m}\boldsymbol{g}(\boldsymbol{Z}_{j-1}), \qquad t\in\R_+,
             \quad n\in\N.
    \end{equation}
    
\subsection*{Step 4}\label{step4}

  As a consequence of Steps 1--3, one can easily derive
  \eqref{eq-convergence-main-thm}.
  Namely, let us define 
    \begin{align}\label{eq-definition-limit-process-main-thm-multi-dim}
        \boldsymbol{\mathcal{Z}}_t &\colonequals \boldsymbol{\Psi}(\boldsymbol{\mathcal{M}}_t) = \tilde{\mathsf{\Pi}}(\boldsymbol{\mathcal{M}}_t+t\mathsf{m}\boldsymbol{\alpha}), \qquad t\in\R_+,\\ 
        \mathcal{Y}_t &\colonequals \tilde{\boldsymbol{v}}^\top \boldsymbol{\mathcal{Z}}_t = \tilde{\boldsymbol{v}}^\top (\boldsymbol{\mathcal{M}}_t+t\mathsf{m}\boldsymbol{\alpha}), \qquad t\in\R_+, \label{eq-definition-limit-process-main-thm-uni-dim}
    \end{align}
    where we used that $\tilde{\boldsymbol{v}}^\top \tilde{\mathsf{\Pi}} = 
         \tilde{\boldsymbol{v}}^\top \tilde{\boldsymbol{u}} \tilde{\boldsymbol{v}}^\top = \tilde{\boldsymbol{v}}^\top$  (see Hypothesis \ref{hyp-primitive-relaxed}).%
    
    Then $\boldsymbol{\mathcal{Z}} = \mathcal{Y} \tilde{\boldsymbol{u}}$,
    since 
    \[
      \mathcal{Y}_t \tilde{\boldsymbol{u}}
         = \tilde{\boldsymbol{v}}^\top (\boldsymbol{\mathcal{M}}_t+t\mathsf{m}  
           \boldsymbol{\alpha}) \tilde{\boldsymbol{u}}
         = \tilde{\boldsymbol{u}} (\tilde{\boldsymbol{v}}^\top 
               \boldsymbol{\mathcal{M}}_t) 
             + t \tilde{\boldsymbol{u}} (\tilde{\boldsymbol{v}}^\top \mathsf{m}  \boldsymbol{\alpha})
        =  \tilde{\mathsf{\Pi}} \boldsymbol{\mathcal{M}}_t
            +t \tilde{\mathsf{\Pi}}\mathsf{m}  \boldsymbol{\alpha}
       =   \boldsymbol{\mathcal{Z}}_t, \qquad t\in\R_+.
    \]

Further, by Itô's formula, we can show that $(\mathcal{Y}_t)_{t\in\R_+}$ is the pathwise unique strong solution of the SDE (\ref{eq-SDE-main-thm-1-dim}) with initial value 0, thus \eqref{eq-convergence-main-thm} follows and the proof concludes. Indeed,
  the SDE \eqref{eq-SDE-M} can also be written in the form
    \begin{equation}\label{eq-SDE-M-rewritten}
        \dif \boldsymbol{\mathcal{M}}_t 
        = \sqrt{(\tilde{\boldsymbol{v}}^\top(\boldsymbol{\mathcal{M}}_t+t\mathsf{m}\boldsymbol{\alpha}))^+ ((\mathsf{\Lambda}\tilde{\boldsymbol{u}}) \odot \boldsymbol{\mathsf{\Sigma}})} \dif \boldsymbol{\mathcal{W}}_t, \qquad t\in\R_+,
    \end{equation}
    since $\mathsf{\Lambda}\Tilde{\boldsymbol{u}}\in\R_+^{p}$ (see  Hypothesis \ref{hyp-primitive-relaxed}) implies that
    \begin{equation*}
    (\tilde{\boldsymbol{v}}^\top(\boldsymbol{\mathcal{M}}_t+t\mathsf{m}\boldsymbol{\alpha}))^+  \mathsf{\Lambda}\tilde{\boldsymbol{u}}
        =  ((\mathsf{\Lambda}\tilde{\boldsymbol{u}}) \tilde{\boldsymbol{v}}^
            \top(\boldsymbol{\mathcal{M}}_t+t\mathsf{m}\boldsymbol{\alpha}))^+  
        = ( \mathsf{\Lambda}\tilde{\mathsf{\Pi}} (\boldsymbol{\mathcal{M}}_t+t\mathsf{m}\boldsymbol{\alpha}))^+,\qquad t\in\R_+.
    \end{equation*}
Therefore,  $(\mathcal{Y}_t)_{t\in\R_+}$ is the pathwise unique strong 
 solution of the SDE
  \begin{equation}\label{SDE-1-dim-transform}
      \dif\mathcal{Y}_t = \tilde{\boldsymbol{v}}^\top \mathsf{m}\boldsymbol{\alpha}\dif t + \sqrt{\mathcal{Y}_t^+}\tilde{\boldsymbol{v}}^\top \sqrt{(\mathsf{\Lambda}\tilde{\boldsymbol{u}})\odot\boldsymbol{\mathsf{\Sigma}}}\dif \boldsymbol{\mathcal{W}}_t, \qquad t\in\R_+,
  \end{equation}
 with initial value $0$, since, 
  by \eqref{eq-definition-limit-process-main-thm-uni-dim}, $\dif \mathcal{Y}_t = \tilde{\boldsymbol{v}}^\top \mathsf{m}\boldsymbol{\alpha} \dif t + \tilde{\boldsymbol{v}}^\top \dif \boldsymbol{\mathcal{M}}_t$, $t\in\R_+$, and 
 $(\boldsymbol{\mathcal{M}}_t)_{t\in\R_+}$ is the pathwise unique strong solution of \eqref{eq-SDE-M-rewritten} with initial value $\boldsymbol{0}_p$.

Suppose that $\Tilde{\boldsymbol{v}}^\top((\mathsf{\Lambda}\tilde{\boldsymbol{u}})\odot\boldsymbol{\mathsf{\Sigma}})\Tilde{\boldsymbol{v}} \neq 0$, 
  then the process $(\mathcal{W}_t)_{t\in\R_+}$ given by
  \[
  \mathcal{W}_t \colonequals \left(\Tilde{\boldsymbol{v}}^\top((\mathsf{\Lambda}\tilde{\boldsymbol{u}})\odot\boldsymbol{\mathsf{\Sigma}})\Tilde{\boldsymbol{v}}\right)^{-1/2}\Tilde{\boldsymbol{v}}^\top\sqrt{(\mathsf{\Lambda}\tilde{\boldsymbol{u}})\odot\boldsymbol{\mathsf{\Sigma}}} \,\boldsymbol{\mathcal{W}}_t,
  \qquad t\in\R_+,
  \]
   is a well-defined one-dimensional standard Wiener process, since $\mathsf{\Lambda}\tilde{\boldsymbol{u}}\in\R_+^p$,  yielding that $(\mathsf{\Lambda}\tilde{\boldsymbol{u}})\odot\boldsymbol{\mathsf{\Sigma}}$ is a positive semi-definite matrix. In such a case, the  SDE
   \eqref{SDE-1-dim-transform} can be written as
    \begin{align}\label{SDE_calP}
        \dif\mathcal{Y}_t = \tilde{\boldsymbol{v}}^\top \mathsf{m}\boldsymbol{\alpha}\dif t + \sqrt{\mathcal{Y}_t^+} \sqrt{\Tilde{\boldsymbol{v}}^\top((\mathsf{\Lambda}\tilde{\boldsymbol{u}})\odot\boldsymbol{\mathsf{\Sigma}})\Tilde{\boldsymbol{v}}}\dif \mathcal{W}_t,
        \qquad t\in\R_+,
    \end{align}
    which coincides with the SDE \eqref{eq-SDE-main-thm-1-dim},  as desired. 
    Otherwise, it is trivial because if $\Tilde{\boldsymbol{v}}^\top((\mathsf{\Lambda}\tilde{\boldsymbol{u}})\odot\boldsymbol{\mathsf{\Sigma}})\Tilde{\boldsymbol{v}} = 0$, then $\|\Tilde{\boldsymbol{v}}^\top\sqrt{(\mathsf{\Lambda}\tilde{\boldsymbol{u}})\odot\boldsymbol{\mathsf{\Sigma}}}\| = 0$,
    yielding that SDEs \eqref{eq-SDE-main-thm-1-dim} and \eqref{SDE-1-dim-transform} correspond to the ODE $\dif x(t) = \tilde{\boldsymbol{v}} \mathsf{m}\boldsymbol{\alpha}\dif t$, $t\in\R_+$,  with initial value $0$,
     and this ODE has the identically zero solution.

    Finally, notice that the statement in Remark \ref{rem-limit-process-multi} can be derived using the equality $\boldsymbol{\mathcal{Z}} = \mathcal{Y} \tilde{\boldsymbol{u}}$, 
    \eqref{eq-conv-sum-seq},
     \eqref{eq-definition-limit-process-main-thm-multi-dim},
     \eqref{eq-definition-limit-process-main-thm-uni-dim}, 
     and \eqref{eq-SDE-M-rewritten}.
Namely, we have 
 \begin{align*}
 \dif\boldsymbol{\mathcal{Z}}_t 
  = \tilde{\mathsf{\Pi}} \mathsf{m}\boldsymbol{\alpha}\dif t 
    + \tilde{\mathsf{\Pi}} \dif \boldsymbol{\mathcal{M}}_t
  = \tilde{\mathsf{\Pi}} \mathsf{m}\boldsymbol{\alpha}\dif t 
    + \tilde{\mathsf{\Pi}}  
     \sqrt{ \mathcal{Y}_t^+ ( (\mathsf{\Lambda}\tilde{\boldsymbol{u}})\odot\boldsymbol{\mathsf{\Sigma}} ) }
     \dif \boldsymbol{\mathcal{W}}_t, \qquad t\in\R_+.
 \end{align*}
Since $\mathsf{\Lambda}\tilde{\boldsymbol{u}}\in\R_+^p$, we have 
 \[
   \mathcal{Y}_t^+ \mathsf{\Lambda}\tilde{\boldsymbol{u}}
   = (\mathcal{Y}_t\mathsf{\Lambda}\tilde{\boldsymbol{u}})^+
   = (\mathsf{\Lambda}\mathcal{Y}_t\tilde{\boldsymbol{u}})^+ 
   = (\mathsf{\Lambda}\boldsymbol{\mathcal{Z}}_t)^+, \qquad t\in\R_+,
  \]
  yielding \eqref{SDE_limit_new}, as desired.
    
\subsection*{Proof of Step 1}

 We need to show that (\ref{eq-conv-M}) holds. To this end, we apply  Theorem \ref{thm-ispany-pap} with $\boldsymbol{b}(t,\boldsymbol{x})= \boldsymbol{0}_p\in\R^p$, $\mathsf{C} (t,\boldsymbol{x}) = \sqrt{(\mathsf{\Lambda\tilde{\Pi}}(\boldsymbol{x}+t\mathsf{m}\boldsymbol{\alpha}))^+ \odot \boldsymbol{\mathsf{\Sigma}}}\in\R^{p\times p}$, $(t,\boldsymbol{x})\in\R_+\times\R^p$,   $\boldsymbol{\eta}$ is the Dirac measure concentrated at $\boldsymbol{0}_p$, $ \boldsymbol{\mathcal{U}} = \boldsymbol{\mathcal{M}}$, $\mathcal{F}_k^{(n)} = \mathcal{F}_{k}$, $ \boldsymbol{U}_{0}^{(n)} = n^{-1} \boldsymbol{Z}_{0}$ and $ \boldsymbol{U}_{k}^{(n)} 
 = n^{-1} \boldsymbol{M}_{k}$, $k,n\in\N$.
    
First, we check that the SDE (\ref{eq-SDE-M}) has a pathwise unique strong solution for all $\R^p$--valued initial values.
As seen in Step 4, we can rewrite the SDE \eqref{eq-SDE-M} in the form \eqref{eq-SDE-M-rewritten}.

If $(\boldsymbol{\mathcal{M}}_t^{(\boldsymbol{y}_0)})_{t\in\R_+}$ is a strong solution of (\ref{eq-SDE-M-rewritten}) with initial value $\boldsymbol{\mathcal{M}}_0^{(\boldsymbol{y}_0)} = \boldsymbol{y}_0 \in\R^p$, then we check that the process $((\mathcal{P}_t^{(\boldsymbol{y}_0)},\boldsymbol{\mathcal{Q}}_t^{(\boldsymbol{y}_0)}))_{t\in\R_+}$, defined by
    \begin{equation*}
        \mathcal{P}_t^{(\boldsymbol{y}_0)} \colonequals \tilde{\boldsymbol{v}}^\top(\boldsymbol{\mathcal{M}}_t^{(\boldsymbol{y}_0)}+t\mathsf{m}\boldsymbol{\alpha}), \qquad \boldsymbol{\mathcal{Q}}_t^{(\boldsymbol{y}_0)} \colonequals \boldsymbol{\mathcal{M}}_t^{(\boldsymbol{y}_0)} - \mathcal{P}_t^{(\boldsymbol{y}_0)}\tilde{\boldsymbol{u}}, \qquad t\in\R_+,
    \end{equation*}
    is a strong solution of the SDE
    \begin{equation}\label{eq-system-DE}
        \begin{cases}
            \dif\mathcal{P}_t = \tilde{\boldsymbol{v}}^\top \mathsf{m}\boldsymbol{\alpha}\dif t + \sqrt{\mathcal{P}_t^+}\tilde{\boldsymbol{v}}^\top \sqrt{(\mathsf{\Lambda}\tilde{\boldsymbol{u}})\odot\boldsymbol{\mathsf{\Sigma}}}\dif \boldsymbol{\mathcal{W}}_t,\\
            \dif \boldsymbol{\mathcal{Q}}_t = -\tilde{\mathsf{\Pi}}\mathsf{m}\boldsymbol{\alpha}\dif t + \sqrt{\mathcal{P}_t^+}(\mathsf{I}_p - \tilde{\mathsf{\Pi}}) \sqrt{(\mathsf{\Lambda}\tilde{\boldsymbol{u}})\odot\boldsymbol{\mathsf{\Sigma}}}\dif \boldsymbol{\mathcal{W}}_t,
        \end{cases}
        \qquad t\in\R_+,
    \end{equation}
    with initial value 
    \[
    (\mathcal{P}_0^{(\boldsymbol{y}_0)},\boldsymbol{\mathcal{Q}}_0^{(\boldsymbol{y}_0)}) 
    = (\tilde{\boldsymbol{v}}^\top\boldsymbol{y}_0, 
         \boldsymbol{y}_0 - \tilde{\boldsymbol{v}}^\top\boldsymbol{y}_0 \tilde{\boldsymbol{u}})
    = (\tilde{\boldsymbol{v}}^\top\boldsymbol{y}_0, 
         \boldsymbol{y}_0 
        - \tilde{\boldsymbol{u}}\tilde{\boldsymbol{v}}^\top\boldsymbol{y}_0 )    
    = (\tilde{\boldsymbol{v}}^\top\boldsymbol{y}_0,(\mathsf{I}_p - \tilde{\mathsf{\Pi}})\boldsymbol{y}_0).
    \]
Notice that the first SDE in \eqref{eq-system-DE} readily follows from \eqref{eq-SDE-M-rewritten}, and the second one can be checked as follows
    \begin{align*}
        \dif \boldsymbol{\mathcal{Q}}_t^{(\boldsymbol{y}_0)} &= \dif \boldsymbol{\mathcal{M}}_t^{(\boldsymbol{y}_0)} - \tilde{\boldsymbol{u}}\dif \mathcal{P}_t^{(\boldsymbol{y}_0)} = -\tilde{\boldsymbol{u}} \tilde{\boldsymbol{v}}^\top \mathsf{m}\boldsymbol{\alpha} \dif t + (\mathsf{I}_p - \tilde{\boldsymbol{u}} \tilde{\boldsymbol{v}}^\top) \dif \boldsymbol{\mathcal{M}}_t^{(\boldsymbol{y}_0)} \\
        &= -\tilde{\mathsf{\Pi}}\mathsf{m}\boldsymbol{\alpha}\dif t + (\mathsf{I}_p - \tilde{\mathsf{\Pi}}) \sqrt{(\mathcal{P}_t^{(\boldsymbol{y}_0)})^+((\mathsf{\Lambda}\tilde{\boldsymbol{u}})\odot\boldsymbol{\mathsf{\Sigma}})}\dif \boldsymbol{\mathcal{W}}_t , \qquad  t\in\R_+.
    \end{align*} 
    
   Conversely, if $((\mathcal{P}_t^{(p_0,\boldsymbol{q}_0)},\boldsymbol{\mathcal{Q}}_t^{(p_0,\boldsymbol{q}_0)}))_{t\in\R_+}$ is  a strong solution of  the SDE (\ref{eq-system-DE}) with initial value $(\mathcal{P}_0^{(p_0,\boldsymbol{q}_0)},\boldsymbol{\mathcal{Q}}_0^{(p_0,\boldsymbol{q}_0)}) = (p_0, \boldsymbol{q}_0)\in \R\times\R^p$,
     again by Itô's formula,
     the process
    \begin{equation*}
        \boldsymbol{\mathcal{M}}_t^{(p_0,\boldsymbol{q}_0)}
        := \mathcal{P}_t^{(p_0,\boldsymbol{q}_0)}\tilde{\boldsymbol{u}} + \boldsymbol{\mathcal{Q}}_t^{(p_0,\boldsymbol{q}_0)}, \qquad t\in\R_+,
    \end{equation*}
    is a strong solution of (\ref{eq-SDE-M-rewritten}) with initial value $\boldsymbol{\mathcal{M}}_0^{(p_0,\boldsymbol{q}_0)}=p_0 \Tilde{\boldsymbol{u}}+\boldsymbol{q}_0$.

Now, let us see that the map $\R^p \ni \boldsymbol{y}  \mapsto (\tilde{\boldsymbol{v}}^\top\boldsymbol{y},(\mathsf{I}_p - \tilde{\mathsf{\Pi}})\boldsymbol{y} )$ is a 
 bijection between $\R^p$ and 
 $\R\times\operatorname{Null} (\tilde{\boldsymbol{v}}^\top)$.
 Let $\boldsymbol{x}_0, \boldsymbol{y}_0 \in\R^p$ be two vectors with the same image under the map  in question. 
Then the injectivity (i.e. $\boldsymbol{x}_0 = \boldsymbol{y}_0$) follows from
 \begin{equation*}
 \begin{cases}
     \Tilde{\boldsymbol{v}}^\top\boldsymbol{x}_0 = \Tilde{\boldsymbol{v}}^\top\boldsymbol{y}_0 \\
     (\mathsf{I}_p - \Tilde{\mathsf{\Pi}}) \boldsymbol{x}_0 = (\mathsf{I}_p - \Tilde{\mathsf{\Pi}}) \boldsymbol{y}_0,
 \end{cases} \quad \text{ yielding that } \quad 
 \begin{cases}
     \Tilde{\boldsymbol{u}}(\Tilde{\boldsymbol{v}}^\top\boldsymbol{x}_0) = \Tilde{\boldsymbol{u}}(\Tilde{\boldsymbol{v}}^\top\boldsymbol{y}_0) \\
     \boldsymbol{x}_0 - \Tilde{\boldsymbol{u}}(\Tilde{\boldsymbol{v}}^\top\boldsymbol{x}_0) = \boldsymbol{y}_0 - \Tilde{\boldsymbol{u}}(\Tilde{\boldsymbol{v}}^\top\boldsymbol{y}_0),
 \end{cases}
 \end{equation*}
 where we have used  Hypothesis \ref{hyp-primitive-relaxed}.
Further, for all $ (p_0, \boldsymbol{q}_0)\in \R\times\operatorname{Null} (\tilde{\boldsymbol{v}}^\top)$, the vector $\boldsymbol{y}_0 \colonequals p_0 \Tilde{\boldsymbol{u}} + \boldsymbol{q}_0$ is an element of  the preimage  of $(p_0, \boldsymbol{q}_0)$,
 since, using that $\Tilde{\boldsymbol{v}}^\top \Tilde{\boldsymbol{u}}=1$,
 we have 
 ${\boldsymbol{v}}^\top(p_0 \Tilde{\boldsymbol{u}} + \boldsymbol{q}_0)= p_0+0=p_0$ and
 \[
   (\mathsf{I}_p - \Tilde{\mathsf{\Pi}})
      (p_0 \Tilde{\boldsymbol{u}} + \boldsymbol{q}_0)
       = p_0 \Tilde{\boldsymbol{u}} + \boldsymbol{q}_0
         - p_0 \Tilde{\boldsymbol{u}} (\Tilde{\boldsymbol{v}}^\top   
                  \Tilde{\boldsymbol{u}}) 
         - \Tilde{\boldsymbol{u}}\Tilde{\boldsymbol{v}}^\top \boldsymbol{q}_0
        = \boldsymbol{q}_0 -  \Tilde{\boldsymbol{u}} \cdot 0
        =\boldsymbol{q}_0 .
 \]
Hence the map in question is surjective.

All in all, it is enough to prove that SDE \eqref{eq-system-DE} has a pathwise unique strong solution for all initial values in $\R\times\R^p$, in particular for all initial values in $\R\times\operatorname{Null} (\tilde{\boldsymbol{v}}^\top)$, which is shown below.
 The pathwise uniqueness for the first SDE in \eqref{eq-system-DE}  (see also the SDE \eqref{SDE-1-dim-transform}) is clear by the discussion in Step 4 and Remark \ref{rem-unique}.
We now proceed with the second SDE in \eqref{eq-system-DE}. 
One can easily get that its pathwise unique strong solution with initial value $\boldsymbol{Q}_t^{(p_0,\boldsymbol{q}_0)} = \boldsymbol{q}_0$ 
takes the form
    \begin{equation*}
        \boldsymbol{Q}_t^{(p_0,\boldsymbol{q}_0)} = \boldsymbol{q}_0 -\tilde{\mathsf{\Pi}}\mathsf{m}\boldsymbol{\alpha} t + (\mathsf{I}_p - \tilde{\mathsf{\Pi}}) \sqrt{(\mathsf{\Lambda}\tilde{\boldsymbol{u}})\odot\boldsymbol{\mathsf{\Sigma}}}\int_0^t \sqrt{(\mathcal{P}_s^{(p_0)})^+} \dif \boldsymbol{\mathcal{W}}_s,
        \qquad t\in\R_+,
    \end{equation*}
    and not only for all $(p_0,\boldsymbol{q}_0)\in\R\times\operatorname{Null}(\tilde{\boldsymbol{v}}^\top)$ but also for all $(p_0,\boldsymbol{q}_0)\in\R\times\R^p$.    

\smallskip

In what follows, we check that the assumptions of Theorem \ref{thm-ispany-pap} hold with our previous choices.

\subsubsection*{Step 1/A}

For each $n,k\in\N$,  we get that 
 $\ev{\|n^{-1}\boldsymbol{M}_k\|^2}<\infty$, 
  since
    \begin{align*}
        \ev{\|\boldsymbol{M}_k\|^2}
         & =  \ev{\tr{\boldsymbol{M}_{k}(\boldsymbol{M}_{k})^\top}}
         = \tr{\var{\boldsymbol{M}_{k}}} = \tr{ \var{\evcond{\boldsymbol{M}_{k}}{\mathcal{F}_{k-1}}} +
               \ev{\varcond{\boldsymbol{M}_{k}}{\mathcal{F}_{k-1}}}} \\
         &=  \tr{ \ev{\varcond{\boldsymbol{M}_{k}}{\mathcal{F}_{k-1}}}} =  \tr{ \ev{\varcond{\boldsymbol{Z}_{k}}{\mathcal{F}_{k-1}}}}
            < \infty,
    \end{align*}
    where we used the variance decomposition formula and, at the last step, equation \eqref{eq-conditional-variance}.
    Further, $n^{-1}\bZ_0\stackrel{\mathcal{L}}{\longrightarrow} \boldsymbol{0}_p$
    as $n\to\infty$.

\subsubsection*{Step 1/B}
    
For all $T>0,$ the condition (i) of Theorem 
      \ref{thm-ispany-pap}  is trivially satisfied, since
      $(\boldsymbol{M}_{k})_{k\in\N}$ is a sequence of martingale differences 
      yielding that $\evcond{\boldsymbol{M}_{k}}{\mathcal{F}_{k-1}}=  \boldsymbol{0}_p$, $k\in\N$.
      
      The conditions (ii) and (iii) {of Theorem 
      \ref{thm-ispany-pap}} can be written in the following forms
    \begin{equation}\label{eq-our-cond-ii}
        \sup_{t\in [0,T]} \left\| \frac{1}{n^2} \sum_{k=1}^{\lfloor nt \rfloor} \evcond{\boldsymbol{M}_{k}(\boldsymbol{M}_{k})^\top}{\mathcal{F}_{k-1}} - \int_0^t (\boldsymbol{\mathcal{R}}_s^{(n)})^+ \dif s \odot \boldsymbol{\mathsf{\Sigma}} \right\| \stackrel{\mathrm{P}}{\longrightarrow} 0 \qquad \text{as } n\to\infty,
    \end{equation}
    \begin{equation}\label{lindeberg-condition}
        \frac{1}{n^2} \sum_{k=1}^{\lfloor nT \rfloor} \evcond{\|\boldsymbol{M}_k\|^2 \mathds{1}_{\{\|\boldsymbol{M}_k\|>n\theta\}}}{\mathcal{F}_{k-1}} \stackrel{\mathrm{P}}{\longrightarrow} 0 
        \qquad \text{as } n\to\infty \text{ for all} \;\theta >0,
    \end{equation}
    where the process $(\boldsymbol{\mathcal{R}}_t^{(n)})_{t\in\R_+}$ is defined by
    \begin{equation*}
        \boldsymbol{\mathcal{R}}_t^{(n)} \colonequals \mathsf{\Lambda}\tilde{\mathsf{\Pi}}(\boldsymbol{\mathcal{M}}_t^{(n)}+t\mathsf{m}\boldsymbol{\alpha}), \qquad t\in\R_+,\quad n\in\N.
    \end{equation*}

\subsubsection*{Step 1/C}
    
We show (\ref{eq-our-cond-ii}). Taking into account $\evcond{\boldsymbol{M}_{k}(\boldsymbol{M}_{k})^\top}{\mathcal{F}_{k-1}} = \varcond{\boldsymbol{Z}_{k}}{\mathcal{F}_{k-1}}$, $k\in\N$, equation (\ref{eq-conditional-variance}) and {Hypothesis \ref{hyp-g}}, we can obtain
    \begin{align}\label{eq-first-term-ii}
    \begin{split}
        \frac{1}{n^2} \sum_{k=1}^{\lfloor nt \rfloor} \evcond{\boldsymbol{M}_{k}(\boldsymbol{M}_{k})^\top}{\mathcal{F}_{k-1}} &=  \frac{1}{n^2} \Bigg( \lfloor nt \rfloor \boldsymbol{\alpha} + \sum_{k=1}^{\lfloor nt \rfloor} \left(\mathsf{\Lambda}\boldsymbol{Z}_{k-1} +  \boldsymbol{g}(\boldsymbol{Z}_{k-1})  \right) \Bigg) \odot\boldsymbol{\mathsf{\Sigma}}  \\
        &\quad + \mathsf{m} \Bigg( \frac{1}{n^2} \sum_{k=1}^{\lfloor nt \rfloor} \mathsf{\Gamma}(\boldsymbol{Z}_{k-1}) \Bigg) \mathsf{m}^{\top},
        \qquad t\in\R_+.
    \end{split}
    \end{align} 
    Using again Hypothesis \ref{hyp-g} together with equation (\ref{eq-conditional-expected-value}), the martingale differences in (\ref{eq-def-mtg-dif}) can be written as
    \begin{equation}\label{eq-explicit-mtg-diff}
        \boldsymbol{M}_{k} = \boldsymbol{Z}_{k} - \tilde{\mathsf{m}}\boldsymbol{Z}_{k-1} - \mathsf{m}\boldsymbol{\alpha} - \mathsf{m}\boldsymbol{g}(\boldsymbol{Z}_{k-1}),
        \qquad k\in\N.
    \end{equation}
With this expression and equation (\ref{eq-def-step-from-mtg}), we have
    \begin{align*}
        \boldsymbol{\mathcal{R}}_t^{(n)} &= \mathsf{\Lambda}\tilde{\mathsf{\Pi}} \left( \frac{1}{n} \Bigg( \boldsymbol{Z}_{0} + \sum_{k=1}^{\lfloor nt\rfloor} \Big( \boldsymbol{Z}_{k} - \tilde{\mathsf{m}}\boldsymbol{Z}_{k-1} - \mathsf{m}\boldsymbol{\alpha} - \mathsf{m}\boldsymbol{g}(\boldsymbol{Z}_{k-1})\Big) \Bigg) + t\mathsf{m}\boldsymbol{\alpha}\right) \nonumber \\
        &= \frac{1}{n} \mathsf{\Lambda}\tilde{\mathsf{\Pi}} \boldsymbol{Z}_{\lfloor nt\rfloor} + \frac{nt - \lfloor nt\rfloor}{n} \mathsf{\Lambda}\tilde{\mathsf{\Pi}}\mathsf{m}\boldsymbol{\alpha} - \frac{1}{n} \sum_{k=1}^{\lfloor nt\rfloor} \mathsf{\Lambda}\tilde{\mathsf{\Pi}}\mathsf{m} \boldsymbol{g}(\boldsymbol{Z}_{k-1}),
        \qquad t\in\R_+,\quad n\in\N,
    \end{align*}
    because $\tilde{\mathsf{\Pi}}\tilde{\mathsf{m}} 
    = (\lim_{l\to\infty} \tilde{\mathsf{m}}^l) \tilde{\mathsf{m}} 
      = \lim_{l\to\infty} \tilde{\mathsf{m}}^{l+1} = \tilde{\mathsf{\Pi}}$ by Hypothesis \ref{hyp-primitive-relaxed}.

  Since $\mathsf{\Lambda} \Tilde{\boldsymbol{u}}\in\R_+^p$, we have  
  $\mathsf{\Lambda}\tilde{\mathsf{\Pi}} = (\mathsf{\Lambda}\Tilde{\boldsymbol{u}})\Tilde{\boldsymbol{v}}^\top \in\R_+^{p\times p}$ and, 
 consequently, we get that
    \[
     {\boldsymbol{\mathcal{B}}_{- , t}^{(n)}
             \preceq  (\boldsymbol{\mathcal{R}}_t^{(n)})^+ 
             \preceq \boldsymbol{\mathcal{B}}_{+ , t}^{(n)}, 
             \qquad n\in\N, \quad t\in\R_+,}
    \]
   where the bounds $\boldsymbol{\mathcal{B}}_{- , t}^{(n)}$ and $\boldsymbol{\mathcal{B}}_{+ , t}^{(n)}$ are defined by 
    \begin{equation*}
        \boldsymbol{\mathcal{B}}_{\pm , t}^{(n)}  \colonequals  \frac{1}{n} \mathsf{\Lambda}\tilde{\mathsf{\Pi}} \boldsymbol{Z}_{\lfloor nt\rfloor}  \pm \frac{nt - \lfloor nt\rfloor}{n} |\mathsf{\Lambda}\tilde{\mathsf{\Pi}}\mathsf{m}\boldsymbol{\alpha}|  \pm \frac{1}{n} \sum_{k=1}^{\lfloor nt\rfloor} |\mathsf{\Lambda}\tilde{\mathsf{\Pi}}\mathsf{m} \boldsymbol{g}(\boldsymbol{Z}_{k-1})|,
    \end{equation*}
    and recall that $\vert \bz\vert = (\vert z_1\vert, \ldots, \vert z_p\vert)^\top\in\R_+^p$ for any $\bz=(z_1,\ldots,z_p)^\top\in\R^p$.
    Hence we get
    \begin{equation}\label{eq-bounds-of-integral}
        \int_0^t \boldsymbol{\mathcal{B}}_{- , s}^{(n)}\dif s 
        \preceq \int_0^t (\boldsymbol{\mathcal{R}}_s^{(n)})^+ \dif s 
        \preceq \int_0^t \boldsymbol{\mathcal{B}}_{+ , s}^{(n)} \dif s,
        \qquad t\in\R_+,\;\; n\in\N.
    \end{equation}
    
    Next, we check that, for all $t\in\R_+$,
    \begin{align}\label{eq-integration-of-bounds}
        \begin{split}
        \int_0^t \boldsymbol{\mathcal{B}}_{\pm , s}^{(n)} \dif s&=  \frac{1}{n^2} \sum_{j=0}^{\lfloor nt\rfloor -1} \mathsf{\Lambda}\tilde{\mathsf{\Pi}} \boldsymbol{Z}_{j} + \frac{nt - \lfloor nt \rfloor}{n^2} \mathsf{\Lambda}\tilde{\mathsf{\Pi}} \boldsymbol{Z}_{\lfloor nt\rfloor} 
          \pm \frac{\lfloor nt\rfloor + (nt - \lfloor nt\rfloor)^2}{2n^2} |\mathsf{\Lambda}\tilde{\mathsf{\Pi}}\mathsf{m}\boldsymbol{\alpha}| \\
        &\quad \pm \frac{1}{n^2} \sum_{j=0}^{\lfloor nt\rfloor -1} \sum_{k=1}^{j} |\mathsf{\Lambda}\tilde{\mathsf{\Pi}}\mathsf{m} \boldsymbol{g}(\boldsymbol{Z}_{k-1})| 
         \pm \frac{nt - \lfloor nt\rfloor}{n^2} \sum_{k=1}^{\lfloor nt\rfloor} |\mathsf{\Lambda}\tilde{\mathsf{\Pi}}\mathsf{m} \boldsymbol{g}(\boldsymbol{Z}_{k-1})| .
      \end{split}  
    \end{align}
  To verify \eqref{eq-integration-of-bounds}, we split the integral as follows
    \begin{equation*}
        \int_0^t \boldsymbol{\mathcal{B}}_{\pm , s}^{(n)} \dif s = \sum_{j=0}^{\lfloor nt\rfloor -1} \int_{\frac{j}{n}}^{\frac{j+1}{n}} \boldsymbol{\mathcal{B}}_{\pm , s}^{(n)} \dif s + \int_{\frac{\lfloor nt\rfloor}{n}}^{t} \boldsymbol{\mathcal{B}}_{\pm , s}^{(n)} \dif s,
    \end{equation*}
    and we take into account
    \begin{equation*}
        \int_{\frac{j}{n}}^{\frac{j+1}{n}} \boldsymbol{\mathcal{B}}_{\pm , s}^{(n)} \dif s = \frac{1}{n^2} \mathsf{\Lambda}\tilde{\mathsf{\Pi}} \boldsymbol{Z}_{j} \pm \frac{1}{2n^2} |\mathsf{\Lambda}\tilde{\mathsf{\Pi}}\mathsf{m}\boldsymbol{\alpha}| \pm \frac{1}{n^2} \sum_{k=1}^{j} |\mathsf{\Lambda}\tilde{\mathsf{\Pi}}\mathsf{m} \boldsymbol{g}(\boldsymbol{Z}_{k-1})|
    \end{equation*}
    {for $j\in\{0,\ldots,\lfloor nt\rfloor-1\}$,} and 
    \begin{equation*}
        \int_{\frac{\lfloor nt\rfloor}{n}}^{t} \boldsymbol{\mathcal{B}}_{\pm , s}^{(n)} \dif s = \frac{nt - \lfloor nt \rfloor}{n^2} \mathsf{\Lambda}\tilde{\mathsf{\Pi}} \boldsymbol{Z}_{\lfloor nt\rfloor} \pm \frac{(nt - \lfloor nt\rfloor)^2}{2n^2} |\mathsf{\Lambda}\tilde{\mathsf{\Pi}}\mathsf{m}\boldsymbol{\alpha}| \pm \frac{nt - \lfloor nt\rfloor}{n^2} \sum_{k=1}^{\lfloor nt\rfloor} |\mathsf{\Lambda}\tilde{\mathsf{\Pi}}\mathsf{m} \boldsymbol{g}(\boldsymbol{Z}_{k-1})|,
    \end{equation*}
    {where we used that 
    \[
        \int_{\frac{j}{n}}^{\frac{j+1}{n}} \frac{ns-j}{n}\dif s = \frac{1}{2n^2},
                     \qquad j\in\{0,\ldots,\lfloor nt\rfloor-1\},
    \]
    and
    \[
        \int_{\frac{\lfloor nt\rfloor}{n}}^t \frac{ns-\lfloor nt\rfloor}{n}\dif s
                 = \frac{(nt - \lfloor nt\rfloor)^2}{2n^2}.
    \]
    }

In what follows, we will use that for all $\mathsf{A}\in\R^{p\times p}$ and 
 $\boldsymbol{x},\boldsymbol{y}\in\R^p$, $\boldsymbol{y}_0\in\R_+^p$ with 
 $\boldsymbol{y} - \boldsymbol{y}_0 \preceq  \boldsymbol{x} \preceq  \boldsymbol{y} + \boldsymbol{y}_0$,
 we have 
 \begin{align}\label{help_A_Sigma}
    \Vert \mathsf{A} - \boldsymbol{x} \odot \boldsymbol{\mathsf{\Sigma}}\Vert
        \leq \Vert \mathsf{A} - \boldsymbol{y} \odot \boldsymbol{\mathsf{\Sigma}}\Vert
             + \Vert \boldsymbol{y_0} \Vert \Vert \boldsymbol{\mathsf{\Sigma}} \Vert,
 \end{align}
 where we recall $\|\boldsymbol{\mathsf{\Sigma}}\|  =  \sum_{i=1}^p \|\mathsf{\Sigma}_i\|$.
Indeed, using that $\Vert \boldsymbol{z}\odot \boldsymbol{\mathsf{\Sigma}}\Vert \leq \Vert \boldsymbol{z} \Vert \Vert \boldsymbol{\mathsf{\Sigma}} \Vert$ for all $\boldsymbol{z}\in\R^p$ (see \eqref{help_z_circ_norm}), we get 
  \begin{align*}
   \Vert \mathsf{A} - \boldsymbol{x} \odot \boldsymbol{\mathsf{\Sigma}}\Vert
     &\leq \Vert \mathsf{A} - \boldsymbol{y} \odot \boldsymbol{\mathsf{\Sigma}}\Vert 
           + \Vert (\boldsymbol{y} - \boldsymbol{x}) \odot \boldsymbol{\mathsf{\Sigma}}\Vert 
      = \Vert \mathsf{A} - \boldsymbol{y} \odot \boldsymbol{\mathsf{\Sigma}}\Vert 
           + \left\Vert \sum_{i=1}^p (y_i-x_i)\mathsf{\mathsf{\Sigma}}_i\right\Vert \\     
     &\leq \Vert \mathsf{A} - \boldsymbol{y} \odot \boldsymbol{\mathsf{\Sigma}}\Vert
           + \sum_{i=1}^p  \vert y_i-x_i\vert  \Vert \mathsf{\mathsf{\Sigma}}_i\Vert  
    \leq \Vert \mathsf{A} - \boldsymbol{y} \odot \boldsymbol{\mathsf{\Sigma}}\Vert
           + \Vert \boldsymbol{y}-\boldsymbol{x}\Vert \sum_{i=1}^p  \Vert \mathsf{\mathsf{\Sigma}}_i\Vert\\
     &\leq \Vert \mathsf{A} - \boldsymbol{y} \odot \boldsymbol{\mathsf{\Sigma}}\Vert
            + \Vert \boldsymbol{y}_0 \Vert  \Vert \boldsymbol{\mathsf{\Sigma}}\Vert,
  \end{align*}
  as desired.
Hence, from equations (\ref{eq-first-term-ii}), (\ref{eq-bounds-of-integral}), 
  (\ref{eq-integration-of-bounds}) and \eqref{help_A_Sigma}, we obtain  
    \begin{align}\label{help_Step1C}
     \begin{split}
        & \left\| \frac{1}{n^2} \sum_{k=1}^{\lfloor nt \rfloor} \evcond{\boldsymbol{M}_{k}(\boldsymbol{M}_{k})^\top}{\mathcal{F}_{k-1}} - \int_0^t (\boldsymbol{\mathcal{R}}_s^{(n)})^+ \dif s \odot \boldsymbol{\mathsf{\Sigma}} \right\| \\
        & \quad \leq \left\| \frac{1}{n^2} \sum_{k=0}^{\lfloor nt \rfloor -1} \mathsf{\Lambda}(\mathsf{I}_p-\tilde{\mathsf{\Pi}})\boldsymbol{Z}_{k} \odot \boldsymbol{\mathsf{\Sigma}} \right\| + \left\| \frac{nt-\lfloor nt \rfloor}{n^2} \mathsf{\Lambda}\tilde{\mathsf{\Pi}}\boldsymbol{Z}_{\lfloor nt \rfloor} \odot \boldsymbol{\mathsf{\Sigma}} \right\| \\
        &\qquad + \left\| \frac{\lfloor nt\rfloor }{n^2} \boldsymbol{\alpha} \odot \boldsymbol{\mathsf{\Sigma}}\right\| + \left\| \frac{\lfloor nt\rfloor + (nt - \lfloor nt\rfloor)^2}{2n^2} |\mathsf{\Lambda}\tilde{\mathsf{\Pi}}\mathsf{m}\boldsymbol{\alpha}| \right\| \|\boldsymbol{\mathsf{\Sigma}}\| \\
        &\qquad + \left\|\frac{1}{n^2} \sum_{k=1}^{\lfloor nt\rfloor} \boldsymbol{g}(\boldsymbol{Z}_{k-1}) \odot \boldsymbol{\mathsf{\Sigma}} \right\| + 
        \left\| \mathsf{m} \Bigg( \frac{1}{n^2} \sum_{k=1}^{\lfloor nt \rfloor} \mathsf{\Gamma}(\boldsymbol{Z}_{k-1}) \Bigg) \mathsf{m}^{\top} \right\| \\
        &\qquad +  \left\| \frac{1}{n^2} \sum_{j=0}^{\lfloor nt\rfloor -1} \sum_{k=1}^{j} |\mathsf{\Lambda}\tilde{\mathsf{\Pi}}\mathsf{m} \boldsymbol{g}(\boldsymbol{Z}_{k-1})|  \right\| \Vert \boldsymbol{\mathsf{\Sigma}}  \Vert + 
        \left\| \frac{nt - \lfloor nt\rfloor}{n^2} \sum_{k=1}^{\lfloor nt\rfloor} |\mathsf{\Lambda}\tilde{\mathsf{\Pi}}\mathsf{m} \boldsymbol{g}(\boldsymbol{Z}_{k-1})| \right\|
        \Vert \boldsymbol{\mathsf{\Sigma}} \Vert.
       \end{split} 
    \end{align}
    
    As a consequence of Hypothesis \ref{hyp-g}, the function $\boldsymbol{g}$ 
    is bounded, and
    $\Vert | \mathsf{A}\bz | \Vert = \Vert  \mathsf{A}\bz \Vert 
      \leq \Vert \mathsf{A} \Vert \Vert \bz\Vert$ for all $\mathsf{A}\in\R^{p\times p}$ and $\bz\in\R^p$. Thus,
    in order to show (\ref{eq-our-cond-ii}), we can easily see that it is enough to prove that 
    \begin{align}
        \sup_{t\in [0,T]} \frac{1}{n^2} \sum_{k=0}^{\lfloor nt \rfloor -1} \|(\mathsf{I}_p-\tilde{\mathsf{\Pi}})\boldsymbol{Z}_{k}\| &\stackrel{\mathrm{P}}{\longrightarrow} 0, \label{eq-cond-aux-1}\\
        \sup_{t\in [0,T]} \frac{1}{n^2} \|\boldsymbol{Z}_{\lfloor nt \rfloor}\| &\stackrel{\mathrm{P}}{\longrightarrow} 0, \label{eq-cond-aux-2}\\
        \sup_{t\in [0,T]} \frac{1}{n^2} \sum_{k=1}^{\lfloor nt \rfloor} \| \mathsf{\Gamma}(\boldsymbol{Z}_{k-1})\| &\stackrel{\mathrm{P}}{\longrightarrow} 0 \label{eq-cond-aux-gamma},\\
        \sup_{t\in [0,T]} \frac{1}{n^2} \sum_{j=0}^{\lfloor nt\rfloor -1} \sum_{k=1}^{j} \|\boldsymbol{g}(\boldsymbol{Z}_{k-1})\|  &\stackrel{\mathrm{P}}{\longrightarrow} 0, \label{eq-cond-aux-g}
    \end{align}
    as $n\to \infty$ for all $T>0$.
     Indeed, for all $T>0$, the supremum on $[0,T]$ of the fifth and eight terms on the right hand side of \eqref{help_Step1C} tend to $0$ as $n\to\infty$ in probability, since
 \begin{align*}
 \sup_{t\in [0,T]}
  \left\| \frac{nt - \lfloor nt\rfloor}{n^2} \sum_{k=1}^{\lfloor nt\rfloor} |\boldsymbol{g}(\boldsymbol{Z}_{k-1})| \right\|
  \leq \sup_{t\in [0,T]} 
        \left\| \frac{1}{n^2} \sum_{k=1}^{\lfloor nt\rfloor}  |\boldsymbol{g}(\boldsymbol{Z}_{k-1})| \right\|
  \leq \frac{\lfloor nT\rfloor}{n^2}
        \sup_{\bz\in\Z_+^p} \Vert \boldsymbol{g}(\bz) \Vert
  =\OO\left(\frac{1}{n}\right)
 \end{align*}
 as  $n\to\infty$.
 We remark that, later, at the end of the proof of Step 3, it will turn out that 
 $n^{-1} \sum_{k=1}^{\lfloor nT\rfloor} \ev{\|\boldsymbol{g}
 (\boldsymbol{Z}_{k-1})\|} \to 0$ as $n\to\infty$ for all $T>0$ holds as well.
 
     Let us start proving (\ref{eq-cond-aux-1}) and (\ref{eq-cond-aux-2}).
    From equation (\ref{eq-explicit-mtg-diff}), the following expression can be obtained recursively,
    \begin{equation}\label{help_Z_n_k}
        \boldsymbol{Z}_{k} = \tilde{\mathsf{m}}^k \boldsymbol{Z}_{0} + \sum_{j=1}^k \tilde{\mathsf{m}}^{k-j} \left( \boldsymbol{M}_{j} + \mathsf{m}\boldsymbol{\alpha} + \mathsf{m}\boldsymbol{g}(\boldsymbol{Z}_{j-1})\right),
        \qquad k\in\N,
    \end{equation}
    and, using it with $\tilde{\mathsf{\Pi}}\tilde{\mathsf{m}}^k 
     = \tilde{\mathsf{\Pi}}$, ${k\in\Z_+}$, we have
    \begin{equation*}
        (\mathsf{I}_p-\tilde{\mathsf{\Pi}})\boldsymbol{Z}_{k} = (\tilde{\mathsf{m}}^k -\tilde{\mathsf{\Pi}})  \boldsymbol{Z}_{0} + \sum_{j=1}^k (\tilde{\mathsf{m}}^{k-j} -\tilde{\mathsf{\Pi}})  \left( \boldsymbol{M}_{j} + \mathsf{m}\boldsymbol{\alpha} + \mathsf{m}\boldsymbol{g}(\boldsymbol{Z}_{j-1})\right)
    \end{equation*}
    for $k\in\N$.

    Let $\tilde{c}\in\R_{++}$ and $\tilde{r}\in(0,1)$ be the constants given in part \ref{hyp-pi-tilde} of Hypothesis \ref{hyp-primitive-relaxed} such that $\|\tilde{\mathsf{m}}^k-\tilde{\mathsf{\Pi}}\|\leq \tilde{c}\tilde{r}^k$ for each $k\in\N$, and let us introduce the constants
    \begin{align}
        \tilde{C} & \colonequals  \sup_{j\in\Z_+} \|\tilde{\mathsf{m}}^j\| \leq \tilde{c} + \|\tilde{\mathsf{\Pi}}\|, \label{eq-C-tilde} \\
        C_{\boldsymbol{g}} & \colonequals  \sup_{\boldsymbol{z}\in\Z_+^p} \|\boldsymbol{g}(\boldsymbol{z})\| < \infty . \label{eq-C-g}
    \end{align}
    
    Then, using \eqref{help_Z_n_k}, the following inequalities hold
    \begin{align*}
        \|\boldsymbol{Z}_{\lfloor nt \rfloor}\| 
        &\leq \|\tilde{\mathsf{m}}^{\lfloor nt \rfloor} \| \|\boldsymbol{Z}_{0} \| + \sum_{j=1}^{\lfloor nt \rfloor} \|\tilde{\mathsf{m}}^{\lfloor nt \rfloor-j}\| \Big( \|\boldsymbol{M}_{j}\| + \|\mathsf{m}\|\big(\|\boldsymbol{\alpha}\| + \|\boldsymbol{g}(\boldsymbol{Z}_{j-1})\|\big)\Big)\nonumber \\
        &\leq \tilde{C} \left( \|\boldsymbol{Z}_{0}\| + \lfloor nt \rfloor\|\mathsf{m}\|\big(\|\boldsymbol{\alpha}\| + C_{\boldsymbol{g}}\big) + \sum_{j=1}^{\lfloor nt \rfloor} \|\boldsymbol{M}_{j}\|\right),
        \qquad n\in\N,\quad t\in\R_+,
    \end{align*}
    and
    \begin{align*}
        \sum_{k=0}^{\lfloor nt \rfloor -1} \|(\mathsf{I}_p-\tilde{\mathsf{\Pi}})\boldsymbol{Z}_{k}\| &\leq \tilde{c} \sum_{k=0}^{\lfloor nt \rfloor -1} \tilde{r}^k \|\boldsymbol{Z}_{0}\| + \tilde{c} \sum_{k=0}^{\lfloor nt \rfloor -1} \sum_{j=1}^k \tilde{r}^{k-j} \Big( \|\boldsymbol{M}_{j}\| + \|\mathsf{m}\|\big(\|\boldsymbol{\alpha}\| + \|\boldsymbol{g}(\boldsymbol{Z}_{j-1})\|\big)\Big) \\
        &\leq \frac{\tilde{c}}{1-\tilde{r}} \left( \|\boldsymbol{Z}_{0}\| + \lfloor nt \rfloor\|\mathsf{m}\|\big(\|\boldsymbol{\alpha}\| + C_{\boldsymbol{g}}\big) + \sum_{j=1}^{\lfloor nt \rfloor}  \|\boldsymbol{M}_{j}\|\right),\qquad n\in\N, \quad t\in\R_+,
    \end{align*}
    where at the last inequality we also used that 
    \begin{align*}
      \sum_{k=0}^{\lfloor nt \rfloor-1} \sum_{j=1}^k \tilde r^{k-j} x_j 
                 & = \sum_{j=1}^{\lfloor nt \rfloor-1} \sum_{k=j}^{\lfloor nt \rfloor-1}
                      \tilde r^{k-j} x_j 
                  \leq \sum_{j=1}^{\lfloor nt \rfloor-1} x_j \sum_{k=j}^\infty 
                  \tilde r^{k-j} 
                  = \frac{1}{1-\tilde{r}} \sum_{j=1}^{\lfloor nt \rfloor-1} x_j \\
                 & \leq \frac{1}{1-\tilde{r}} \sum_{j=1}^{\lfloor nt \rfloor} x_j
                   \qquad \text{for all } \boldsymbol{x} = (x_1,\ldots,x_{\lfloor nt \rfloor})^\top 
                         \in\R_+^{\lfloor nt \rfloor}.
    \end{align*}
  Hence, to get (\ref{eq-cond-aux-1}) and (\ref{eq-cond-aux-2}) it is enough to show 
    \begin{align}\label{help_seged1}
        \frac{1}{n^2} \|\boldsymbol{Z}_{0}\| \stackrel{\mathrm{P}}{\longrightarrow} 0, \qquad \frac{1}{n^2} \sum_{k=1}^{\lfloor nT \rfloor}  \|\boldsymbol{M}_{k}\|
        \stackrel{\mathrm{P}}{\longrightarrow} 0 \qquad \text{as } n\to\infty,
    \end{align}
    for all $T>0$.
    
    The first convergence in \eqref{help_seged1} holds trivially, since,
    in fact, $n^{-2}\Vert \bZ_0\Vert\to 0$ as $n\to\infty$ almost surely.
    
     The second convergence in \eqref{help_seged1} follows from \eqref{eq-O-mart-diff-1},
     since,  by Jensen inequality, $\ev{\|\boldsymbol{M}_{k}\|} \leq \sqrt{\ev{\|\boldsymbol{M}_{k}\|^2} } = \OO(k^{1/2})$ as $k\to\infty$, and hence, for all $T>0$, we get that
     \begin{equation*}
         \sum_{k=1}^{\lfloor nT \rfloor} \ev{\|\boldsymbol{M}_{k}\|}  
         = \OO(n^{3/2}) \qquad \text{as } n\to\infty.
     \end{equation*}
  Consequently, $n^{-2} \sum_{k=1}^{\lfloor nT \rfloor}  \|\boldsymbol{M}_{k}\|$
     converges to $0$ in $L^1$ as $n\to\infty$.

 Next, we check  convergence (\ref{eq-cond-aux-gamma}). 
It is enough to prove that for all $T>0$, we have that
 \begin{align}\label{help_Gamma_L1}
  \frac{1}{n^2} \sum_{k=1}^{\lfloor nT \rfloor} \ev{\| \mathsf{\Gamma}(\boldsymbol{Z}_{k-1})\|}
  \to 0 \qquad  \text{as $n\to\infty$.}  
  \end{align}
By Hypothesis \ref{hyp-variance-control-o}, for all $\epsilon>0$ there exists $K(\epsilon)>0$ such that $\|\mathsf{\Gamma}(\boldsymbol{z})\| < \epsilon \|\boldsymbol{z}\|$ for each $\boldsymbol{z}\in\Z_+^p$ with $\|\boldsymbol{z}\|>K(\epsilon)$. 
Since
    \begin{equation*}
        \ev{\| \mathsf{\Gamma}(\boldsymbol{Z}_{k-1})\|} 
        = \ev{\| \mathsf{\Gamma}(\boldsymbol{Z}_{k-1})\|\left(\1_{\{\|\boldsymbol{Z}_{k-1}\|\leq K(\epsilon)\}} + \1_{\{\|\boldsymbol{Z}_{k-1}\|> K(\epsilon)\}}\right)},
        \qquad k\in\N,
    \end{equation*}
   for all $T>0$, $\epsilon>0$ and $n\in\N$, we have
    \begin{align*}
        \frac{1}{n^2} \sum_{k=1}^{\lfloor nT \rfloor} \ev{\| \mathsf{\Gamma}(\boldsymbol{Z}_{k-1})\|} 
        & \leq \frac{1}{n^2} \sum_{k=1}^{\lfloor nT \rfloor}
         \left(\max_{\{\boldsymbol{z}\in\Z_+^p : \|\boldsymbol{z}\| \leq K(\epsilon) \}} \| \mathsf{\Gamma}(\boldsymbol{z})\| + \epsilon \ev{\|\boldsymbol{Z}_{k-1}\|}\right) \\
        &\leq \frac{T}{n} \max_{\{\boldsymbol{z}\in\Z_+^p : \|\boldsymbol{z}\| \leq   
                                      K(\varepsilon) \}} \|\mathsf{\Gamma}(\boldsymbol{z})\|
                 + \frac{\epsilon}{n^2}\sum_{k=1}^{\lfloor nT \rfloor} \ev{\|
                         \boldsymbol{Z}_{k-1}\|},
    \end{align*}
    where $ \max_{\{\boldsymbol{z}\in\Z_+^p : \|\boldsymbol{z}\| \leq   
    K(\varepsilon) \}} \|\mathsf{\Gamma}(\boldsymbol{z})\| < \infty$.
    Using part (i) of Lemma \ref{lemma-moments}, we get 
    \[
    \frac{1}{n^2}\sum_{k=1}^{\lfloor nT \rfloor} \ev{\|
                         \boldsymbol{Z}_{k-1}\|} = \OO(1)
                         \qquad \text{as } n\to\infty.
    \]
    Hence for all $T>0$ and $\epsilon>0$, we get 
    \[
     \lim_{n\to\infty} 
     \left( \frac{T}{n} \max_{\{\boldsymbol{z}\in\Z_+^p : \|\boldsymbol{z}\| \leq   
                                      K(\varepsilon) \}} \|\mathsf{\Gamma}(\boldsymbol{z})\|
                 + \frac{\epsilon}{n^2}\sum_{k=1}^{\lfloor nT \rfloor} \ev{\|
                         \boldsymbol{Z}_{k-1}\|} \right)
                =0,
    \]
  which yields \eqref{help_Gamma_L1}, as desired.
    
    Similarly, we prove (\ref{eq-cond-aux-g}) by checking 
    that for all $T>0$, we get that
    \begin{align}\label{help_g_L1}
      \frac{1}{n^2} \sum_{j=1}^{\lfloor nT\rfloor} \sum_{k=0}^{j-1} \ev{\|\boldsymbol{g}(\boldsymbol{Z}_{k})\|}  \to 0
      \qquad \text{as } n\to\infty.
    \end{align}
      By Hypothesis \ref{hyp-g}, for all $\epsilon>0$ there exists $N(\epsilon)>0$ such that $\|\boldsymbol{g}(\boldsymbol{z})\| < \epsilon$ for all $\boldsymbol{z}\in\Z_+^p$ with $\|\boldsymbol{z}\|>N(\epsilon)$. 
      Then, using the notation $C_{\boldsymbol{g}}$ introduced in \eqref{eq-C-g}, we get
    \begin{align}\label{help_g_L1_2}
        \ev{\| \boldsymbol{g}(\boldsymbol{Z}_{k})\|} &= \ev{\| \boldsymbol{g}(\boldsymbol{Z}_{k})\|\left(\1_{\{\|\boldsymbol{Z}_{k}\|\leq N(\epsilon)\}} + \1_{\{\|\boldsymbol{Z}_{k}\|> N(\epsilon)\}}\right)} \nonumber\\
        & \leq C_{\boldsymbol{g}} \prob{\|\boldsymbol{Z}_{k}\|\leq N(\epsilon)} + \epsilon , \qquad k\in\N, 
    \end{align}
    and then using Hypothesis \ref{hyp-explosion}  we can derive that
    for all $T>0$ and $\epsilon>0$,
    \begin{align}\label{hyp-explosion-use-1}
       \begin{split}
        \sum_{j=1}^{\lfloor nT\rfloor} \sum_{k=0}^{j-1} \ev{\|\boldsymbol{g}(\boldsymbol{Z}_{k})\|} &
           = \sum_{k=0}^{\lfloor nT\rfloor- 1} \sum_{j=k+1}^{\lfloor nT\rfloor} 
             \ev{\|\boldsymbol{g}(\boldsymbol{Z}_{k})\|} \\ 
        &\leq \epsilon \lfloor nT\rfloor^2 
         + C_{\boldsymbol{g}} \sum_{k=0}^{\lfloor nT\rfloor-1}  (\lfloor nT\rfloor -k)\prob{\|\boldsymbol{Z}_{k}\|\leq N(\epsilon)} \\
        & \leq \epsilon \lfloor nT\rfloor^2 
         + C_{\boldsymbol{g}} \sum_{k=0}^{k_0 (\varepsilon,N(\epsilon)) -1}  (\lfloor nT\rfloor -k)\prob{\|\boldsymbol{Z}_{k}\|\leq N(\epsilon)}\\
        & \quad + C_{\boldsymbol{g}} \sum_{k=k_0 (\epsilon,N(\epsilon))}^{\lfloor nT\rfloor -1}  (\lfloor nT\rfloor -k)\epsilon
      \end{split}  
    \end{align}
   for sufficiently large $n\in\N$, where
   \[
    \sum_{k=0}^{k_0 (\epsilon,N(\epsilon)) -1}  (\lfloor nT\rfloor -k)\prob{\|\boldsymbol{Z}_{k}\|\leq N(\epsilon)}
        \leq k_0 (\epsilon,N(\epsilon)) \lfloor nT\rfloor,
   \]
  and 
  \[
    \sum_{k=k_0 (\epsilon,N(\epsilon))}^{\lfloor nT\rfloor -1}  (\lfloor nT\rfloor -k)\epsilon \leq \epsilon n^2 T^2.
  \]
 As a consequence, for all $T>0$ and $\epsilon>0$, we have 
  \[
   \limsup_{n\to\infty} \frac{1}{n^2} \sum_{j=1}^{\lfloor nT\rfloor} \sum_{k=0}^{j-1} \ev{\|\boldsymbol{g}(\boldsymbol{Z}_{k})\|}
    \leq \epsilon T^2 + \epsilon C_{\boldsymbol{g}} T^2.
  \]
  Hence, by taking the limit as  $\epsilon\downarrow 0$, we obtain \eqref{help_g_L1}, as desired.

\subsubsection*{Step 1/D}
   
We show (\ref{lindeberg-condition}).
 For this, it is enough to verify that for all $T>0$ and $\theta>0$, we have
 \begin{align}\label{help_Step1D}
  \frac{1}{n^2} \sum_{k=1}^{\lfloor nT\rfloor}
     \ev{\|\boldsymbol{M}_k\|^2 \mathds{1}_{\{\|\boldsymbol{M}_k
       \|>n\theta\}}} \to 0   \qquad \text{as } n\to\infty.
 \end{align}
Using that
\begin{equation*}
     \ev{\|\boldsymbol{M}_k\|^2 \mathds{1}_{\{\|\boldsymbol{M}_k\|>n\theta\}}} \leq \frac{1}{n^2 \theta^2} \ev{\|\boldsymbol{M}_k\|^4},
    \qquad n,k\in\N,
\end{equation*}
the convergence \eqref{help_Step1D} follows from equation \eqref{eq-O-mart-diff-4}, since 
\begin{equation*}
   \frac{1}{n^2} \sum_{k=1}^{\lfloor nT\rfloor}
     \ev{\|\boldsymbol{M}_k\|^2 \mathds{1}_{\{\|\boldsymbol{M}_k
       \|>n\theta\}}}
    \leq \frac{1}{\theta^2 n^4} \sum_{k=1}^{\lfloor nT \rfloor} \ev{\|\boldsymbol{M}_k\|^4}\to 0 \qquad \text{as } n\to\infty .
\end{equation*}
Consequently, we finished the proof of \eqref{eq-conv-M}.

\subsection*{Proof of Step 2}

We want to apply Theorem \ref{cont-map-thm} to prove the convergence \eqref{help_psi_n_to_psi}.
   We need to check that the assumptions of Theorem \ref{cont-map-thm} are satisfied.
 The continuity of $\boldsymbol{\Psi}$ can be straightforwardly deduced 
 (following also from Jacod and Shiryaev \cite[Chapter VI, Proposition 1.23]{JACOD_SHIRYAEV_2003}), 
  so the measurability of $\boldsymbol{\Psi}$ holds. For the sequence $(\boldsymbol{\Psi}_n)_{n\in\N}$ and $N\in\N$, let us introduce the 
  localized sequence $(\boldsymbol{\Psi}_n^N)_{n\in\N}$ given by
  $\boldsymbol{\Psi}_n^N : \mathbf{D}(\R_+,\R^p)\to\mathbf{D}(\R_+,\R^p)$, $\boldsymbol{\Psi}_n^N (\boldsymbol{f}) (t)  \colonequals  \boldsymbol{\Psi}_n (\boldsymbol{f}) (t \wedge N)$ for $\boldsymbol{f}\in \mathbf{D}(\R_+,\R^p)$, $t\in\R_+$, $n\in\N$.
Since, for each $n\in\N$ and $\boldsymbol{f}\in\mathbf{D}(\R_+,\R^p)$,  
 $\boldsymbol{\Psi}_n^N(\boldsymbol{f}) \to \boldsymbol{\Psi}_n(\boldsymbol{f})$ as $N\to\infty$, it is enough to check the measurability of $\boldsymbol{\Psi}_n^N$, $n\in\N$ (see Barczy et al.\ \cite[page 603]{BARCZY_ISPANY_PAP_2011} for the details). 
Briefly, for each $n\in\N$, one can introduce the auxiliary measurable mappings $\boldsymbol{\Psi}_n^{N,1} :\mathbf{D}(\R_+,\R^p)\to (\R^p)^{nN+1}$ and  $\boldsymbol{\Psi}_n^{N,2} :(\R^p)^{nN+1}\to \mathbf{D}(\R_+,\R^p)$ defined by
    \begin{align*}
        \boldsymbol{\Psi}_n^{N,1} (\boldsymbol{f}) & \colonequals\left(\boldsymbol{f}\left(0\right), \boldsymbol{f}\left(\frac{1}{n}\right),\boldsymbol{f}\left(\frac{2}{n}\right),\ldots,\boldsymbol{f}\left(N\right)\right)
    \end{align*}
    for $\boldsymbol{f}\in  \mathbf{D}(\R_+,\R^p)$, and
     \begin{align*}
        \boldsymbol{\Psi}_n^{N,2}(\boldsymbol{x}_0, \boldsymbol{x}_1,\ldots,\boldsymbol{x}_{nN})(t) &\colonequals \tilde{\mathsf{m}}^{\lfloor n(t\wedge N)\rfloor}\boldsymbol{x}_0 + \sum_{j=1}^{\lfloor n(t\wedge N)\rfloor} \tilde{\mathsf{m}}^{\lfloor n(t\wedge N)\rfloor - j} \left(\boldsymbol{x}_j - \boldsymbol{x}_{j-1} + \frac{1}{n}\mathsf{m}\boldsymbol{\alpha}\right)
    \end{align*}
    for $(\boldsymbol{x}_0, \boldsymbol{x}_1,\ldots,\boldsymbol{x}_{nN})\in(\R^p)^{nN+1}$ and $t\in\R_+$.
    Then we have that $\boldsymbol{\Psi}_n^N 
     =  \boldsymbol{\Psi}_n^{N,2} \circ\boldsymbol{\Psi}_n^{N,1}$, $n,N\in\N$.

   Consider the set $\mathbf{C}=\{\boldsymbol{f}\in\mathbf{C}(\R_+,\R^p):\tilde{\mathsf{\Pi}}\boldsymbol{f}(0)=\boldsymbol{f}(0)\}$.   
   Let us check now that $\mathbf{C}$ is measurable and $\prob{\boldsymbol{\mathcal{M}}\in\mathbf{C}}=1.$
The projection $\mathbf{D}(\R_+,\R^p)\ni \boldsymbol{f} \mapsto \boldsymbol{\pi}_0 (\boldsymbol{f}) \colonequals \boldsymbol{f}(0)\in\R^p$ and the mapping $\R^p \ni \boldsymbol{x} \mapsto (\mathsf{I}_p - \Tilde{\mathsf{\Pi}})\boldsymbol{x}\in\R^p$ are measurable.
Since $\mathbf{C}(\R_+,\R^p)$ is a measurable set (see Ethier and Kurtz \cite[Problem 3.11.25]{EthKur}), we have $\mathbf{C} = \mathbf{C}(\R_+,\R^p) \cap \boldsymbol{\pi}_0^{-1} (\operatorname{Null}(\mathsf{I}_p-\Tilde{\mathsf{\Pi}}))\in\mathcal{D}_\infty (\R_+,\R^p)$.
Furthermore, in Step 1, we proved that $\boldsymbol{\mathcal{M}}$ is the pathwise unique strong solution of SDE \eqref{eq-SDE-M} with initial value $\boldsymbol{0}_p$. Hence it has continuous sample paths almost surely, and 
   $\tilde{\mathsf{\Pi}}\boldsymbol{\mathcal{M}}_0= \tilde{\mathsf{\Pi}}\boldsymbol{0}_p=\boldsymbol{0}_p=\boldsymbol{\mathcal{M}}_0$.
   Finally, the procedure to show that $\mathbf{C}\subset\mathbf{C}_{\boldsymbol{\Phi},(\boldsymbol{\Phi}_n)_{n\in\N}}$ follows the same steps as on page 736
    in Isp\'any and Pap \cite{ispany-pap-2014}.

\subsection*{Proof of Step 3}

First, we check that 
 $\boldsymbol{\mathcal{Z}}^{(n)} = \boldsymbol{\Psi}_n(\boldsymbol{\mathcal{M}}^{(n)})  + \boldsymbol{\mathcal{V}}^{(n)}$, $n\in\N$.
Using  Hypothesis \ref{hyp-g} and equations \eqref{eq-conditional-expected-value}, 
 \eqref{eq-def-mtg-dif}, \eqref{help_psi_n}, and \eqref{help_V}, we get
 \begin{align*}
    (\boldsymbol{\Psi}_n(\boldsymbol{\mathcal{M}}^{(n)}))_t
     &= \tilde{\mathsf{m}}^{\lfloor nt\rfloor} 
              \frac{1}{n}\boldsymbol{Z}_0   
             + \sum_{j=1}^{\lfloor nt \rfloor} \tilde{\mathsf{m}}^{\lfloor nt\rfloor -j}  
                 \Big(\frac{1}{n} \boldsymbol{M}_j + \frac{1}{n}\mathsf{m} \boldsymbol{\alpha} \Big) \\
    & = \frac{1}{n} \tilde{\mathsf{m}}^{\lfloor nt\rfloor} 
              \boldsymbol{Z}_0   
             + \frac{1}{n} \sum_{j=1}^{\lfloor nt \rfloor} \tilde{\mathsf{m}}^{\lfloor nt\rfloor -j}  
                 \big( \boldsymbol{Z}_j  - \mathsf{m}\varepsilon(\boldsymbol{Z}_{j-1}) 
                        + \mathsf{m}\boldsymbol{\alpha} \big)  \\             
    &=  \frac{1}{n} \tilde{\mathsf{m}}^{\lfloor nt\rfloor} \boldsymbol{Z}_0  
      + \frac{1}{n}\sum_{j=1}^{\lfloor nt \rfloor} \tilde{\mathsf{m}}^{\lfloor nt\rfloor -j} 
          \Big( \boldsymbol{Z}_j - \tilde{\mathsf{m}} \boldsymbol{Z}_{j-1} 
                     -\mathsf{m}  g(\boldsymbol{Z}_{j-1}) \Big)  \\
    &= n^{-1} \boldsymbol{Z}_{\lfloor nt \rfloor} - n^{-1} \sum_{j=1}^{\lfloor nt\rfloor} \tilde{\mathsf{m}}^{\lfloor nt\rfloor - j} \mathsf{m}\boldsymbol{g}(\boldsymbol{Z}_{j-1})\\
   &= n^{-1} \boldsymbol{Z}_{\lfloor nt \rfloor} - \boldsymbol{\mathcal{V}}_t^{(n)}
     = \boldsymbol{\mathcal{Z}}^{(n)}_t - \boldsymbol{\mathcal{V}}_t^{(n)},
     \qquad  n\in\N, \quad t\in\R_+,
\end{align*}
 as desired.
 
Next, we show \eqref{eq-conv-sum-seq}. 
Taking into account \eqref{help_psi_n_to_psi} and Jacod and Shiryaev \cite[Chapter VI, Lemma 3.31]{JACOD_SHIRYAEV_2003}, in order to prove \eqref{eq-conv-sum-seq},
 it is enough to see that for all $T>0$ and $\delta>0$, we get
\begin{equation}\label{help_cV_conv}
    \lim_{n\to\infty} \prob{\sup_{t\in[0,T]} \|\boldsymbol{\mathcal{V}}_t^{(n)}\|\geq\delta} = 0.
\end{equation}

This can be checked as follows.
For all $T>0$ and $\delta>0$, by Markov's inequality and \eqref{eq-C-tilde}, we get
    \begin{align*}
        \prob{\sup_{t\in[0,T]} \|\boldsymbol{\mathcal{V}}_t^{(n)}\|\geq\delta} &\leq \delta^{-1} \ev{n^{-1} \sum_{j=1}^{\lfloor nT\rfloor} \|\tilde{\mathsf{m}}^{\lfloor nT\rfloor - j} \mathsf{m}\boldsymbol{g}(\boldsymbol{Z}_{j-1}) \|} \\
        & \leq \tilde{C}\|\mathsf{m}\| \delta^{-1} n^{-1} \sum_{j=1}^{\lfloor nT\rfloor}  \ev{\|\boldsymbol{g}( \boldsymbol{Z}_{j-1}) \|},
        \qquad n\in\N,
    \end{align*}
    and then we proceed as in the proof of (\ref{eq-cond-aux-g}). 
    For all $\epsilon>0$, there exists $N(\epsilon)>0$ such that $\|\boldsymbol{g}(\boldsymbol{z})\| < \epsilon$ for each $\boldsymbol{z}\in\Z_+^p$ with $\|\boldsymbol{z}\|>N(\epsilon)$. 
    So, using Hypothesis \ref{hyp-explosion} and \eqref{help_g_L1_2}, we get
    for all $T>0$ and $\epsilon>0$
    \begin{align}\label{hyp-explosion-use-2}
    \begin{split}
        \frac{1}{n} \sum_{j=1}^{\lfloor nT\rfloor} \ev{\|\boldsymbol{g}(\boldsymbol{Z}_{j-1})\|} & \leq \frac{1}{n} \sum_{j=1}^{\lfloor nT\rfloor} \left(C_{\boldsymbol{g}} \prob{\|\boldsymbol{Z}_{j-1}\|\leq N(\epsilon)} + \epsilon\right) \\
        & \leq \epsilon T + \frac{C_{\boldsymbol{g}}}{n} \left(\sum_{j=0}^{k_0 (\epsilon,N(\epsilon)) -1}  \prob{\|\boldsymbol{Z}_{j}\|\leq N(\epsilon)} + \sum_{j=k_0 (\epsilon,N(\epsilon))}^{\lfloor nT \rfloor -1}  \epsilon \right)  \\
        &\leq \epsilon T + C_{\boldsymbol{g}} k_0 (\epsilon,N(\epsilon)) 
                                \frac{1}{n} + \epsilon  C_{\boldsymbol{g}} T
    \end{split}
    \end{align}
    for sufficiently large $n\in\N$.
    By taking the limit as $n\to\infty$, we have that for all $T>0$ and $\epsilon>0$,
    \[
      \limsup_{n\to\infty} \frac{1}{n} \sum_{j=1}^{\lfloor nT\rfloor} \ev{\|\boldsymbol{g}
                        (\boldsymbol{Z}_{j-1})\|} 
        \leq  \epsilon T + \epsilon C_{\boldsymbol{g}} T,   
    \]
    and then taking the limit as $\epsilon \downarrow 0$, we obtain that 
    $n^{-1} \sum_{j=1}^{\lfloor nT\rfloor} \ev{\|\boldsymbol{g}
    (\boldsymbol{Z}_{j-1})\|} \to 0$ as $n\to\infty$,
    yielding \eqref{help_cV_conv}, as desired.
\proofend

\appendix
\section{Appendix}\label{appendix}

The reader can consult Horn and Johnson \cite{horn_johnson_1985}
 for the following known facts about primitive non-negative matrices (see Definition 8.5.0 and Theorems 8.2.8, 8.5.1 and 8.5.2  in \cite{horn_johnson_1985}). 
A primitive matrix $\mathsf{A}\in\R_+^{p\times p}$ is an irreducible matrix with only one eigenvalue of maximum modulus (the so-called Perron--Frobenius eigenvalue), or equivalently, $\mathsf{A}\in\R_{+}^{p\times p}$ is primitive  if and only if there exists $n\in\N$ such that $\mathsf{A}^n\in\R_{++}^{p\times p}.$ By the Perron--Frobenius theorem, we have the following lemma.

\begin{lemma}\label{lemma-primitive}
    Let $\mathsf{A}\in\R_{+}^{p\times p}$ be a primitive matrix and let $\rho$ be its Perron--Frobenius eigenvalue. The following statements hold:
    \begin{enumerate}[label=(\roman*)]
        \item $\rho\in\R_{++},$ its algebraic and geometric multiplicity are equal to 1, and the absolute values of the other eigenvalues of $\mathsf{A}$ are less than $\rho$.
        \item There exists a unique right eigenvector $\boldsymbol{u}\in\R_{++}^p$ and a unique left eigenvector $\boldsymbol{v}\in\R_{++}^p$ corresponding to $\rho$ such that the sum of the coordinates of $\boldsymbol{u}$ is 1 and $\boldsymbol{v}^\top \boldsymbol{u} = 1$.
        One calls $\boldsymbol{u}$ and $\boldsymbol{v}$ the right and left
         Perron--Frobenius eigenvector, respectively.
        \item If $\rho =1,$ then $\lim_{k\to\infty} \mathsf{A}^k = \mathsf{\Pi}$ and there exist $c\in\R_{++}$ and $r\in(0,1)$ such that $\|\mathsf{A}^k-\mathsf{\Pi}\|\leq cr^k$ for each $k\in\N$, where $\mathsf{\Pi} \colonequals \boldsymbol{u}\boldsymbol{v}^\top\in\R_{++}^{p\times p}$.
    \end{enumerate}
\end{lemma}

Next, we will investigate the asymptotic behaviour of the first and second moments 
 of the norm  $\Vert\bZ_k\Vert$ for a critical CMBP $(\bZ_k)_{k\in\Z_+}$. 
We will also study the same question for the  second and fourth moments of 
 $\|\boldsymbol{M}_k\|$, where $(\boldsymbol{M}_k)_{k\in\N}$ is the martingale 
 difference sequence, given in \eqref{eq-def-mtg-dif}, 
 built from $(\bZ_k)_{k\in\Z_+}$.
We need the following auxiliary result,  presented and proved first. 

\begin{lemma}\label{lemma-aux}
 Let $A$, $A_n$, $n\in\N$, be independent and identically distributed $\Z_+$--valued random variables with zero mean. 
If $B$ is a $\Z_+$--valued random variable independent of $A$ and $A_n$, $n\in\N$, then
    \begin{equation*}
        \ev{\left(\sum_{i=1}^B A_i\right)^4} = 3 \Sigma_A^2 (\Gamma_B + \mu_B^2) + (\zeta_A- 3 \Sigma_A^2)\mu_B,
    \end{equation*}
    where $\Sigma_A \colonequals \var{A}$, $\zeta_A \colonequals \ev{A^4}$, $\mu_B \colonequals \ev{B}$, 
    and $\Gamma_B \colonequals \var{B}$.
\end{lemma}

\begin{proof}
    By the properties of the conditional expectation, we can get
    \begin{equation*}
        \ev{\left(\sum_{i=1}^B A_i\right)^4} = \ev{\evcond{\sum_{i,j,k,l=1}^B A_iA_jA_kA_l}{B}} = \ev{\sum_{i=1}^B \ev{A_i^4}} + 3\ev{\sum_{\substack{i,j=1 \\ i\neq j}}^B \ev{A_i^2A_j^2}},
    \end{equation*}
    where the terms that would correspond to $\ev{A_i^3A_j}$, $\ev{A_i^2A_jA_k}$ and $\ev{A_iA_jA_kA_l}$ (indices $i,j,k,l$ are different from each other) vanish due to independence and the zero mean of the random variables in question.

Let us calculate the summands on the right hand side of the expression
 above:
    \begin{equation*}
        \ev{\sum_{i=1}^B \ev{A_i^4}} = \zeta_A \mu_B,
    \end{equation*}
  and
    \begin{equation*}
        \ev{\sum_{\substack{i,j=1 \\ i\neq j}}^B \ev{A_i^2A_j^2}} = \Sigma_A^2 \ev{B(B-1)} = \Sigma_A^2 (\Gamma_B + \mu_B^2-\mu_B).
    \end{equation*}
This yields the assertion.
\end{proof}

\begin{lemma}\label{lemma-moments}
    Let $(\boldsymbol{Z}_k)_{k\in\Z_+}$ be a controlled $p$--type branching process given in \eqref{CMBP_def} such that
    $\ev{\|\bZ_0\|}$, $\ev{\|\boldsymbol{X}_{0,1,i}\|}$ and 
    $\ev{\|\boldsymbol{\phi}_{0}(\boldsymbol{z})\|}$ are finite for $i\in\{1,\ldots,p\}$ and $\bz\in\Z_+^p$.
    Assume that    
    $\boldsymbol{\varepsilon}(\boldsymbol{z}) = \mathsf{\Lambda}\boldsymbol{z} + \boldsymbol{h}(\boldsymbol{z})$, $ \bz\in\Z_+^p$, where $\mathsf{\Lambda}\in\R^{p\times p}$ and $\boldsymbol{h}:\Z_+^p\to\R^p$  satisfies $\|\boldsymbol{h}(\boldsymbol{z})\| = \OO(1)$ as $\|\boldsymbol{z}\|\to\infty$. 
    Further, suppose that Hypothesis \ref{hyp-primitive-relaxed} holds.
    \begin{enumerate}[label=(\roman*)]
        \item  Then we have
        \begin{equation}\label{eq-big-O-first-moment}
            \ev{\|\boldsymbol{Z}_{k}\|} = \OO(k) \qquad \text{as $k\to\infty$.} 
        \end{equation}
        \item If, in addition, $\ev{\|\bZ_0\|^2}$, $\ev{\|\boldsymbol{X}_{0,1,i}\|^2}$ and $\ev{\|\boldsymbol{\phi}_{0}(\boldsymbol{z})\|^2}$ are finite for $i\in\{1,\ldots,p\}$, $\bz\in\Z_+^p$, and $\|\mathsf{\Gamma}(\boldsymbol{z})\|= \OO(\|\boldsymbol{z}\|)$ as $\Vert \bz\Vert\to\infty$, then we have \eqref{eq-big-O-first-moment} and
        \begin{align}
            \ev{\|\boldsymbol{Z}_{k}\|^2} &= \OO(k^2) \qquad \text{as } k\to\infty, \label{eq-big-O-second-moment} \\
            \ev{\|\boldsymbol{M}_{k}\|^2} &= \OO(k) \qquad \text{as } k\to\infty. \label{eq-O-mart-diff-1}
        \end{align}
        \item If, in addition, Hypotheses \ref{hyp-second-moment-Z0}, \ref{hyp-moment-4} and
        $\|\mathsf{\Gamma}(\boldsymbol{z})\|= \OO(\|\boldsymbol{z}\|)$ as $\Vert \bz\Vert\to\infty$ hold, then we have \eqref{eq-big-O-first-moment}, \eqref{eq-big-O-second-moment}, \eqref{eq-O-mart-diff-1} and
        \begin{equation}\label{eq-O-mart-diff-4}
            \ev{\|\boldsymbol{M}_k\|^4} = \OO(k^2) \qquad \text{as } k\to\infty .
        \end{equation}
    \end{enumerate}
\end{lemma}

\begin{proof}
 (i). By \eqref{eq-conditional-expected-value}, for each $k\in\N$, we have
    \begin{equation}\label{help_Exp_Z_k_rec}
        \ev{\boldsymbol{Z}_{k}} = \ev{\mathsf{m}\boldsymbol{\varepsilon}(\boldsymbol{Z}_{k-1})} = \tilde{\mathsf{m}}\ev{\boldsymbol{Z}_{k-1}} + \mathsf{m}\ev{\boldsymbol{h}(\boldsymbol{Z}_{k-1})} = \tilde{\mathsf{m}}^{k}\ev{\boldsymbol{Z}_{0}} + \sum_{j=0}^{k-1} \tilde{\mathsf{m}}^{j}\mathsf{m}\ev{\boldsymbol{h}(\boldsymbol{Z}_{k-1-j})}.
    \end{equation}
Hence, using the triangle inequality, for each $k\in\N$, we get
    \begin{equation*}
        \left\|\ev{\boldsymbol{Z}_{k}}\right\| \leq \|\tilde{\mathsf{m}}^{k}\|\ev{\|\boldsymbol{Z}_{0}\|} + \sum_{j=0}^{k-1} \|\tilde{\mathsf{m}}^{j}\|\|\mathsf{m}\|\ev{\|\boldsymbol{h}(\boldsymbol{Z}_{k-1-j})\|}.
    \end{equation*} 
We recall the power mean inequality used in the proofs several times:
 for all $n\in\N$, $a_i\in\R_+$, $i=1,\ldots,n$, and $0<k_1 \leq k_2$, 
 we get
 \begin{align}\label{mean_power_ineq}
  \left(\frac{1}{n}\sum_{i=1}^n a_i^{k_1}\right)^{\frac{1}{k_1}} 
     \leq \left(\frac{1}{n}\sum_{i=1}^n a_i^{k_2}\right)^{\frac{1}{k_2}}. 
\end{align}
 Using the inequality 
 \[
 \Vert \boldsymbol{x}\Vert
     \leq \sqrt{(x_1+\ldots+x_p)^2 + \cdots + (x_1+\ldots+x_p)^2}
     = \sqrt{p}\sum_{i=1}^p x_i,
     \qquad \boldsymbol{x} = (x_1,\ldots,x_p)\in\R_+^p,
 \]  
 and the power mean inequality \eqref{mean_power_ineq}, 
  for any random vector $\boldsymbol{\xi}=(\xi_1,\ldots,\xi_p)$ having 
 non-negative coordinates, we obtain that 
 \begin{align*}
   \ev{\|\boldsymbol{\xi}\|} 
     \leq  \sqrt{p}\ev{\xi_1+\ldots+\xi_p} 
     \leq  \sqrt{p}\cdot p
         \left(\frac{(\ev{\xi_1})^2 + \ldots + (\ev{\xi_p})^2 }{p}\right)^{\frac{1}{2}}
     =p \Vert \ev{\boldsymbol{\xi}}\Vert.    
 \end{align*}
Consequently, we get
 \begin{equation*}
        \ev{\|\boldsymbol{Z}_{k}\|}         
       \leq p \left\|\ev{\boldsymbol{Z}_{k}}\right\| 
        = p \tilde{C}(\ev{\|\boldsymbol{Z}_{0}\|}+\|\mathsf{m}\|C_{\boldsymbol{h}}k),
    \end{equation*}
    where $\tilde{C}=\sup_{j\in\Z_+} \|\tilde{\mathsf{m}}^j\|<\infty$ 
     due to Hypothesis \ref{hyp-primitive-relaxed} (see \eqref{eq-C-tilde})
    and the constant $C_{\boldsymbol{h}}$ is defined by
    \begin{equation}\label{eq-C-h}
        C_{\boldsymbol{h}}  \colonequals  \sup_{\boldsymbol{z}\in\Z_+^p} \|\boldsymbol{h}(\boldsymbol{z})\| < \infty .
    \end{equation}
This yields \eqref{eq-big-O-first-moment}.
    
 (ii). 
  Since the finiteness of the second moment of the norm of a random vector
 implies that of the first moment, using part (i) of the present lemma,
 we have \eqref{eq-big-O-first-moment}.
Taking into account that the expected value and the trace of a random square matrix commute, we have
    \begin{align*}
        \ev{\|\boldsymbol{Z}_{k}\|^2} &= \ev{\tr{\boldsymbol{Z}_{k}(\boldsymbol{Z}_{k})^\top}}
         = \tr{\ev{\bZ_k(\bZ_k)^\top}} = \tr{\ev{\boldsymbol{Z}_{k}}\ev{\boldsymbol{Z}_{k}}^\top} + \tr{\var{\boldsymbol{Z}_{k}}}  \\
        & = \left\| \ev{\bZ_k} \right\|^2 + \tr{\var{\boldsymbol{Z}_{k}}} \leq \left(\ev{\|\boldsymbol{Z}_{k}\|}\right)^2 + p \left\| \var{\boldsymbol{Z}_{k}} \right\| ,
    \end{align*}
     where for the last inequality, we used that for any matrix
      $\mathsf{B}\in\R^{p\times p}$,
     we have $\tr{\mathsf{B}} = \sum_{i=1}^p b_{i,i} \leq \sum_{i=1}^p \vert b_{i,i}\vert
     \leq p\Vert \mathsf{B}\Vert$.
    We know that (\ref{eq-big-O-first-moment}) holds, therefore, in order to get (\ref{eq-big-O-second-moment}), it is enough to see $\left\| \var{\boldsymbol{Z}_{k}} \right\| = \OO(k^2)$ as $k\to\infty$.
  Using the variance decomposition formula, i.e.
    \begin{equation*}
        \var{\boldsymbol{Z}_{k}} = \var{\evcond{\boldsymbol{Z}_{k}}{\mathcal{F}_{k-1}}} + \ev{\varcond{\boldsymbol{Z}_{k}}{\mathcal{F}_{k-1}}}, \qquad k\in\N,
    \end{equation*}
  formulas (\ref{eq-conditional-expected-value}) and (\ref{eq-conditional-variance}),
 and the assumption on $\boldsymbol{\varepsilon}(\bz)$, $\bz\in\Z_+^p$, together with the properties of variance, for each $k\in\N$, we obtain
    \begin{align*}
        \var{\boldsymbol{Z}_{k}}
        & = \var{\mathsf{m}\boldsymbol{\varepsilon}(\boldsymbol{Z}_{k-1})}
                  + \ev{\boldsymbol{\varepsilon}(\boldsymbol{Z}_{k-1})
                         \odot\boldsymbol{\mathsf{\Sigma}}
                        +  \mathsf{m} \mathsf{\Gamma}(\boldsymbol{Z}_{k-1}) \mathsf{m}^\top } \\
        & =
          \var{ \mathsf{m} \mathsf{\Lambda} \boldsymbol{Z}_{k-1}
                + \mathsf{m} \boldsymbol{h}(\boldsymbol{Z}_{k-1}) }
                + \ev{\boldsymbol{\varepsilon}(\boldsymbol{Z}_{k-1})}\odot\boldsymbol{\mathsf{\Sigma}} + \mathsf{m} \ev{\mathsf{\Gamma}(\boldsymbol{Z}_{k-1})}\mathsf{m}^{\top} \\
        &= \tilde{\mathsf{m}}\var{\boldsymbol{Z}_{k-1}}\tilde{\mathsf{m}}^{\top} + \mathsf{m} \var{\boldsymbol{h}(\boldsymbol{Z}_{k-1})}\mathsf{m}^{\top} + \tilde{\mathsf{m}}\cov{\boldsymbol{Z}_{k-1}}{\boldsymbol{h}(\boldsymbol{Z}_{k-1})}\mathsf{m}^{\top} \\
        &\quad + \mathsf{m}\cov{\boldsymbol{h}(\boldsymbol{Z}_{k-1})}{\boldsymbol{Z}_{k-1}}\tilde{\mathsf{m}}^{\top}  + \ev{\boldsymbol{\varepsilon}(\boldsymbol{Z}_{k-1})}\odot\boldsymbol{\mathsf{\Sigma}} + \mathsf{m} \ev{\mathsf{\Gamma}(\boldsymbol{Z}_{k-1})}\mathsf{m}^{\top}.
    \end{align*}
Proceeding recursively, we can reach the following expression
    \begin{align*}
    \var{\boldsymbol{Z}_{k}} &= \tilde{\mathsf{m}}^{k}\var{\boldsymbol{Z}_{0}}(\tilde{\mathsf{m}}^{\top})^{k}
     + \sum_{j=0}^{k-1} \tilde{\mathsf{m}}^{j} \mathsf{m} \var{\boldsymbol{h}(\boldsymbol{Z}_{k-1-j})}\mathsf{m}^{\top}  (\tilde{\mathsf{m}}^{\top})^{j}  \\
    &\quad + \sum_{j=0}^{k-1} \tilde{\mathsf{m}}^{j} \tilde{\mathsf{m}}\cov{\boldsymbol{Z}_{k-1-j}}{\boldsymbol{h}(\boldsymbol{Z}_{k-1-j})}\mathsf{m}^{\top} (\tilde{\mathsf{m}}^{\top})^{j}  \\
    &\quad + \sum_{j=0}^{k-1} \tilde{\mathsf{m}}^{j} \mathsf{m}\cov{\boldsymbol{h}(\boldsymbol{Z}_{k-1-j})}{\boldsymbol{Z}_{k-1-j}}\tilde{\mathsf{m}}^{\top} (\tilde{\mathsf{m}}^{\top})^{j}  \\
    &\quad + \sum_{j=0}^{k-1} \tilde{\mathsf{m}}^{j} \left( \ev{\boldsymbol{\varepsilon}(\boldsymbol{Z}_{k-1-j})}\odot\boldsymbol{\mathsf{\Sigma}}\right) (\tilde{\mathsf{m}}^{\top})^{j}
     + \sum_{j=0}^{k-1} \tilde{\mathsf{m}}^{j} \mathsf{m} \ev{\mathsf{\Gamma}(\boldsymbol{Z}_{k-1-j})}\mathsf{m}^{\top} (\tilde{\mathsf{m}}^{\top})^{j}.
    \end{align*}
In what follows, we will use that for all $\bz=(z_1,\ldots,z_p)\in\R^p$, we have
 \begin{align}\label{help_z_circ_norm}
   \Vert \bz \odot \boldsymbol{\mathsf{\Sigma}} \Vert
    = \left\Vert \sum_{i=1}^p z_i \mathsf{\Sigma}_i\right\Vert
    \leq \Vert \bz\Vert \sum_{i=1}^p \Vert \mathsf{\Sigma}_i\Vert
    = \Vert \bz\Vert \|\boldsymbol{\mathsf{\Sigma}}\|,
 \end{align}
 where $\|\boldsymbol{\mathsf{\Sigma}}\|  \colonequals  \sum_{i=1}^p \|\mathsf{\Sigma}_i\|$.
By \eqref{help_z_circ_norm}, the triangle inequality  and the symmetry of covariance, we get
    \begin{align}\label{help_1}
    \begin{split}
    \left\| \var{\boldsymbol{Z}_{k}} \right\| & \leq  \|\tilde{\mathsf{m}}^{k}\|^2  \left\|\var{\boldsymbol{Z}_{0}}\right\|  + \sum_{j=0}^{k-1} \|\tilde{\mathsf{m}}^{j}\|^2  \|\mathsf{m}\|^2  \left\|\var{\boldsymbol{h}(\boldsymbol{Z}_{k-1-j})}\right\|  \\
    & +  2 \sum_{j=0}^{k-1} \|\tilde{\mathsf{m}}^{j}\|^2  \|\tilde{\mathsf{m}}\|
         \|\mathsf{m}\|  \left\| \cov{\boldsymbol{Z}_{k-1-j}}{\boldsymbol{h}(\boldsymbol{Z}_{k-1-j})}\right\|   \\
    & + \sum_{j=0}^{k-1} \|\tilde{\mathsf{m}}^{j}\|^2  \ev{\|\boldsymbol{\varepsilon}(\boldsymbol{Z}_{k-1-j})\|}\|\boldsymbol{\mathsf{\Sigma}}\|  + \sum_{j=0}^{k-1} \|\tilde{\mathsf{m}}^{j}\|^2  \|\mathsf{m}\|^2  \ev{\|\mathsf{\Gamma}(\boldsymbol{Z}_{k-1-j})\|},
    \end{split}
    \end{align}
    and now we look for an upper bound for each term on the right hand side
    of \eqref{help_1}.

    In case of the first term, we easily have
    \begin{equation*}
        \left\|\var{\boldsymbol{Z}_{0}}\right\|
         =\left\Vert \ev{\bZ_0(\bZ_0)^\top}
                           - \ev{\bZ_0} \ev{(\bZ_0)^\top}
            \right\Vert \leq \ev{\|\boldsymbol{Z}_{0}\|^2}
            + \big(\ev{\|\boldsymbol{Z}_{0}\|}\big)^2
          <\infty.
    \end{equation*}

In case of the second term,
  since $\|\boldsymbol{h}(\boldsymbol{z})\| = \OO(1)$ as $\Vert \bz\Vert\to\infty$,  by \eqref{eq-C-h}, we get
    \begin{equation*}
        \left\|\var{\boldsymbol{h}(\boldsymbol{Z}_{k})}\right\| \leq \ev{\|\boldsymbol{h}(\boldsymbol{Z}_{k})\|^2}
        + \left(\ev{\|\boldsymbol{h}(\boldsymbol{Z}_{k})\|}\right)^2 = \OO(1)
        \qquad \text{as } k\to\infty.
    \end{equation*}

In case of the third term, using 
 (\ref{eq-big-O-first-moment}) and \eqref{eq-C-h}, we get
    \begin{align*}
        \left\| \cov{\boldsymbol{Z}_{k}}{\boldsymbol{h}(\boldsymbol{Z}_{k})}\right\| &\leq \ev{\|\boldsymbol{Z}_{k}\|  \|\boldsymbol{h}(\boldsymbol{Z}_{k})\|} + \ev{\|\boldsymbol{Z}_{k}\|}\ev{\|\boldsymbol{h}(\boldsymbol{Z}_{k})\|} = \OO(k) \qquad \text{as } k\to\infty.
     \end{align*}
     
In case of the fourth term, by the assumption $\boldsymbol{\varepsilon}(\boldsymbol{z}) = \mathsf{\Lambda}\boldsymbol{z} + \boldsymbol{h}(\boldsymbol{z})$, $\bz\in\Z_+^p$, we have
\begin{equation*}
        \ev{\|\boldsymbol{\varepsilon}(\boldsymbol{Z}_{k})\|}  \leq \|\mathsf{\Lambda}\|\ev{\|\boldsymbol{Z}_{k}\|} + C_{\boldsymbol{h}},
        \qquad k\in\N,
\end{equation*}
    where the constant $C_{\boldsymbol{h}}$ is defined in  \eqref{eq-C-h}.  Using \eqref{eq-big-O-first-moment}, this implies that
     \begin{equation}\label{epsilon_exp_asymp}
    \ev{\|\boldsymbol{\varepsilon}(\boldsymbol{Z}_{k})\|} = \OO(k)  \qquad \text{as }   k\to\infty.
    \end{equation}
    
In case of the  fifth term,  by the assumption
 $\|\mathsf{\Gamma}(\boldsymbol{z})\|=\OO(\|\boldsymbol{z}\|)$ as $\Vert \bz\Vert\to\infty$
  and the equation \eqref{eq-big-O-first-moment}, we get
    \begin{equation}\label{Gamma_exp_asymp}
        \ev{\|\mathsf{\Gamma}(\boldsymbol{Z}_{k})\|}
        \leq  \|\mathsf{\Gamma}(\boldsymbol{0}_p)\|  + C_{\mathsf{\Gamma}}\ev{\|\boldsymbol{Z}_{k}\|} = \OO(k) \qquad \text{as } k\to\infty ,
    \end{equation}
    where
    \begin{equation*}
        C_{\mathsf{\Gamma}} = \sup_{\boldsymbol{z}\in\Z_+^p \setminus\{\boldsymbol{0}_p\}} \|\boldsymbol{z}\|^{-1} \|\mathsf{\Gamma}(\boldsymbol{z})\| <\infty.
    \end{equation*}

 Taking into account the previous estimations and that
  $\sup_{j\in\Z_+}\|\tilde{\mathsf{m}}^j\|<\infty$
   due to Hypothesis \eqref{hyp-primitive-relaxed} (see \eqref{eq-C-tilde}),
  the inequality \eqref{help_1} implies that
    \begin{equation*}
        \left\| \var{\boldsymbol{Z}_{k}} \right\|
        = \OO(1) +  \OO(k)+ \sum_{j=0}^{k-1} \OO(k-1-j) = \OO(k^2)
         \qquad \text{as } k\to\infty ,
    \end{equation*}
    which concludes the proof of \eqref{eq-big-O-second-moment}.

Next, we verify \eqref{eq-O-mart-diff-1}.
Then
    \begin{align}\label{exp_M_ineq}
        \ev{\|\boldsymbol{M}_{k}\|^2}={\ev{\tr{\boldsymbol{M}_{k}(\boldsymbol{M}_{k})^\top}}} = {\tr{\ev{\boldsymbol{M}_{k}(\boldsymbol{M}_{k})^\top}}},\qquad k\in\N,
    \end{align}
 where, by the tower rule, we have $\ev{\boldsymbol{M}_{k}(\boldsymbol{M}_{k})^\top} = \ev{\varcond{\boldsymbol{Z}_{k}}{\mathcal{F}_{k-1}}}$. Consequently,  using (\ref{eq-conditional-variance}), \eqref{help_z_circ_norm} and the inequality $\tr{\mathsf{B}} \leq p\Vert \mathsf{B}\Vert$ for any matrix
      $\mathsf{B}\in\R^{p\times p}$  (justified earlier),
  we get that
\begin{align}\label{tr_M_exp}
        {\tr{\ev{\boldsymbol{M}_{k}(\boldsymbol{M}_{k})^\top}}}
        &= \tr{\ev{\boldsymbol{\varepsilon}(\boldsymbol{Z}_{k-1})}\odot\boldsymbol{\mathsf{\Sigma}}}
        + \tr{\mathsf{m} \ev{\mathsf{\Gamma}(\boldsymbol{Z}_{k-1})}  \mathsf{m}^{\top}} \nonumber\\
        &\leq {p\|\ev{\boldsymbol{\varepsilon}(\boldsymbol{Z}_{k-1})}\odot\boldsymbol{\mathsf{\Sigma}}\|} + {p\|\mathsf{m} \ev{\mathsf{\Gamma}(\boldsymbol{Z}_{k-1})}  \mathsf{m}^{\top}\|}\nonumber \\
        & \leq {p\ev{\|\boldsymbol{\varepsilon}(\boldsymbol{Z}_{k-1})\|}\|\boldsymbol{\mathsf{\Sigma}}\|} + {p\ev{\|\mathsf{\Gamma}(\boldsymbol{Z}_{k-1})\|} \|\mathsf{m}\|^{2}},
        \qquad k\in\N.
    \end{align}
  Therefore, \eqref{eq-O-mart-diff-1}
 follows from   \eqref{epsilon_exp_asymp}, \eqref{Gamma_exp_asymp}, \eqref{exp_M_ineq} and \eqref{tr_M_exp}.

    (iii). Consider the following reformulation of the martingale difference sequence $(\boldsymbol{M}_k)_{k\in\N}$ defined in \eqref{eq-def-mtg-dif}:
\begin{equation*}
    \boldsymbol{M}_k = \sum_{i=1}^p \left(\sum_{j=1}^{\phi_{k-1,i}(\boldsymbol{Z}_{k-1})} \boldsymbol{X}_{k-1,j,i} - \varepsilon_i(\boldsymbol{Z}_{k-1})\boldsymbol{m}_i \right),
    \qquad {k\in\N,}
\end{equation*}
 where we used \eqref{eq-conditional-expected-value}.
 
Applying the power mean inequality \eqref{mean_power_ineq} twice,
 and adding and subtracting  the random variable
 $\sum_{j=1}^{\phi_{k-1,i}(\boldsymbol{Z}_{k-1})}  \boldsymbol{m}_{i}  = \phi_{k-1,i}(\boldsymbol{Z}_{k-1}) \boldsymbol{m}_{i}$, we get
\begin{align*}
    \|\boldsymbol{M}_k\|^4 
    &\leq \left( \sum_{i=1}^p  \left \Vert  \sum_{j=1}^{\phi_{k-1,i}(\boldsymbol{Z}_{k-1})} \boldsymbol{X}_{k-1,j,i} - \varepsilon_i(\boldsymbol{Z}_{k-1})\boldsymbol{m}_i \right \Vert \right)^4 \\ 
    &\leq p^3 \sum_{i=1}^p \left\| \sum_{j=1}^{\phi_{k-1,i}(\boldsymbol{Z}_{k-1})} \boldsymbol{X}_{k-1,j,i} - \varepsilon_i(\boldsymbol{Z}_{k-1})\boldsymbol{m}_i \right\|^4 \\
    & \leq 8p^3 \sum_{i=1}^p \left(\left\| \sum_{j=1}^{\phi_{k-1,i}(\boldsymbol{Z}_{k-1})} (\boldsymbol{X}_{k-1,j,i} -\boldsymbol{m}_i) \right\|^4 + \left\| \big(\phi_{k-1,i}(\boldsymbol{Z}_{k-1}) - \varepsilon_i(\boldsymbol{Z}_{k-1}) \big) \boldsymbol{m}_i \right\|^4\right)
\end{align*}
 for $k\in\N$.
Using again the power mean inequality \eqref{mean_power_ineq} 
 and that $\|\boldsymbol{z}\|^4 = \left(z_1^2 + \ldots + z_p^2\right)^2$, 
 $\bz=(z_1,\ldots,z_p)\in\Z_+^p$, we get
\begin{align*}
    \ev{\|\boldsymbol{M}_k\|^4} &\leq 8p^4 \sum_{i=1}^p \sum_{l=1}^p \ev{ \left( \sum_{j=1}^{\phi_{k-1,i}(\boldsymbol{Z}_{k-1})} (X_{k-1,j,i,l} -m_{i,l})\right)^4 }\\
    & \quad + 8p^4 \sum_{i=1}^p \sum_{l=1}^p \ev{\left(\bigl(\phi_{k-1,i}(\boldsymbol{Z}_{k-1}) - \varepsilon_i(\boldsymbol{Z}_{k-1}) \bigr) m_{i,l} \right)^4}.
\end{align*}
We will compute the previous upper bounds by first determining 
 the corresponding conditional expectations  with respect to $\mathcal{F}_{k-1}$.
Using the notations \eqref{notats_moments_3}  and the Markov property of $(\bZ_k)_{k\in\Z_+}$, for each $k\in\N$, $i,l\in\{1,\ldots,p\}$, we obtain
\begin{equation*}
    \evcond{\left(\phi_{k-1,i}(\boldsymbol{Z}_{k-1}) - \varepsilon_i(\boldsymbol{Z}_{k-1}) \right)^4 m_{i,l} ^4}{\mathcal{F}_{k-1}} = m_{i,l}^4 \kappa_i(\boldsymbol{Z}_{k-1}),
\end{equation*}
and, by  \eqref{notats_moments_2}, \eqref{notats_moments_3} and 
 Lemma \ref{lemma-aux}  together with the independence of 
 $\boldsymbol{\phi}_{k-1}(\boldsymbol{z})$, $\bz\in\Z_+^p$,
 $\mathbf{X}_{k-1,j,i}$, $j\in\N$, $i\in\{1,\ldots,p\}$ and $\bZ_{k-1}$,
 we get that
\begin{align*}
    \evcond{ \left( \sum_{j=1}^{\phi_{k-1,i}(\boldsymbol{Z}_{k-1})} (X_{k-1,j,i,l} -m_{i,l})\right)^4 }{\mathcal{F}_{k-1}} 
    &= 3 \mathsf{\Sigma}_{i,l,l}^2 
    \left( \mathsf{\Gamma}_{i,i}(\boldsymbol{Z}_{k-1}) 
      + \varepsilon_i (\boldsymbol{Z}_{k-1})^2 \right) \\
    &\quad + (\zeta_{i,l}-3 \mathsf{\Sigma}_{i,l,l}^2) \varepsilon_i (\boldsymbol{Z}_{k-1}),
\end{align*}
 where  $\mathsf{\Sigma}_{i,l,l}$ is the $l$-th diagonal element of the variance matrix $\mathsf{\Sigma}_i$ (see \eqref{notats_moments_2}).

Consequently, we get
\begin{align*}
    \ev{\|\boldsymbol{M}_k\|^4} &\leq 8 p^4 \sum_{i=1}^p \sum_{l=1}^p
     \Big( m_{i,l}^4 \ev{\kappa_i(\boldsymbol{Z}_{k-1})} + 3 \mathsf{\Sigma}_{i,l,l}^2\ev{\mathsf{\Gamma}_{i,i}(\boldsymbol{Z}_{k-1})} \\
    &\phantom{\leq 8 p^4 \sum_{i=1}^p \sum_{l=1}^p \quad} 
     + 3 \mathsf{\Sigma}_{i,l,l}^2 \ev{\varepsilon_i (\boldsymbol{Z}_{k-1})^2} + (\zeta_{i,l}-3 \mathsf{\Sigma}_{i,l,l}^2) \ev{\varepsilon_i (\boldsymbol{Z}_{k-1})} \Big).
\end{align*}
 Thus, to get \eqref{eq-O-mart-diff-4}  it is enough to check that 
 for each $i\in\{1,\ldots,p\}$, $\ev{\varepsilon_i (\boldsymbol{Z}_{k})}$, $\ev{\mathsf{\Gamma}_{i,i}(\boldsymbol{Z}_{k})}$, $\ev{\varepsilon_i (\boldsymbol{Z}_{k})^2}$, $\ev{\kappa_i(\boldsymbol{Z}_{k})}$ are $\OO(k^2)$ 
   as $k\to\infty$.
Using \eqref{epsilon_exp_asymp} and \eqref{Gamma_exp_asymp}, we get 
\begin{equation*}
    \ev{\varepsilon_i (\boldsymbol{Z}_{k})} \leq \ev{\|\boldsymbol{\varepsilon}(\boldsymbol{Z}_{k})\|} = \OO(k), \qquad \ev{\mathsf{\Gamma}_{i,i}(\boldsymbol{Z}_{k})} \leq \ev{\|\mathsf{\Gamma}(\boldsymbol{Z}_{k})\|} = \OO(k)
\end{equation*}
  as $k\to\infty$ for $i\in\{1,\ldots,p\}$.
Moreover, by \eqref{eq-big-O-first-moment} and \eqref{eq-big-O-second-moment}, 
 we obtain
 \begin{align*}
     \ev{\varepsilon_i (\boldsymbol{Z}_{k})^2} &\leq \ev{\|\boldsymbol{\varepsilon}(\boldsymbol{Z}_{k})\|^2} = \ev{\|\mathsf{\Lambda}\boldsymbol{Z}_k + \boldsymbol{h}(\boldsymbol{Z}_k)\|^2} \leq  2( \|\mathsf{\Lambda}\|^2 \ev{\|\boldsymbol{Z}_{k}\|^2} 
      +   C_{\boldsymbol{h}}^2) = \OO(k^2) 
 \end{align*}
as $k\to\infty$ for $i\in\{1,\ldots,p\}$,  where $C_{\boldsymbol{h}}$ is defined in \eqref{eq-C-h}.
Finally, by Hypothesis \ref{hyp-moment-4}, we get
\begin{equation*}
    C_{\kappa} := \sup_{\substack{\boldsymbol{z}\in\Z_+^p \setminus\{\boldsymbol{0}_p\} \\ i = 1,\ldots,p}}  \|\boldsymbol{z}\|^{-2} \|\kappa_i (\boldsymbol{z})\| <\infty ,
\end{equation*}
 and hence, by \eqref{eq-big-O-second-moment}, we have $\ev{\kappa_i(\boldsymbol{Z}_{k})} \leq \kappa_i(\boldsymbol{0}_p) + C_\kappa \ev{\|\boldsymbol{Z}_{k}\|^2} = \OO(k^2)$ as $k\to\infty$ for $i\in\{1,\ldots,p\}$.
\end{proof}

Next, we recall a result on weak convergence of random step processes toward a 
 diffusion process due to Isp\'any and Pap \cite[Corollary 2.2]{ispany-pap-2010}.

\begin{theorem}\label{thm-ispany-pap}
Let $\boldsymbol{b}:\R_+\times\R^p\to\R^p$ and $\mathsf{C}:\R_+\times\R^p\to\R^{p\times r}$ be continuous functions. Assume that {uniqueness in the sense of probability law holds for the SDE} 
\begin{equation}\label{eq-thm-ispany}
    \dif \boldsymbol{\mathcal{U}}_t = \boldsymbol{b}(t,\boldsymbol{\mathcal{U}}_t)\dif t +  \mathsf{C}(t,\boldsymbol{\mathcal{U}}_t)\dif \boldsymbol{\mathcal{W}}_t,\qquad t \in\R_+,
\end{equation}
with initial value $\boldsymbol{\mathcal{U}}_0=\boldsymbol{u}_0$ for all $\boldsymbol{u}_0\in \R^p$, where $(\boldsymbol{\mathcal{W}}_t)_{t\in\R_+}$ is an $r$--dimensional standard Wiener process. Let $\boldsymbol{\eta}$ be a probability measure on $ (\R^p,\mathcal{B}(\R^p))$, and let $(\boldsymbol{\mathcal{U}}_t)_{t\in\R_+}$ be a solution of (\ref{eq-thm-ispany}) with initial distribution $\boldsymbol{\eta}$. 
For each $n \in \N,$ let $(\boldsymbol{U}_{k}^{(n)})_{k\in\Z_+}$ be a sequence of 
 $\R^p$--valued random vectors adapted to a filtration $(\mathcal{F}_{k}^{(n)})_{k\in\Z_+}$
 (i.e, $\boldsymbol{U}_{k}^{(n)}$ is $\mathcal{F}_{k}^{(n)}$--measurable)
  such that $\ev{\|\boldsymbol{U}_k^{(n)}\|^2}<\infty$ for each $n, k \in \N$. 
Let 
 \[
  \boldsymbol{\mathcal{U}}_t^{(n)}  \colonequals  \sum_{k=0}^{\lfloor n t\rfloor}    
                                       \boldsymbol{U}_{k}^{(n)}, \qquad t\in\R_+,\quad n\in\N.
 \] 
Suppose $\boldsymbol{U}_0^{(n)} \stackrel{\mathcal{L}}{\longrightarrow} \boldsymbol{\eta}$ as $n \rightarrow \infty$,  and that for all $T>0,$
\begin{enumerate}[label=(\roman*)]
    \item $\sup_{t\in [0,T]} \left\| \sum_{k=1}^{\lfloor n t\rfloor} \evcond{\boldsymbol{U}_k^{(n)}}{\mathcal{F}_{k-1}^{(n)}} - \int_0^t \boldsymbol{b}(s,\boldsymbol{\mathcal{U}}_s^{(n)})\dif s \right\| \stackrel{\mathrm{P}}{\longrightarrow} 0$ {as $n\to\infty$,}
    \item $\sup_{t\in [0,T]} \left\| \sum_{k=1}^{\lfloor n t\rfloor} \varcond{\boldsymbol{U}_k^{(n)}}{\mathcal{F}_{k-1}^{(n)}} - \int_0^t \mathsf{C}(s,\boldsymbol{\mathcal{U}}_s^{(n)})\mathsf{C}(s,\boldsymbol{\mathcal{U}}_s^{(n)})^\top\dif s \right\| \stackrel{\mathrm{P}}{\longrightarrow} 0$ {as $n\to\infty$,}
    \item $\sum_{k=1}^{\lfloor n T \rfloor} \evcond{\|\boldsymbol{U}_k^{(n)}\|^2 \1_{\{\|\boldsymbol{U}_k^{(n)}\|>\theta\}}}{\mathcal{F}_{k-1}^{(n)}} \stackrel{\mathrm{P}}{\longrightarrow} 0$ {as $n\to\infty$} for all $\theta >0$.
\end{enumerate}
Then $\boldsymbol{\mathcal{U}}^{(n)}\stackrel{\mathcal{L}}{\longrightarrow}\boldsymbol{\mathcal{U}}$ as $n \to \infty$.
\end{theorem}  

For measurable mappings $\boldsymbol{\Phi}$, $\boldsymbol{\Phi}_n: \mathbf{D}(\R_+,\R^p)\to\mathbf{D}(\R_+,\R^p)$, $n\in\N$, let $\mathbf{C}_{\boldsymbol{\Phi},(\boldsymbol{\Phi}_n)_{n\in\N}}$ be the set of all functions $\boldsymbol{f}\in \mathbf{C}(\R_+,\R^p)$ for which $\boldsymbol{\Phi}_n (\boldsymbol{f}_n) \to \boldsymbol{\Phi}(
    \boldsymbol{f})$ whenever $\boldsymbol{f}_n\stackrel{\mathcal{\mathrm{lu}}}{\longrightarrow} 
    \boldsymbol{f}$ with $\boldsymbol{f}_n\in \mathbf{D}(\R_+,\R^p)$, $n\in\N$. The notation $\stackrel{\mathcal{\mathrm{lu}}}{\longrightarrow}$ refers to local uniform convergence, i.e., $\sup_{t\in [0,T]} \|\boldsymbol{f}_n(t)-\boldsymbol{f}(t)\|\to 0$ as $n\to\infty$ for all $T\in\R_{++}$. The following result is a kind of continuous mapping theorem,  which can be considered as a consequence of Theorem 3.27 in Kallenberg \cite{Kallenberg_1997}, and its proof also appears in Isp\'any and Pap \cite[Lemma 3.1]{ispany-pap-2010}.
    
\begin{theorem}\label{cont-map-thm}
    Let $(\boldsymbol{\mathcal{U}}_t)_{t\in\R_+}$ and $(\boldsymbol{\mathcal{U}}_t^{(n)})_{t\in\R_+}$, $n\in\N$, be 
    $\R^p$--valued stochastic processes with càdlàg paths
    such that $\boldsymbol{\mathcal{U}}^{(n)}\stackrel{\mathcal{L}}{\longrightarrow}
    \boldsymbol{\mathcal{U}}$ as $n \to \infty$. Let $\boldsymbol{\Phi}: \mathbf{D}(\R_+,\R^p)\to\mathbf{D}(\R_+,\R^p)$ and $\boldsymbol{\Phi}_n: \mathbf{D}(\R_+,\R^p)\to\mathbf{D}(\R_+,\R^p)$, $n\in\N$, be measurable mappings such that there exists a measurable set $\mathbf{C}\subset\mathbf{C}_{\boldsymbol{\Phi},(\boldsymbol{\Phi}_n)_{n\in\N}}$ with 
    $ \prob{\boldsymbol{\mathcal{U}}\in\mathbf{C}}=1$. 
    Then $\boldsymbol{\Phi}_n (\boldsymbol{\mathcal{U}}^{(n)})\stackrel{\mathcal{L}}{\longrightarrow} \boldsymbol{\Phi}(
    \boldsymbol{\mathcal{U}})$ as $n\to\infty$.
\end{theorem}

\section*{Acknowledgements}
We would like to thank the referee for her/his comments that helped us improve the paper.

\section*{Funding}
M\'aty\'as Barczy was supported by the project TKP2021-NVA-09.
Project no.\ TKP2021-NVA-09 has been implemented with the support
 provided by the Ministry of Innovation and Technology of Hungary from the National Research, Development and Innovation Fund,
 financed under the TKP2021-NVA funding scheme. 
Miguel Gonz\'alez, Pedro Mart\'in-Ch\'avez and  In\'es del Puerto are supported by grant PID2019-108211GB-I00 funded by MCIN/AEI/10.13039/501100011033, by “ERDF A way of making Europe”. Pedro Mart\'in-Ch\'avez is also grateful to the Spanish Ministry of Universities for support from a predoctoral fellowship Grant no.\ FPU20/06588.

\section*{Declarations}

{\bf Conflict of interest.} 
The authors declare that they have no conflict of interest.

\bibliographystyle{acm}
\bibliography{references}
\end{document}